\documentclass[10pt,a4paper]{article}
\usepackage{graphicx}
\usepackage{amsmath,amsfonts,amsthm}
\usepackage{xfrac,esint,mathrsfs,color}
\usepackage[utf8]{inputenc}
\usepackage[colorlinks=true]{hyperref}
\usepackage{a4wide}

\newcommand{\C}{\mathbb{C}}
\newcommand\res{\mathop{\hbox{\vrule height 7pt width .5pt depth 0pt
\vrule height .5pt width 6pt depth 0pt}}\nolimits}

\definecolor{green}{rgb}{0,.5,0}

\newcommand\eps{\varepsilon}
\renewcommand\div{\mathrm{div\,}}

\newcommand\sym{\mathrm{sym}}
\newcommand\loc{\mathrm{loc}}

\newcommand\dist{\mathrm{dist}}

\newcommand\R{\mathbb{R}}
\newcommand\N{\mathbb{N}}

\newcommand\calH{\mathcal{H}}

\newcommand{\calL}{\mathcal{L}}

\newtheorem{theorem}{Theorem}[section]

\newtheorem{proposition}[theorem]{Proposition}
\newtheorem{lemma}[theorem]{Lemma}

\newtheorem{remark}[theorem]{Remark}
\newtheorem{corollary}[theorem]{Corollary}

\numberwithin{equation}{section}
\newcounter{Nummer}

\usepackage{fancyhdr}
\newcommand{\Fh}{F_h}
\newcommand{\FFh}{\mathscr{F}_h}
\newcommand{\Finf}{F_\infty}
\newcommand{\FFinf}{\mathscr{F}_\infty}
\newcommand{\Exc}{\mathscr{E}}

\newcommand{\Fz}{\mathscr{F}_{\mu,0}}
\newcommand{\FF}{\mathscr{F}_{\mu,\cost}}

\newcommand{\cost}{\kappa}

\newcommand{\zetaq}{\zeta^{p}}
\newcommand{\zetaqu}{\zeta^{p-1}}

\begin{document}
\begin{center}
  {\Large A note on the {H}ausdorff dimension of the singular set of solutions to elasticity type systems
}\\[5mm]
{\today}\\[5mm]
Sergio Conti$^{1}$, Matteo Focardi$^{2}$, and Flaviana Iurlano$^{3}$\\[2mm]
{\em $^{1}$
 Institut f\"ur Angewandte Mathematik,
Universit\"at Bonn\\ 53115 Bonn, Germany}\\[1mm]
{\em $^{2}$ DiMaI, Universit\`a di Firenze\\ 50134 Firenze, Italy}\\[1mm]
{\em $^{3}$ Laboratoire Jacques-Louis Lions, Sorbonne Université (Paris 6)\\ 75005 Paris, France}\\[3mm]
    \begin{minipage}[c]{0.8\textwidth}
    \end{minipage}
\end{center}
\begin{abstract}
We prove partial regularity for minimizers to elasticity type energies   in the nonlinear framework {with $p$-growth, $p>1$,}
in dimension $n\geq 3$. It is an open problem in such a setting either to establish full regularity 
or to provide counterexamples. In particular, we give an estimate on the Hausdorff dimension of 
the potential singular set by proving that is strictly less than 
{$n-(p^*\wedge 2)$, and actually $n-2$ in the autonomous case} (full regularity is well-known in dimension $2$). 

The latter result is instrumental to establish existence for the strong formulation of Griffith 
type models in brittle fracture with nonlinear constitutive relations, accounting for damage and 
plasticity in space dimensions $2$ and $3$.
\end{abstract}
\section{Introduction}
In this paper we investigate partial regularity of local minimizers for a class of energies whose 
prototype is 
\begin{equation*}
  \int_{\Omega} \frac 1p\left(\big(\mathbb{C}e(u)\cdot e(u)+\mu\big)^{\sfrac p2}-\mu^{\sfrac p2}\right)dx
  +\cost\int_{\Omega} |u(x)-g(x)|^p\,dx 
\end{equation*}
for $u\in W^{1,p}(\Omega;\R^n)$, $\Omega\subset\R^n$ bounded and open, $p\in(1,\infty)$ (see below 
for the precise assumptions on the relevant quantities). In addition, we establish an estimate on 
the Hausdorff dimension of the related singular set. 

The main motivations for our study arise from Griffith's variational approach to brittle fracture. 
In such a model the equilibrium state of an elastic solid body deformed by external forces is determined by 
the minimization of an energy in which a bulk term and a surface one are in competition 
(see \cite{FrancfortMarigo1998,BourdinFrancfortMarigo2008,dm-toa}). The former represents the 
elastic stored energy in the uncracked part of the  body, instead the latter is related to the energy 
spent to create a crack, and it is typically proportional to the measure of the crack surface itself.
As a model case, for $p\in(1,\infty)$ and $\kappa,\, \mu\geq 0$ one looks for minimizers $(\Gamma,u)$ of
\begin{equation}\label{eqgriffintro}
 E[\Gamma,u]:= \int_{\Omega\setminus\Gamma} 
 \frac 1p\left(\big(\mathbb{C}e(u)\cdot e(u)+\mu\big)^{\sfrac p2}-\mu^{\sfrac p2}\right)\,dx
 +\kappa\int_{\Omega\setminus\Gamma} |u(x)-g(x)|^p\,dx 
 + 2\beta\calH^{n-1}({\Gamma\cap\Omega})
\end{equation}
over all closed 
sets $\Gamma\subset{\overline\Omega}$ and all deformations $u\in C^1(\Omega\setminus\Gamma;{\R^n})$ subject 
to suitable boundary and irreversibility conditions. Here $\Omega\subset\R^n$ is the reference configuration, 
the function $\kappa|\xi-g(x)|^p\in C^0(\Omega\times \R^n)$ represents external volume forces, 
$e(u)=(\nabla u+\nabla u^T)/2$ is the elastic strain, $\C\in \R^{(n\times n)\times (n\times n)}$ is the matrix 
of elastic coefficients, $\beta>0$ the surface energy.
More precisely, the energy in \eqref{eqgriffintro} for $p=2$ corresponds to classical Griffith's fracture model, 
while densities having $p$-growth with $p\neq 2$ may be instrumental for a variational formulation of fracture 
with nonlinear constitutive relations, accounting for damage and plasticity (see for example 
\cite[Sections~10-11]{hutchinson1989course} and references therein).

In their seminal work \cite{DegiorgiCarrieroLeaci1989}, De~Giorgi, Carriero and Leaci have introduced a viable 
strategy to prove existence of minimizers for the corresponding scalar energy, 
\begin{equation}\label{eqgMSintro}
 E_\mathrm{MS}[\Gamma,u]:=\int_{\Omega\setminus\Gamma} \frac1p\Big(\big(|Du|^2+\mu\big)^{\sfrac p2}-\mu^{\sfrac p2}\Big)\,dx
 +\cost\int_{\Omega\setminus\Gamma}|u(x)-g(x)|^p\,dx 
 + 2\beta\calH^{n-1}(\Gamma{\cap\Omega}) \,,
\end{equation}
better known for $p=2$ as the Mumford and Shah functional in image segmentation (cf. the book \cite{AmbrosioFuscoPallara} 
for more details on the Mumford and Shah model and related ones). 
From a mechanical perspective the scalar setting matches the case of anti-plane deformations 
$u:\Omega\setminus\Gamma\to\R$.
Following a customary idea in the Calculus of Variations, the functional $E_\mathrm{MS}$ is first relaxed in a wider space, 
so that existence of minimizers can be obtained. 
The appropriate functional setting in the scalar framework is provided by a suitable subspace of $BV$ functions.
Surface discontinuities in the distributional derivative of the deformation $u$ are then allowed, they are concentrated 
on a $(n-1)$-dimensional (rectifiable) set $S_u$. Then, existence for the strong formulation 
is recovered by establishing a mild regularity result for minimizers $u$ of the weak counterpart: the essential 
closedness of the jump set $S_u$, namely $\calH^{n-1}(\Omega\cap\overline{S_u}\setminus S_u)=0$, complemented with
smoothness of $u$ on $\Omega\setminus\overline{S_u}$. 
Given this, $(u,\overline{S_u})$ turns out to be a minimizing couple for \eqref{eqgriffintro}.

In the approach developed by De~Giorgi, Carriero and Leaci in \cite{DegiorgiCarrieroLeaci1989},
regularity issues for local minimizers of the restriction of $E_\mathrm{MS}$ in \eqref{eqgMSintro} to Sobolev functions,
such as decay properties of the $L^p$ norm of the corresponding gradient, play a key role for establishing both the essential 
closedness of $S_u$ for a minimizer $u$ of \eqref{eqgriffintro} and the smoothness of $u$ itself on $\Omega\setminus\overline{S_u}$. 
Nowadays, these are standard subjects in elliptic regularity theory (cf. for instance the books 
\cite{Giaquinta83,Giusti,GiaquintaMartinazzi2012}).

Following such a streamline of ideas, in a recent paper \cite{ContiFocardiIurlano17} we have proved existence in the 
two dimensional framework for the functional in \eqref{eqgriffintro} for suitably regular $g$ (see also 
\cite{ContiFocardiIurlano2016-CRAS} that settles the case $p=2$). 
In passing we mention that the domain of the relaxed 
functional is provided for the current problem by a suitable subset, $SBD$ (actually $GSBD$), of the space 
$BD$ ($GBD$) of functions with (generalized) bounded deformation (we omit the precise definitions since they are 
inessential for the purposes of the current paper and rather refer to \cite{ContiFocardiIurlano17,ChambolleContiIurlano2017}).
More in details, our modification of the De~Giorgi, Carriero, and Leaci approach rests on three main ingredients:
the compactness and the asymptotic analysis of sequences in $SBD$ having vanishing jump energy; 
the approximation in energy of general $(G)SBD$ maps with more regular ones; and the decay and 
smoothness properties of local minimizers of the functional in \eqref{eqgriffintro} when restricted to Sobolev functions.
The compactness issue is dealt with in \cite{ContiFocardiIurlano17} in the two dimensional case {and in \cite{ChambolleContiIurlano2017} in higher dimensions, 
{in both papers for all $p>1$.}} 
The asymptotic analysis {is performed in \cite{ContiFocardiIurlano17} and} holds without dimensional limitations. The approximation property holds in any dimension 
as well, it is established in the companion paper \cite{ContiFocardiIurlanoGBD}.
Instead, the regularity properties of local minimizers of energies like
\begin{equation}\label{eqgriffintro2}
  \int_{\Omega} \frac 1p\left(\big(\mathbb{C}e(u)\cdot e(u)+\mu\big)^{\sfrac p2}-\mu^{\sfrac p2}\right)dx+
  \cost\int_{\Omega} |u(x)-g(x)|^p\,dx 
\end{equation}
on $W^{1,p}(\Omega;\R^n)$ are the object of investigation in the current paper. 
More generally, we study smoothness of local minimizers of elastic-type energies 
\begin{equation}\label{introFF2}
\FF(u)=\int_\Omega f_\mu(e(u))\,dx+\cost\int_{\Omega}|u-g|^pdx,
\end{equation}
on $W^{1,p}(\Omega;\R^n)$, $n\geq 2$, for $f_\mu$ satisfying suitable convexity, smoothness and growth conditions 
(see Section~\ref{s:hpfmu} for the details). {We carry over the analysis in any dimension
since the results of the current paper, {together with
the compactness property established in \cite{ChambolleContiIurlano2017} mentioned above, imply} 
a corresponding existence result for the minimizers of \eqref{introFF2} in the physical dimension $n=3$, for any $p>1$, and for $\mu>0$}. In this respect it is essential for us to derive an estimate on the Hausdorff dimension 
of the (potential) singular set, and prove that it is strictly less than $n-1$. {We recall that if $p=2$ the regularity properties of the aforementioned local minimizers are well-known, so that the corresponding existence result for the minimizers of \eqref{introFF2} follows straightforwardly {from \cite{ContiFocardiIurlano17} in dimension $n=2$ and from 
\cite{ChambolleContiIurlano2017} in any dimension.}}

The starting point of our study is the equilibrium system satisfied by minimizers of \eqref{introFF2} that reads as 
\begin{equation}\label{introFF}
-\div(\nabla f_\mu(e(u)))+\cost p|u-g|^{p-2}(u-g)=0,
\end{equation}
in the distributional sense on $\Omega$.
Variants of \eqref{introFF} have been largely studied in fluid dynamics (we refer to the monograph \cite{FuchsSeregin}
for all the details). In this context the system \eqref{introFF} with $\kappa=0$ is coupled with a divergence-free constraint 
and represents a stationary generalized Stokes system. It describes a steady flow of a fluid when the velocity $u$ is 
small and the convection can be neglected. 
To our knowledge all contributions present in literature and concerning \eqref{introFF} are in this framework, 
apart from the case $p=2$ which is classical, see for example \cite{Giaquinta83,GiaquintaMartinazzi2012,Morrey-Multiple}.

Under the divergence-free constraint and $\kappa=0$, regularity of solutions has been established first for 
$p\geq2$ and every $\mu\geq0$, see \cite{FuchsSeregin,BreitDieningFuchs}, then for $1<p<2$ and every $\mu\geq0$ 
in the planar setting, see \cite{BildhauerFuchs,BildhauerFuchsZhong05,DieningKaplickySchwarzacher14} 
(the first two papers actually deal with the more general case of integrands satisfying $p-q$ growth conditions, 
the latter with the case of growth in terms of $N$-functions). $L^q$ estimates for \eqref{introFF} in the $3$-dimensional setting 
have been obtained in \cite{DieningKaplicky} for every $\mu\geq0$. 
Regularity up to the boundary for the second derivative of solutions is proved for $p>2$ and $\mu>0$ in 
\cite{BeiraoCrispo}. 

We stress explicitly that we have not been able to find in literature the mentioned estimate on the Hausdorff dimension of 
the singular set. Moreover, we also point out that the special structure of our lower order term does not fit the usual 
assumptions in literature (see for instance \cite[Theorem~1.2]{KristensenMingione} in the case of the $p$-laplacian).
Despite this, it is possible to extend the results of this paper to a wider class of energies, as those satisfying
for instance the conditions \cite[(1.1)-(1.2)]{KristensenMingione} building upon the ideas and techniques developed in 
\cite{KristensenMingione,KristensenMingione05,KristensenMingione07} (see also \cite{Mingione08} for a complete report).

{In conclusion, 
we provide here} detailed proofs for the decay estimates (with $\kappa,\mu\geq0$, see Proposition~\ref{p:decayVmu} and 
Corollary~\ref{c:decayVmu bis}) and for full or partial regularity of solutions (the former for $n=2$, the latter for 
$n\geq 3$ and $\mu>0$, see Section~\ref{ss:enhanced}). 
We stress that if $n\geq 3$ it is a major open problem to prove or disprove full regularity
even in the non degenerate, i.e.~$\mu>0$, symmetrized $p$-laplacian case for $p\neq 2$. In these regards, if $n\geq 3$ 
we provide an estimate of the Hausdorff dimension of the potential singular set that seems to have been overlooked 
in the literature. In particular, the potential singular set has dimension strictly less than {$n-1$}.

Finally, we resume briefly the structure of the paper. In Section~\ref{ss:prelim} we introduce the notation and the 
(standard) assumptions on the class of integrands $f_\mu$. We also recall the basic properties of the nonlinear potential 
$V_\mu$, an auxiliary function commonly employed in literature for regularity results in the non quadratic case. In addition, 
we review the framework of shifted $N$-functions introduced in \cite{DieningEttwein}, that provides the right technical tool 
for deriving Caccioppoli's type inequalities for energies depending on the symmetrized gradient. Caccioppoli's inequalities 
are the content of Section~\ref{ss:caccioppoli}, as a consequence of those in Section~\ref{s:decay} we derive the mentioned 
decay properties of the $L^2$ norm of $V_\mu(e(u))$. 
We remark that the Morrey type estimates in Section~\ref{s:decay} and the improvement in Corollary~\ref{c:decayVmu bis} 
are helpful for the purposes of \cite{ContiFocardiIurlano17,ChambolleContiIurlano2017} only for $n\in\{2,3\}$ in view of the decay rate established 
there. Partial regularity with an estimate on the Hausdorff dimension of the singular set are the objects of 
Section~\ref{ss:enhanced}.
More precisely, the higher integrability of $V_\mu(e(u))$ is addressed in Section~\ref{ss:HI}, from this the full
regularity of local minimizers in the two dimensional case easily follows by Sobolev embedding (cf. Section~\ref{ss:2d}).
Section~\ref{ss:partial autonomous} deals with 
the autonomous case $\cost=0$, for which we use a linearization argument in the spirit of vectorial regularity results (the 
needed technicalities for these purposes are collected in Appendix~\ref{a:technical}). 
The non-autonomous case is then a consequence of a perturbative approach as in the classical paper \cite{GiaquintaGiusti83} 
(see Section~\ref{ss:partial non-autonomous}).

\section{Preliminaries}\label{ss:prelim}

With $\Omega$ we denote an open and bounded Lipschitz set in $\R^n$, $n\geq 2$. The Euclidean scalar product is
indicated by $\langle \cdot,\cdot\rangle$.
We use standard notation for Lebesgue and Sobolev spaces. By $s^\ast$ we denote the Sobolev exponent 
of $s$ if $s\in[1,n)$, otherwise it can be any positive number strictly bigger than $n$.
If $w\in L^1(B;\R^n)$, $B\subseteq\Omega$, we set
\begin{equation}\label{e:media}
 (w)_B:=\fint_B w(y)dy.
\end{equation}
In what follows we shall use the standard notation for difference quotients
\begin{equation}\label{e:rapporti incrementali}
 \triangle_{s,h}v(x):=\frac{1}{h}(v(x+h\epsilon_s)-v(x)),\quad
  \tau_{s,h}v(x):=h\,\triangle_{s,h}v(x),
\end{equation}
if $x\in\Omega_{s,h}:=\{x\in\Omega:\,x+h\epsilon_s\in\Omega\}$ and $0$ otherwise in $\Omega$,
where $v:\Omega\to\R^n$ is any measurable map and $\epsilon_s$ is any coordinate unit vector 
of $\R^n$.

\subsection{Assumptions on the integrand}\label{s:hpfmu}

For given $\mu\geq 0$ and $p>1$ we consider a function $f_\mu:\R^{n\times n}_\sym\to\R$ satisfying
\begin{itemize}

\item[(Reg)] $f_\mu\in C^2(\R^{n\times n}_\sym)$ if $p\in(1,2)$ and $\mu>0$ or $p\in[2,\infty)$ 
and $\mu\geq 0$, while $f_0\in C^1(\R^{n\times n}_\sym)\cap C^2(\R^{n\times n}_\sym\setminus\{0\})$ 
if $p\in(1,2)$;

\item[(Conv)] for all $p\in(1,\infty)$ and for all symmetric matrices 
$\xi$ and $\eta\in\R^{n\times n}_\sym$ we have
\begin{equation}\label{e:bdhessf}
\frac 1c\big(\mu+|\xi|^2\big)^{\sfrac p2-1}|\eta|^2
\leq\langle\nabla^2 f_\mu(\xi)\eta,\eta\rangle\leq 
c\big(\mu+|\xi|^2\big)^{\sfrac p2-1}|\eta|^2,
\end{equation}
with $c=c(p)>0$, unless $\mu=|\xi|=0$ and $p\in(1,2)$. {We further assume $f_\mu(0)=0$ and $Df_\mu(0)=0$.}
 \end{itemize}

\begin{remark}
The prototype functions we have in mind for applications to the mentioned Griffith fracture 
model are defined by 
\begin{equation}\label{e:fmu}
f_\mu(\xi)=\frac 1p\left(\big(\mathbb{C}\xi\cdot\xi+\mu\big)^{\sfrac p2}-\mu^{\sfrac p2}\right),
\end{equation}
for all $\mu\geq 0$ and $p\in(1,\infty)$. Clearly (Reg) is satisfied, moreover we have
\[
 \nabla f_\mu(\xi)=\big(\mathbb{C}\xi\cdot\xi+\mu\big)^{\sfrac p2-1}\mathbb{C}\xi
\]
(with $\nabla f_0(0)=0$), and in addition 
\begin{equation}\label{e:hessf}
 \nabla^2 f_\mu(\xi)=\big(\mathbb{C}\xi\cdot\xi+\mu\big)^{\sfrac p2-2}
 \Big((p-2)\mathbb{C}\xi\otimes \mathbb{C}\xi+(\mathbb{C}\xi\cdot\xi+\mu)\mathbb{C}\Big)
\end{equation}
(with $\nabla^2 f_0(0)=0$ if $p\in(2,\infty)$, $\nabla^2f_0(0)=\C$ if $p=2$, $\nabla^2f_0(0)$ 
undefined if $p\in (1,2)$). The lower inequality in \eqref{e:bdhessf} is clearly satisfied for $p\in[2,\infty)$; 
to check it if $p\in(1,2)$ consider the quantity
\begin{equation*}
\alpha:= (p-2)(\C\xi\cdot \eta)^2 + (\C\xi\cdot\xi+\mu)(\C\eta\cdot\eta).
\end{equation*}
Since $\C$ defines a scalar product on the space of symmetric matrices, Cauchy-Schwarz inequality
\[
\C\xi\cdot \eta \le (\C\xi\cdot\xi)^{1/2}
(\C\eta\cdot\eta)^{1/2}
\]
yields for $p\in(1,2)$ 
\begin{equation*}
\alpha\ge [(p-2)(\C\xi\cdot \xi) + (\C\xi\cdot\xi+\mu)](\C\eta\cdot\eta)
\ge (p-1) (\C\xi\cdot\xi+\mu)(\C\eta\cdot\eta),
\end{equation*}
the other inequality in \eqref{e:bdhessf} can be proved analogously.
\end{remark}

Note that from (Conv) we deduce the $p$-growth conditions
\begin{equation}\label{e:growth condition}
{c^{-1}(|\xi|^{2}+\mu)^{\sfrac p2-1}|\xi|^2 \leq f_\mu(\xi)\leq c(|\xi|^{2}+\mu)^{\sfrac p2-1}|\xi|^2 }
\end{equation}
and
\begin{equation}\label{e:growth condition nabla}
|\nabla f_\mu(\xi)|\leq c{  (|\xi|^{2}+\mu)^{\sfrac p2-1} |\xi|}
\end{equation}
for all $\xi\in\R^{n\times n}_\sym$ with $c={c(p)}>0$ {(see also Lemma \ref{l:tecnico} below).}
Therefore, for all $\cost,\,\mu\geq0$, the functional $\FF:W^{1,p}(\Omega;\R^n)\to\R$ given by 
\begin{equation}\label{e:FF}
\FF(v)=\int_\Omega f_\mu(e(v))dx+\cost\int_{\Omega}|v-g|^pdx
\end{equation}
is well-defined.

\subsection{The nonlinear potential $V_\mu$}

In what follows it will also be convenient to introduce the auxiliary function
$V_\mu:\R^{n\times n}\to\R^{n\times n}$,
\[
 V_\mu(\xi):=(\mu+|\xi|^2)^{\sfrac{(p-2)}{4}}\xi,
\]
with $V_0(0)=0$ {(we do not highlight the $p$ dependence for the sake of simplicity).}
\begin{remark}\label{r:Vmu}
 Note that $|V_0(\xi)|^2=|\xi|^p$ for every $\xi\in\R^{n\times n}$, and for all $\mu>0$
\begin{equation}\label{e:stima banana}
  |V_\mu(\xi)|^2\leq(\mu+|\xi|^2)^{\sfrac p2}= |V_\mu(\xi)|^2+\mu(\mu+|\xi|^2)^{\sfrac p2-1}
 { \le c|V_\mu(\xi)|^2+c\,\mu^{\sfrac p2}}
\end{equation}
with $c=c(p)>0$.
\end{remark}

The following two basic lemmas will be needed in this section 
(see \cite[Lemma~2.1 and Lemma~2.2]{AcerbiFusco89} {and \cite[Lemma 8.3]{Giusti}} for more details). 
\begin{lemma}\label{l:tecnico}
 For every $\gamma>-\sfrac12$, {$r\geq 0$}, and $\mu\geq 0$ we have
 \begin{equation}\label{e:tecnica}
  c_1\leq\frac{\int_0^1\big(\mu+|\eta+t(\xi-\eta)|^2\big)^\gamma{(1-t)^r} dt}{(\mu+|\xi|^2+|\eta|^2)^\gamma}
  \leq c_2,
 \end{equation}
for all $\xi$, $\eta\in\R^k$ such that $\mu+|\xi|^2+|\eta|^2\ne 0$, with $c_i=c_i(\gamma,r)>0$.
\end{lemma}
\begin{proof}
If $\gamma\geq 0$ the upper bound follows easily by $|\eta+t(\xi-\eta)|^2\le |\eta|^2+|\xi|^2$ and the
monotonicity of $(0,\infty)\ni s\mapsto(\mu+ s)^\gamma$ with $c_2=1$.
To prove the lower bound we observe that  if $|\xi|\leq|\eta|$ then 
\[
|\eta+t(\xi-\eta)|\geq |\eta|-t|\xi|-t|\eta| \ge \frac13|\eta|\qquad \forall t\in[0,\sfrac 13],
\]
which implies  the other inequality with $c_1=c_1(\gamma)$. 

The lower bound for $\gamma<0$ is analogous to the previous upper bound. 
The remaining upper bound requires an explicit computation and the integrability
assumption $\gamma>-\sfrac12$, see \cite[Lemma~2.1]{AcerbiFusco89}, which results in
$c_2=\sfrac 8{(2\gamma+1)}$. 
\end{proof}
\begin{lemma}\label{l:tecnico2}
 For every $\gamma>-\sfrac12$ and $\mu\geq 0$ we have
 \begin{equation}\label{e:tecnica2}
  c_3|\xi-\eta|\leq\frac{|(\mu+|\xi|^2)^\gamma\xi-(\mu+|\eta|^2)^\gamma\eta|}
  {(\mu+|\xi|^2+|\eta|^2)^\gamma}\leq c_4\,|\xi-\eta|,
 \end{equation}
 for all $\xi$, $\eta\in\R^n$ such that $\mu+|\xi|^2+|\eta|^2\ne 0$, with $c_i=c_i(\gamma)>0$.
 \end{lemma}
\begin{proof}
 Assume $\mu>0$ and consider the smooth convex function 
 $h(\xi):=\frac{1}{2(\gamma+1)}(\mu+|\xi|^2)^{\gamma+1}$. For all $\xi\in\R^n$ we have
 \[
 \nabla h(\xi)=(\mu+|\xi|^2)^\gamma\xi,\qquad
 \nabla^2 h(\xi)=(\mu+|\xi|^2)^{\gamma}\Big(Id+2\gamma\frac{\xi\otimes\xi}{\mu+|\xi|^2}\Big).
 \]
 Noting that for all $\xi$, $\eta\in\R^n$ it holds
 \[
 \big(1\wedge (1+2\gamma)\big)|\eta|^2
 \leq\frac{\langle\nabla^2 h(\xi)\eta,\eta\rangle}{(\mu+|\xi|^2)^{\gamma}}
 \leq\big(1\vee (1+2\gamma)\big)|\eta|^2,
 \]
 the conclusion follows easily from 
 $\nabla h(\xi)-\nabla h(\eta)=\int_0^1\nabla^2h(\eta + t(\xi-\eta)) (\xi-\eta) dt$
 and Lemma~\ref{l:tecnico} with 
 $c_3=\big(1\wedge (1+2\gamma)\big)c_1$ and $c_4=\big(1\vee (1+2\gamma)\big)c_2$ 
 being $c_1$ and $c_2$ the constants there.
 
 If $\mu=0$ we can simply pass to the limit in formula \eqref{e:tecnica2} as $\mu\downarrow 0$,
 since $c_3$ and $c_4$ depend only on $\gamma$.
 \end{proof}

We collect next 
several properties of $V_\mu$ instrumental for the developments in what follows.
\begin{lemma}\label{l:Vmu}
For all $\xi,\,\eta\in\R^{n\times n}$ and for all $\mu\geq 0$ we have
\begin{itemize}
  \item[(i)] if $p\geq 2$: $c|V_\mu(\xi-\eta)|\leq |V_\mu(\xi)-V_\mu(\eta)|$ 
  for some $c=c(p)>0$, and for all $L>0$ there exists $c=c(\mu,L)>0$ such that 
  $|V_\mu(\xi)-V_\mu(\eta)|\leq c|V_\mu(\xi-\eta)|$ if $|\eta|\leq L$;
  \item[(ii)] if $p\in (1,2)$:
  $|V_\mu(\xi-\eta)|\geq c|V_\mu(\xi)-V_\mu(\eta)|$ for some $c=c(p)>0$, and 
  for all $L>0$ there exists $c=c(\mu,L)>0$ such that 
  $|V_\mu(\xi-\eta)|\leq c|V_\mu(\xi)-V_\mu(\eta)|$ if $|\eta|\leq L$;
    \item[(iii)] $|V_\mu(\xi+\eta)|\leq c(p)(|V_\mu(\xi)|+|V_\mu(\eta)|)$ for all $p\in(1,\infty)$;
  \item[(iv)] $(2(\mu\vee|\xi|^2))^{\sfrac{p}2-1}|\xi|^2\leq|V_\mu(\xi)|^2\leq |\xi|^p$
  if $p\in(1,2)$, $|\xi|^p\leq|V_\mu(\xi)|^2\leq 
  {2^{\sfrac p2-1}(\mu^{p/2-1}|\xi|^2+|\xi|^p)} $ if $p\geq 2$;
  \item[(v)] $\xi\mapsto|V_\mu(\xi)|^2$ is convex for all $p\in[2,\infty)$; for all $p\in(1,2)$ we have 
  $(\mu^{\sfrac{(2-p)}{2}}+|\xi|^{2-p})^{-1}|\xi|^2\leq|V_\mu(\xi)|^2\leq c(p)(\mu^{\sfrac{(2-p)}{2}}+|\xi|^{2-p})^{-1}|\xi|^2$ 
  and $\xi\mapsto(\mu^{\sfrac{(2-p)}{2}}+|\xi|^{2-p})^{-1}|\xi|^2$ is convex.
 \end{itemize}
\end{lemma}
\begin{proof}
If $p\in[2,\infty)$, property (i) follows from Lemma~\ref{l:tecnico2}, 
while properties (iii) and (iv) are simple consequences of the very definition of $V_\mu$.  

Instead, for the case $p\in(1,2)$ we refer to \cite[Lemma~2.1]{CarozzaFuscoMingione98}.
More precisely, item (ii) above is contained in items (v) and (vi) there, 
(iii) above in (iii) there, and (iv) above in (i) there.

Finally, (v) follows by a simple computation. Indeed, first note that 
$|V_\mu(\xi)|^2=\phi_\mu(|\xi|^2)$, where $\phi_\mu(t):=(\mu+t)^{\sfrac p2-1}t$ for $t\geq 0$.
Then $\xi\mapsto|V_\mu(\xi)|^2$ is convex if and only if for all $\eta,\,\xi\in\R^n$
\[
 2\phi_\mu'(|\xi|^2)|\eta|^2+4\phi_\mu''(|\xi|^2)\langle \eta,\xi\rangle^2\geq 0.
\]
Using the explicit formulas for the first and second derivatives of $\phi_\mu$
this amounts to prove for all $\eta,\,\xi\in\R^n$
\[
 (\mu+|\xi|^2)\big(2\mu+ p|\xi|^2\big)|\eta|^2
 +\big(4(p-2)\mu+p
 (p-2)|\xi|^2\big)\langle \eta,\xi\rangle^2\geq 0.
\]
In particular the conclusion is straightforward for $p\geq 2$. 
Instead, for $p\in(1,2)$ we follow \cite[Section~3]{DuzGrotKr}.
We first observe that $\|x\|_{\frac{2}{2-p}}\leq\|x\|_1\leq2^{\sfrac{p}{2}}\|x\|_{\frac{2}{2-p}}$ 
applied to the vector $(\mu^{\sfrac{(2-p)}{2}},\xi^{2-p})$ gives 
$(\mu^{\sfrac{(2-p)}{2}}+|\xi|^{2-p})^{-1}|\xi|^2\leq|V_\mu(\xi)|^2\leq c(p)(\mu^{\sfrac{(2-p)}{2}}+|\xi|^{2-p})^{-1}|\xi|^2$. 
A direct computation finally shows that $t\mapsto(\mu^{\sfrac{(2-p)}{2}}+t^{2-p})^{-1}t^2$ is convex and monotone increasing 
on $[0,+\infty)$, and that it vanishes for $t=0$. We conclude that $\xi\mapsto(\mu^{\sfrac{(2-p)}{2}}+|\xi|^{2-p})^{-1}|\xi|^2$ 
is convex.
\end{proof}
{Finally, we state a useful property established in \cite[Lemma~2.8]{DieningKaplicky}.
\begin{lemma}\label{l:DieningKaplicky}
For all $\mu\geq 0$ there exists a constant $c=c(n,p,\mu)>0$ such that
for every $u\in W^{1,p}(\Omega;\R^n)$ if $B_r(x_0)\subset\Omega$
\begin{align*}
 \int_{B_r(x_0)}\left|V_\mu(e(u))-\big(V_\mu(e(u))\big)_{B_r(x_0)}\right|^2dx &\leq
 \int_{B_r(x_0)}\left|V_\mu(e(u))-V_\mu\big((e(u))_{B_r(x_0)}\big)\right|^2dx\\
 &\leq c
 \int_{B_r(x_0)}\left|V_\mu(e(u))-\big(V_\mu(e(u))\big)_{B_r(x_0)}\right|^2dx.
\end{align*}
\end{lemma}
}

\subsection{Shifted $N$-functions}

We fix $p\in(1,\infty)$ and $\mu\geq 0$, and, following 
\cite[Definition~22]{DieningEttwein} for every $a\geq 0$ we consider
the function $\phi_a:[0,\infty)\to\R$,
\begin{equation}\label{a:a}
 \phi_a(t):=\int_0^t\big(\mu+(a+s)^2\big)^{\sfrac p2-1}s\,ds.
\end{equation}
A simple computation shows that $\phi_a''>0$ and, further,
\begin{equation}\label{eqphiprimophisecondo}
 \phi'_a(t)\le c \phi_a''(t) t \hskip1cm \text{ for all $t\geq 0$}
\end{equation}
($\phi_a$ turns out to be a N-function
in the language of \cite[Appendix]{DieningEttwein}). From the definition 
one easily checks that for all $a,t\geq 0$  we have
\begin{equation}\label{a:n0}
 \phi_a(t)\leq \frac1p\big(\big(\mu+(a+t)^2\big)^{\sfrac p2}-(\mu+a^2)^{\sfrac p2}\big).
\end{equation}
More precisely, for every $t\geq 0$ we have
\begin{equation}\label{a:n}
 (\mu+(a+t)^2)^{\sfrac p2-1}\frac{t^2}2\leq\phi_a(t)\leq 
 \frac{t^p}p
 \qquad\text{if $p\in(1,2)$},
\end{equation}
\begin{equation}\label{a:n2}
 \frac{t^p}p
 \leq \phi_a(t)\leq (\mu+(a+t)^2)^{\sfrac p2-1}\frac{t^2}2\qquad\text{if $p\in[2,\infty)$}.
\end{equation}
In addition, if $p\in(1,2)$, for every $t\geq 0$ we have
\begin{equation}\label{a:n3}
 \phi_a(t)\leq (\mu+a^2)^{\sfrac p2-1}\frac{t^2}2.
\end{equation}

A simple change of variables shows that the 
family $\{\phi_a\}_{a\geq 0}$ satisfies the $\triangle_2$ and $\nabla_2$ conditions 
uniformly in $a$, that is for all $a\geq 0$
\begin{equation}\label{a:b}
 \lambda^{p\wedge 2}\phi_a(t)\leq\phi_a(\lambda\,t)\leq\lambda^{p\vee 2}\phi_a(t),
\end{equation}
for all $\lambda\geq 1$ and $t\geq 0$.
We define the polar of $\phi_a$ in the sense of convex analysis by
\begin{equation}\label{e:phistar}
\phi_a^* (s) := \sup_{t\ge0} \{ st-\phi_a(t)\}\,.
\end{equation}
By convexity and growth of $\phi_a$ one sees that the supremum is attained at a $t$ such that 
$s=\phi_a'(t)$.
For all $a\geq 0$ we have
\begin{equation}\label{a:c}
 \lambda^{\frac p{p-1}\wedge 2}\phi_a^*(s)\leq\phi_a^*(\lambda\,s)\leq\lambda^{\frac p{p-1}\vee 2}\phi_a^*(s),
\end{equation}
for every $\lambda\geq 1$ and for every $t\geq 0$.
In view of \eqref{a:b} and \eqref{a:c} above, Young's inequality holds uniformly in $a\geq 0$: 
for all $\delta\in(0,1]$ there exists $C_{\delta,p}>0$ such that 
\begin{equation}\label{a:e2}
 s\,t\leq
 \delta\,\phi_a^*(s)+C_{\delta,p}\,\phi_a(t)\,
\hskip3mm\text{ and }\hskip3mm
  s\,t\leq
 \delta\,\phi_a(t)+C_{\delta,p}\,\phi_a^*(s)\,
 \end{equation}
for every $s$ and $t\geq 0$ and for all $a\ge 0$ (see also \cite[Lemma~32]{DieningEttwein}).  

Convexity of $\phi_a$  implies 
\begin{equation}\label{a:e}
 \frac t2\,\phi_a^\prime\big(\frac{t}{2}\big)\leq\phi_a(t)\leq t\,\phi_a^\prime(t)
 \qquad{\forall t\geq 0}.
\end{equation}
From  $\phi_a^*(\phi_a^\prime(t))=\phi_a^\prime(t)t-\phi_a(t)$, 
\eqref{a:e}, \eqref{a:e2}
we infer that there is a constant $c>0$ such that for all $a\geq 0$
\begin{equation}\label{a:f}
\frac 1c\,\phi_a(t) \leq\phi_a^*(\phi_a^\prime(t))\leq c\,\phi_a(t),
\end{equation}
for every $t\geq 0$ (see also \cite[formula~(2.3)]{DieningEttwein}).

Finally, note that by the first inequality in Lemma~\ref{l:tecnico2} 
with exponent $\gamma=(p-2)/4>-1/4$ we have
\begin{equation*}
 c|\xi-\eta|^2 (\mu+|\xi|^2+|\eta|^2)^{p/2-1} \le |V_\mu(\xi)-V_\mu(\eta)|^2
\end{equation*}
for every $\xi$, $\eta\in\R^{n\times n}$.
Furthermore, by the second inequality in \eqref{a:e},
\begin{equation*}
 \phi_{|\xi|}(|\xi-\eta|) \le c |\xi-\eta|^2 (\mu+|\xi|^2+|\xi-\eta|^2)^{p/2-1},
\end{equation*}
and therefore
\begin{equation}\label{eqphixivmu}
 \phi_{|\xi|} (|\xi-\eta|) \le c |V_\mu(\xi)-V_\mu(\eta)|^2 \,.
\end{equation}

\section{Basic regularity results}

In this section we prove some regularity results on local minimizers of generalized linear elasticity systems.
The ensuing Propositions~\ref{p:regVmusuper} and \ref{p:regVmusub} contain the main Caccioppoli's type estimates 
in the super-quadratic and sub-quadratic case, respectively. 
In turn, those results immediately entail a higher integrability result in any dimension that 
will be instrumental for establishing partial regularity together with an estimate of the Hausdorff dimension 
of the singular set (see Propositions~\ref{p:HI} and Theorem~\ref{t:partial autonomous}), as well as for proving 
$C^{1,\alpha}$ regularity for minimizers in the two dimensional case. Moreover, in the two and three dimensional setting useful 
decay properties that were needed in the proof of the density lower bound in \cite{ContiFocardiIurlano17} and 
\cite{ChambolleContiIurlano2017} can be deduced from Propositions~\ref{p:regVmusuper}, 
\ref{p:regVmusub}, and \ref{p:HI} (cf. Proposition~\ref{p:decayVmu} and Corollary \ref{c:decayVmu bis}). 

We point out that if $p\in[2,\infty)$ a more direct and standard proof can be provided that does not need the 
shifted N-functions $\phi_a$ in \eqref{a:a}. Instead, those tools seem to be instrumental for the sub-quadratic 
case. 
Therefore, for simplicity, we have decided to provide a common framework for both. 

In what follows we will make extensive use of the difference quotients introduced in 
\eqref{e:rapporti incrementali} and of the mean values in \eqref{e:media}.

\subsection{Caccioppoli's inequalities}\label{ss:caccioppoli}

We start off dealing with the super-linear case. For future applications to higher integrability 
(cf. Proposition~\ref{p:HI}) it is convenient to set, for $p>2$,
\begin{equation}\label{e:ptilde}
\tilde{p}(\lambda):=\frac{\lambda p(p-2)}{\lambda(p-1)-1} 
\end{equation}
for every $\lambda\in(\frac{1}{p-1},1]$. For {$p> 2$}, $\tilde{p}(\cdot)$ is a decreasing 
function on $(\frac{1}{p-1},1]$ with $\tilde{p}(1)=p$ {and $\tilde{p}\to\infty$ as $\lambda\to \frac{1}{p-1}$}. 
In addition, define $\lambda_0\in(\frac{1}{p-1},1]$
to be such that $\tilde{p}(\lambda_0)= p^\ast$, being $p^\ast:=\frac{np}{n-p}$, if $p\in(1,n)$, and $\lambda_0=\frac{1}{p-1}$ 
otherwise and $\tilde{p}(\lambda_0)$ can be any positive exponent. If $p=2$ we set $\lambda=\lambda_0=1$ and $\tilde p(\lambda_0)=p^*$.
In particular, by Sobolev embedding, $u\in L^{\tilde p(\lambda)}(\Omega;\R^n)$
for all $\lambda\in[\lambda_0,1]$.

\begin{proposition}\label{p:regVmusuper}
Let $n\ge 2$, $p\in [2,\infty)$, $\cost$ and $\mu\geq 0$, $g\in W^{1,p}(\Omega;\R^n)$
and let $u\in W^{1,p}(\Omega;\R^n)$ be a local minimizer of $\FF$ defined in \eqref{e:FF}.

Then, $V_\mu(e(u))\in W^{1,2}_\loc(\Omega;\R^{n\times n}_\sym)$ and, 
in addition, $u\in W^{2,2}_\loc(\Omega;\R^{n})$ if $p>2$ for $\mu>0$, 
and if $p=2$ for $\mu\geq 0$. 
More precisely, if $\lambda\in[\lambda_0,1]$ there is a constant $c=c(n,p,\lambda)>0$ such that for 
$B_{2r}(x_0)\subset\Omega$ 
\begin{multline}\label{e:CaccioppoliVmusuper}
\int_{B_r(x_0)}|\nabla \big(V_\mu(e(u))\big)|^2dx
\leq c\frac {1+\cost}{r^2}\int_{B_{2r}(x_0)}
|V_\mu(e(u))-(V_\mu(e(u)))_{B_{2r}(x_0)}|^2dx\\
+c\,\cost r^{\frac2{p-1}}\int_{B_{2r}(x_0)}\big(|u-g|^{\tilde{p}(\lambda)}+|\nabla(u-g)|^{\lambda p}\big)dx.
\end{multline}
\end{proposition}
\begin{proof}
We begin with showing that there is a constant $c=c(n,p)>0$ such that if $B_{2r}(x_0)\subset\Omega$ 
{then for any matrix} $Q\in\R^{n\times n}$ we have 
 \begin{multline}\label{e:reg1b}
  \int_{B_r(x_0)}|\nabla(V_\mu(e(u)))|^2dx\leq
  \frac c{r^2}\int_{B_{2r}(x_0)\setminus B_{r}(x_0)}\phi_{|e(u)|}(|\nabla u-Q|)dx\\
  +\frac{\cost}{r^2}\int_{B_{2r}(x_0)}|\nabla u-Q|^pdx
  +c\,\cost\,r^{\frac2{p-1}}\int_{B_{2r}(x_0)}\big(|u-g|^p+|\nabla(u-g)|^p\big)dx\,.
 \end{multline}
In particular, on account of \eqref{a:n2}, we infer from \eqref{e:reg1b} that
$V_\mu(e(u){)}\in W^{1,2}(B_r(x_0))$. A covering argument implies then that 
$V_\mu(e(u))\in W^{1,2}_\loc(\Omega;\R^{n\times n}_\sym)$.

Local minimality yields that $u$ is a solution of
 \begin{equation}\label{e:ELpde}
  \int_\Omega\langle \nabla f_\mu(e(u)),e(\varphi)\rangle dx
  +\cost\,p\int_\Omega|u-g|^{p-2}\langle u-g,\varphi\rangle dx
   =0\quad\forall\varphi\in W^{1,p}_0(\Omega;\R^n).
 \end{equation}
 We can use the test field $\varphi:=\triangle_{s,-h}\big(\zetaq(\triangle_{s,h}u-Q\epsilon_s)\big)$, 
 with $\zeta\in C^\infty_c(B_{2r}(x_0))$, $0\leq\zeta\leq 1$, 
 $\zeta|_{B_r(x_0)}\equiv 1$ and $|\nabla\zeta|\leq c/r$ 
 to infer, for $h$ sufficiently small,
\begin{multline}\label{e:test1}
\int_\Omega\langle \triangle_{s,h}\big(\nabla f_\mu(e(u))\big),
\zetaq\triangle_{s,h}(e(u))\rangle dx\\
=-{p}\int_\Omega\langle \triangle_{s,h}\big(\nabla f_\mu(e(u))\big),
 \zetaqu\nabla\zeta\odot(\triangle_{s,h}u-Q\epsilon_s)
\rangle dx\\
-\cost\,p\int_\Omega\langle \triangle_{s,h}\big(|u-g|^{p-2}(u-g)\big),
\zetaq(\triangle_{s,h}u-Q\epsilon_s)\rangle dx.
\end{multline}
Recalling that $f_\mu\in C^2(\R^{n\times n})$ if $p\geq2$ for all $\mu\geq0$ we compute
\begin{multline}\label{e:tshf1}
 \triangle_{s,h}\big(\nabla f_\mu(e(u))\big)(x)=
 \int_0^1\nabla^2 f_\mu\big(e(u)+th\,\triangle_{s,h}(e(u))\big)\,\triangle_{s,h}(e(u))dt
\\ =:\mathbb{A}_{s,h}(x)\triangle_{s,h}(e(u))(x).
\end{multline}
By taking into account \eqref{e:tshf1}, equality \eqref{e:test1} rewrites as
\begin{align}\label{e:test2}
\int_\Omega \zetaq &\langle\mathbb{A}_{s,h}(x)\triangle_{s,h}(e(u)),\triangle_{s,h}(e(u)\big)\rangle dx\notag\\
=&-p\int_\Omega \zetaqu\langle \mathbb{A}_{s,h}(x)\triangle_{s,h}(e(u)),\nabla\zeta\odot(\triangle_{s,h}u-Q\epsilon_s)\rangle dx
\notag\\
&-\cost\,p\int_\Omega\langle \triangle_{s,h}\big(|u-g|^{p-2}(u-g)\big),\zetaq(\triangle_{s,h}u-Q\epsilon_s)\rangle dx\,.
\end{align}
Setting 
\[
W_{s,h}(x):=\int_0^1\big(\mu+|e(u)(x)+t\,\tau_{s,h}(e(u))(x)|^2\big)^{\sfrac p2-1}dt,
\]
the estimates in \eqref{e:bdhessf} give for all $\eta\in\R^{n\times n}_\sym$
\begin{equation}\label{e:wsh}
\frac 1c W_{s,h}(x)|\eta|^2\leq 
\langle \mathbb{A}_{s,h}(x)\eta,\eta\rangle\leq 
c\,W_{s,h}(x)|\eta|^2
\end{equation}
with $c=c(p)>0$. Therefore, using \eqref{e:wsh} in \eqref{e:test2} yields for some
$c=c(p)>0$ 
\begin{multline}\label{e:test3}
\int_\Omega\zetaq W_{s,h}(x) |\triangle_{s,h}(e(u))|^2dx
\leq c\int_\Omega\zetaqu W_{s,h}(x)|\nabla\zeta||\triangle_{s,h}(e(u))||\triangle_{s,h}u-Q\epsilon_s|dx\\
-\cost\,p\int_\Omega\langle \triangle_{s,h}\big(|u-g|^{p-2}(u-g)\big),
\zetaq(\triangle_{s,h}u-Q\epsilon_s)\rangle dx\,.
\end{multline}
Proceeding as in (\ref{e:tshf1}), and using 
$|\nabla V_\mu(\xi)|\le c\,(\mu+|\xi|^2)^{\sfrac{(p-2)}4}$, we obtain
\[
|\triangle_{s,h}\big(V_\mu(e(u))\big)|\leq
c \int_0^1 
\big(\mu+|e(u)(x)+t\,\tau_{s,h}(e(u))(x)|^2\big)^{\sfrac{(p-2)}4}dt 
|\triangle_{s,h}(e(u))|\,.
\]
Using Jensen's inequality in this integral and then comparing with the definition of 
$W_{s,h}$ we infer from \eqref{e:test3}
\begin{multline*}
 \int_{\Omega}\zetaq|\triangle_{s,h}\big(V_\mu(e(u))\big)|^2dx\leq
c\int_\Omega\zetaqu W_{s,h}(x)|\nabla\zeta||\triangle_{s,h}(e(u))||\triangle_{s,h}u-Q\epsilon_s|dx\\
-\cost\,p\int_\Omega\langle \triangle_{s,h}\big(|u-g|^{p-2}(u-g)\big),
\zetaq(\triangle_{s,h}u-Q\epsilon_s)\rangle dx\,.
\end{multline*}
In turn, from this inequality and \eqref{e:tecnica} we get for some $c=c(p)>0$ 
\begin{multline}\label{e:test2t}
 \int_{B_{2r}(x_0)}\zetaq|\tau_{s,h}\big(V_\mu(e(u))\big)|^2dx\\
 \leq\frac cr\int_{B_{2r}(x_0)\setminus B_{r}(x_0)}
\zetaqu\big(\mu+|e(u)|^2+|e(u)(x+he_s)|^2\big)^{\sfrac p2-1}|\tau_{s,h}(e(u))|
|\tau_{s,h}u-hQ\epsilon_s|dx\\
-\cost\,p\int_{B_{2r}(x_0)}\langle \tau_{s,h}\big(|u-g|^{p-2}(u-g)\big),\zetaq(\tau_{s,h}u-Qh\epsilon_s)\rangle 
 dx=:I_1+I_2.
 \end{multline}
By considering the functions $\phi_a$, $a\geq 0$, introduced in \eqref{a:a} above, 
the first term on the right hand side of the last inequality can be estimated by 
\[
 I_1=\frac cr\int_{B_{2r}(x_0)\setminus B_{r}(x_0)}
\zetaqu\phi^\prime_{|e(u)|}\big(|\tau_{s,h}(e(u))|\big)|\tau_{s,h}u-hQ\epsilon_s|dx.
\]
Since $\zeta\in(0,1]$, Young's inequality in \eqref{a:e2} gives for every $\delta\in(0,1)$ 
and for some $c=c(p)>0$ 
\begin{eqnarray*}
I_1\hskip-0.5cm&&\stackrel{\eqref{a:e2}}{\leq} c\,\delta \int_{B_{2r}(x_0)\setminus B_{r}(x_0)}
\phi^*_{|e(u)|}\big(\zetaqu\phi^\prime_{|e(u)|}(|\tau_{s,h}(e(u))|)\big)dx\\
&&+
C_{\delta,p}\int_{B_{2r}(x_0)\setminus B_{r}(x_0)}
\phi_{|e(u)|}\big(\frac 1r|\tau_{s,h}u-hQ\epsilon_s|\big)dx\\
&&\stackrel{\eqref{a:c}}{\leq}c\,\delta \int_{B_{2r}(x_0)\setminus B_{r}(x_0)}
\zetaq
\phi^*_{|e(u)|}\big(\phi^\prime_{|e(u)|}(|\tau_{s,h}(e(u))|)\big)dx\\
&&
+C_{\delta,p}
\int_{B_{2r}(x_0)\setminus B_{r}(x_0)}
\phi_{|e(u)|}\big(\frac 1r|\tau_{s,h}u-hQ\epsilon_s|\big)dx\\
&&\stackrel{\eqref{a:f}}{\leq}c\,\delta \int_{B_{2r}(x_0)\setminus B_{r}(x_0)}
\zetaq
\phi_{|e(u)|}\big(|\tau_{s,h}(e(u))|\big)dx\\
&&+
C_{\delta,p}
\int_{B_{2r}(x_0)\setminus B_{r}(x_0)}\phi_{|e(u)|}\big(\frac 1r|\tau_{s,h}u-hQ\epsilon_s|\big)dx.
\end{eqnarray*}
By using estimate \eqref{eqphixivmu} in the last but one term from the latter inequality we get 
for some $c=c(p)>0$
\begin{eqnarray}\label{e:I1}
 I_1\hskip-0.5cm&&\leq c\,\delta \int_{B_{2r}(x_0)\setminus B_{r}(x_0)}
\zetaq|\tau_{s,h}\big(V_\mu(e(u))\big)|^2dx\notag\\&&
+C_{\delta,p}\int_{B_{2r}(x_0)\setminus B_{r}(x_0)}
\phi_{|e(u)|}\big(\frac 1r|\tau_{s,h}u-hQ\epsilon_s|\big)dx.
\end{eqnarray}

We now estimate the second term in \eqref{e:test2t}. We preliminarily note 
that by Meyers-Serrin's theorem and the Chain rule formula for Sobolev functions the
field $w:=|u-g|^{p-2}(u-g)$ belongs to $W^{1,p'}(A,\R^n)$ for every Lipschitz 
open subset $A\subseteq\Omega$. More precisely, we have
\begin{equation}\label{e:w}
 \|\nabla w\|_{L^{p'}(A,\R^{n\times n})}\leq 
 c\|u-g\|_{L^{\tilde{p}(\lambda)}(A,\R^n)}^{p-2}
 \|\nabla(u-g)\|_{L^{\lambda p}(A,\R^{n\times n})},
\end{equation}
for some constant $c=c(n,p,\lambda)>0$, for all $\lambda\in[\lambda_0,1]$ where 
$\tilde{p}$ is the function defined in \eqref{e:ptilde} and $p'=\frac p{p-1}$
(we recall that if $p=2$ then $\lambda=\lambda_0=1$).

Therefore, by \eqref{e:w}, H\"older's and Young's inequalities we may estimate 
$I_2$ for $h$ sufficiently small as follows
\begin{eqnarray}\label{e:I2}
h^{-2}I_2\hskip-0.5cm&&\leq \cost\,p
\|\triangle_{s,h}u-Q\epsilon_s\|_{L^{p}(B_{2r}(x_0),\R^n)}
\|\zetaq\triangle_{s,h}w\|_{L^{p'}(B_{2r}(x_0),\R^n)}\notag\\ &&
\leq \cost\,p\,
\|\triangle_{s,h}u-Q\epsilon_s\|_{L^{p}(B_{2r}(x_0),\R^n)}
\|\nabla w\|_{L^{p'}(B_{2r}(x_0),\R^n)}\notag\\
&&\leq 
\frac{\cost}{r^2}\int_{B_{2r}(x_0)}|\triangle_{s,h}u-Q\epsilon_s|^pdx
+\cost(p-1)\,r^{\frac 2{p-1}}\int_{B_{2r}(x_0)}|\nabla w|^{p'}dx\notag\\
&&\leq 
\frac{\cost}{r^2}\int_{B_{2r}(x_0)}|\triangle_{s,h}u-Q\epsilon_s|^pdx 
+c\,\cost\,r^{\frac 2{p-1}}\int_{B_{2r}(x_0)}\big(|u-g|^{\tilde{p}(\lambda)}+|\nabla(u-g)|^{\lambda p}\big)dx
\notag\\
\end{eqnarray}
for some $c=c(n,p,\lambda)>0$.
Hence, from inequalities \eqref{e:test2t}, \eqref{e:I1} and \eqref{e:I2} for 
$\delta=\delta(p)>0$ sufficiently small we conclude that 
\begin{multline}\label{e:intmindue}
\int_{B_{r}(x_0)}|\triangle_{s,h}\big(V_\mu(e(u))\big)|^2dx
\leq \frac{c}{h^2}
\int_{B_{2r}(x_0)\setminus B_{r}(x_0)}\phi_{|e(u)|}\big(\frac 1r
|\tau_{s,h}u-hQ\epsilon_s|\big)dx\\
+
\frac{\cost}{r^2}\int_{B_{2r}(x_0)}|\triangle_{s,h}u-Q\epsilon_s|^pdx
+c\,\cost\, r^{\frac 2{p-1}}\int_{B_{2r}(x_0)}\big(|u-g|^{\tilde{p}(\lambda)}+|\nabla(u-g)|^{\lambda p}\big)dx,
\end{multline}
with $c=c(n,p,\lambda)>0$. Finally, \eqref{a:b} and 
the last inequality for sufficiently small $h$ yield for some $c=c(n,p,\lambda)>0$
 \begin{multline}\label{e:test5b}
  \int_{B_{r}(x_0)}|\triangle_{s,h}\big(V_\mu(e(u))\big)|^2dx
 \leq  
 \frac{c}{r^2}\int_{B_{2r}(x_0)\setminus B_{r}(x_0)}
 \phi_{|e(u)|}\big(|\triangle_{s,h}u-Q\epsilon_s|\big)dx\\
  +\frac{\cost}{r^2}\int_{B_{2r}(x_0)}|\triangle_{s,h}u-Q\epsilon_s|^pdx
 +c\,\cost\,r^{\frac 2{p-1}}\int_{B_{2r}(x_0)}\big(|u-g|^{\tilde{p}(\lambda)}+|\nabla(u-g)|^{\lambda p}\big)dx.
 \end{multline}
Hence, by summing on $s\in\{1,\ldots,n\}$ in inequality \eqref{e:test5b} and by 
letting $h\downarrow 0$ there, we conclude 
\eqref{e:reg1b}. 
Furthermore, since $|\nabla V_\mu(\xi)|^2\geq c(n,p)\mu^{\sfrac p2-1}$ 
for all $\xi\in\R^{n\times n}$ if $p\ge 2$, the latter estimate, \eqref{e:reg1b} 
and a covering argument imply $u\in W^{2,2}_\loc(\Omega;\R^n)$ if $\mu>0$.

To conclude the Caccioppoli's type inequality in \eqref{e:CaccioppoliVmusuper} first observe that
\begin{equation}\label{e:Kornsuper}
 \int_{B_{2r}(x_0)}\phi_{|e(u)|}\big(|\nabla u-(\nabla u)_{B_{2r}(x_0)}|\big)dx\\
 \le c
 \int_{B_{2r}(x_0)}\phi_{|e(u)|}\big(|e(u)-(e(u))_{B_{2r}(x_0)}|\big)dx.
\end{equation}
This follows from Korn's inequality by using that if $\psi_a(t):=a^{p-2}t^2+\mu^{\sfrac p2-1}t^2 +t^p$ 
then $c^{-1}\psi_a(t)\leq \phi_a(t)\leq c\,\psi_a(t)$ for all $t\geq 0$ and for some $c=c(p)>0$. One 
inequality follows from \eqref{a:n2}, the other one is similar. Alternatively, \eqref{e:Kornsuper} follows 
directly from Korn's inequality in Orlicz spaces for shifted N-functions (cf. \cite[Lemma~2.9]{DieningKaplicky}).

Moreover, since for $p\geq 2$ by the very definition of $V_\mu$ and Lemma~\ref{l:tecnico2} 
\[
|\xi-\eta|^p\leq|V_\mu(\xi-\eta)|^2\leq c(p)|V_\mu(\xi)-V_\mu(\eta)|^2\qquad
\forall\,\xi,\,\eta\in\R^{n\times n},
\]
the standard Korn's inequality implies for some $c=c(n,p)>0$
\begin{multline}\label{e:Kornsuper2}
\int_{B_{2r}(x_0)}|\nabla u-(\nabla u)_{B_{2r}(x_0)}|^pdx\leq c
\int_{B_{2r}(x_0)}|e(u)-(e(u))_{B_{2r}(x_0)}|^pdx\\
\leq c\int_{B_{2r}(x_0)}|V_\mu\big(e(u)-(e(u))_{B_{2r}(x_0)}\big)|^2dx
\leq c\int_{B_{2r}(x_0)}|V_\mu(e(u))-V_\mu(e(u)_{B_{2r}(x_0)})|^2dx.
\end{multline}
Thus, by combining \eqref{e:Kornsuper} and  \eqref{e:Kornsuper2} with 
\eqref{e:reg1b}, with $Q:=(\nabla u)_{B_{2r}(x_0)}$, 
we deduce
\begin{multline}\label{e:test0b}
 \int_{B_{r}(x_0)}|\nabla \big(V_\mu(e(u))\big)|^2dx\leq \frac c{r^2}
 \int_{B_{2r}(x_0)}\phi_{|e(u)|}\big(|e(u)-(e(u))_{B_{2r}(x_0)}|\big)dx\\
+\cost\frac{c}{r^2}\int_{B_{2r}(x_0)}|V_\mu(e(u))-V_\mu((e(u))_{B_{2r}(x_0)})|^2dx
 +c\,\cost\,r^{\frac 2{p-1}}\int_{B_{2r}(x_0)}\big(|u-g|^{\tilde{p}(\lambda)}+|\nabla(u-g)|^{\lambda p}\big)dx,
\end{multline}
for some constant $c=c(n,p,\lambda)>0$. Hence, by \eqref{eqphixivmu} we get from 
\eqref{e:test0b}
\begin{multline*}
\int_{B_{r}(x_0)}|\nabla \big(V_\mu(e(u))\big)|^2dx\\
\leq 
c\frac {1+\cost}{r^2}\int_{B_{2r}(x_0)}|V_\mu(e(u))-V_\mu((e(u))_{B_{2r}(x_0)})|^2dx
+c\,\cost\,r^{\frac2{p-1}}\int_{B_{2r}(x_0)}\big(|u-g|^{\tilde{p}(\lambda)}+|\nabla(u-g)|^{\lambda p}\big)dx\\
\leq c\frac {1+\cost}{r^2}
\int_{B_{2r}(x_0)}|V_\mu(e(u))-(V_\mu(e(u)))_{B_{2r}(x_0)}|^2dx
+c\,\cost\,r^{\frac 2{p-1}}\int_{B_{2r}(x_0)}\big(|u-g|^{\tilde{p}(\lambda)}+|\nabla(u-g)|^{\lambda p}\big)dx.
\end{multline*}
The last inequality follows from Lemma~\ref{l:DieningKaplicky}.
\end{proof}

In the sub-quadratic case we use a regularization argument following 
\cite[Theorem~3.2]{DieningKaplicky}. Indeed, even setting $\cost=0$,  
the same arguments as in Proposition~\ref{p:regVmusuper} lead only to 
a Besov type estimate. More precisely, the first part of the argument 
in Proposition~\ref{p:regVmusuper} up to \eqref{e:I1} included, holds for all 
$p\in (1,\infty)$ {(one only has to use $\zeta^2$ instead of $\zeta^p$ as a cutoff function)}. Thus, in case $p\in(1,2)$, arguing similarly to 
Proposition~\ref{p:regVmusuper} one deduces the ensuing estimate 
\begin{equation*}
  \left[V_\mu(e(u))\right]^2_{B^{\sfrac p2,2,\infty}(B_{r}(x_0))}
  \leq\frac{c}{r^p}\int_{B_{2r}(x_0)}
  \big(\mu+|\nabla u-(\nabla u)_{B_{2r}(x_0)}|^2\big)^{\sfrac p2}dx,
\end{equation*}
for some $c=c(n,p)>0$, which is not sufficient for our purposes. Recall that the Besov space 
$B^{\sfrac p2,2,\infty}(A)$, $A\subset\R^n$ open, is the space of maps $v\in L^2(A;\R^{n\times n})$ 
such that
\[
 [v]_{B^{\sfrac p2,2,\infty}(A)}:=
 \sup_{h}{|h|^{-\sfrac p2}}\sum_{s=1}^{n}\|\tau_{s,h}v\|_{L^2(A;\R^{n\times n})}<\infty.
\]
Finally, we point out that the argument we use below requires only minimal assumptions 
on $g$, namely $L^p$ summability. 
We start off with establishing a technical result.
\begin{lemma}\label{l:regVmusub}
Let $n\ge 2$, $p\in (1,2]$, $\cost$ and $\mu\geq 0$, $g\in L^p(B_{2r},\R^n)$,
and $w\in C^\infty(\overline{B}_{2r};\R^n)$. Let $u$ be the minimizer of 
\begin{equation}\label{e:Flk}
F_L(v):=\int_{B_{2r}} f_\mu(e(v)) dx +\cost\int_{B_{2r}}|v-g|^pdx
+\frac1{2L}\int_{B_{2r}} |\nabla^2v|^2 dx
\end{equation}
over the set of $w+W^{2,2}_0(B_{2r},\R^n)$. 
Then, $u\in W^{3,2}_\loc(B_{2r},\R^n)$ and there is a constant $c=c(n,p)>0$ such that {for all $\lambda\geq 0$}
\begin{multline}\label{e:st7}
\frac 1L \int_{B_{r}} |\nabla\Delta u|^2dx
+ \int_{B_{r}} |\nabla(V_\mu(e(u)))|^2dx 
\le 
\frac{c}{Lr^4}\int_{B_{2r}\setminus B_r} |\nabla u-(\nabla u)_{B_{2r}}|^2dx\\
+\frac{c}{Lr^2}\int_{B_{2r}\setminus B_r} |\nabla^2 u|^2dx
+c\,\frac{1+\cost}{r^2} 
\int_{B_{2r}}  |V_\mu(e(u))-(V_\mu(e(u)))_{B_{2r}})|^2  dx\\
+\cost\,{r^{\frac{\lambda  p}{p-1}}}\int_{B_{2r}}|u-g|^pdx
+c\,\frac{\cost}{{{r^{\frac{2 \lambda p}{2-p}}}}}\Big(\int_{B_{2r}}|V_\mu(e(u))|^2dx
+\mu^{\sfrac p2}r^n\Big).
\end{multline}
\end{lemma}
\begin{proof} 
We first prove that $u\in W^{3,2}_\loc(B_{2r},\R^n)$. 
Given $V\subset\subset B_{2r}$, 
we set $d:=\min\{1,\dist(V,\partial B_{2r})\}$ and take $h\leq d/2$. 
For $\rho\in (0,d/2)$ 
we consider the function
\[g(\rho):=\sup\left\{\frac1L \int_{B_\rho(y)}|\triangle_{s,h}\nabla^2 u|^2dx:\, y\in V\right\}.\]
Next we prove that there exists a constant $c>0$ independent from $h$ (but possibly depending on $L$) 
such that
\begin{equation}\label{e:dec}
g(\rho)\leq \frac {g(\rho')}{2}+ \frac{c}{(\rho'-\rho)^4}
+\frac{c\,\cost}{\rho'-\rho},
\end{equation}
for $\rho,\rho'\in (0,d/2)$, $\rho<\rho'$.

Fix $\rho,\rho'$ as above, $y\in V$, and consider
$\zeta\in C^\infty_c(B_{\rho'}(y))$, with $\zeta=1$ on $B_\rho(y)$ and 
$|\nabla^2 \zeta|\leq c/(\rho'-\rho)^2$. We now test 
\begin{equation}\label{e:ulk}
\frac1L\int_{B_{2r}}\langle \nabla^2 u, \nabla^2 \varphi\rangle dx + 
\int_{B_{2r}} \langle \nabla f_\mu(e(u)), e(\varphi)\rangle dx
+\cost\,p\int_{B_{2r}}|u-g|^{p-2}\langle u-g,\varphi\rangle\,dx=0\,,
\end{equation}
 holding for every $\varphi\in C^\infty_c(B_{2r};\R^n)$, with the test function 
 $\varphi:=\triangle_{s,-h}(\zeta\triangle_{s,h}u)$ and we estimate each appearing term.

First note that
\[\int_{B_{\rho'}(y)}\langle \nabla^2 u, \nabla^2 \varphi\rangle dx=-\int_{B_{\rho'}(y)}\langle \triangle_{s,h}\nabla^2 u, \zeta\nabla^2 \triangle_{s,h}u+z\rangle dx,\]
where the function $z$ satisfies
\[\|z\|_{L^2(B_{\rho'}(y))}\leq \frac{c}{(\rho'-\rho)^2}\|u\|_{W^{2,2}(B_{2r})}.\]
Therefore by Young's inequality we obtain
\begin{equation}\label{e:st1}
-\int_{B_{\rho'}(y)}\langle \triangle_{s,h}\nabla^2 u, z\rangle dx\leq \frac12 \int_{B_{\rho'}(y)}|\triangle_{s,h}\nabla^2 u|^2dx+\frac{c}{(\rho'-\rho)^4}\|u\|^2_{W^{2,2}(B_{2r})}.
\end{equation}
Moreover we have
\[
\int_{B_{\rho'}(y)}\langle \triangle_{s,h}\nabla^2 u, \zeta\nabla^2 \triangle_{s,h}u\rangle dx\geq \int_{B_{\rho}(y)}|\triangle_{s,h}\nabla^2 u|^2dx,
\]
so that by \eqref{e:st1}
\begin{multline}\label{e:st4}\int_{B_{\rho'}(y)}\langle \nabla^2 u, \nabla^2 \varphi\rangle dx\leq \frac12 \int_{B_{\rho'}(y)}|\triangle_{s,h}\nabla^2 u|^2dx+
\frac{c}{(\rho'-\rho)^4}\|u\|^2_{W^{2,2}(B_{2r})}\\
-\int_{B_{\rho}(y)}|\triangle_{s,h}\nabla^2 u|^2dx.
\end{multline}
On the set $(B_{2r})_{s,h}$ (recall the notation introduced right after \eqref{e:rapporti incrementali}) 
we define 
\begin{equation*}
\alpha^s(x):=\int_0^1 \nabla f_\mu(e(u(x+th\epsilon_s))) dt
\end{equation*}
and observe that (for $\mu>0$)
\begin{multline}\label{e:st5}
 \triangle_{s,h}\big(\nabla f_\mu(e(u))\big)
=\frac1h \int_0^h \frac{d}{dt} \big(\nabla f_\mu(e(u(x+t\epsilon_s)))\big)dt\\
 =\frac1h \int_0^h \partial_{s} \big(\nabla f_\mu(e(u(x+t\epsilon_s)))\big)dt
=\frac1h \partial_s \int_0^h  \big(\nabla f_\mu(e(u(x+t\epsilon_s)))\big){dt}=\partial_{s} \alpha^s\,.
\end{multline}
By continuity one obtains $ \triangle_{s,h}\big(\nabla f_\mu(e(u))\big)=\partial_{s} \alpha^s$ 
also for $\mu=0$. Therefore we estimate
\begin{multline}\label{e:st3}
\int_{B_{\rho'}(y)}\langle\nabla f_\mu(e(u)),e(\varphi)\rangle dx = 
-\int_{B_{\rho'}(y)}\langle \triangle_{s,h}\big(\nabla f_\mu(e(u))\big),
\zeta\triangle_{s,h}e(u)\rangle dx\\
-\int_{B_{\rho'}(y)}\langle \triangle_{s,h}\big(\nabla f_\mu(e(u))\big),
\nabla\zeta\odot\triangle_{s,h}u\rangle dx\\
 \leq \int_{B_{\rho'}(y)}\langle \alpha^s,
\partial_s\nabla\zeta\odot\triangle_{s,h}u\rangle dx
+\int_{B_{\rho'}(y)}\langle \alpha^s,
\nabla\zeta\odot\triangle_{s,h}\partial_s u\rangle dx\,,
\end{multline}
where we have used \eqref{e:tshf1}, \eqref{e:wsh} and 
\eqref{e:st5}. 
Since $u\in W^{2,2}(B_{2r}(y),\R^n)$ we conclude {with \eqref{e:growth condition nabla}}
\begin{multline}\label{e:st6}
\int_{B_{\rho'}(y)}\langle\nabla f_\mu(e(u)),e(\varphi)\rangle dx\leq 
 \frac{c}{(\rho'-\rho)^2} \|u\|_{W^{1,p}(B_{2r})} \|\alpha^s\|_{L^{p'}(B_{\rho'}(y))}\\+
\frac{c}{\rho'-\rho} \|u\|_{W^{2,p}(B_{2r})} \|\alpha^s\|_{L^{p'}(B_{\rho'}(y))} 
\leq\frac{c}{(\rho'-\rho)^2} \|u\|_{W^{2,p}(B_{2r})} \|\mu^{1/2}+|e(u)|\|_{L^{p}(B_{2r})}^{\sfrac p{p'}}. 
\end{multline}
Eventually, by H\"older's inequality and the standard properties of difference quotients 
we can estimate the last term on the left hand side of \eqref{e:ulk} as follows:
\begin{multline}\label{e:st6b}
\int_{B_{2r}}|u-g|^{p-2}\langle u-g,\varphi\rangle\,dx\leq 
\|u-g\|_{L^p(B_{2r})}^{p-1}
\Big(\frac{1}{\rho'-\rho}\|\nabla u\|_{L^p(B_{2r})}+\|\nabla^2u\|_{L^p(B_{2r})}\Big).
\end{multline}
Estimates \eqref{e:ulk}, \eqref{e:st4}, \eqref{e:st6} and \eqref{e:st6b} yield
\[
\int_{B_{\rho}(y)}|\triangle_{s,h}\nabla^2 u|^2dx\leq
\frac12 \int_{B_{\rho'}(y)}|\triangle_{s,h}\nabla^2 u|^2dx+
\frac{c}{(\rho'-\rho)^4}+\frac{c\,\cost}{\rho'-\rho},
\]
for a constant $c>0$ depending on $n$, $p$, $k$, {$L$}, the $W^{2,2}$ norm of $u$ and 
the $L^p$ norm of $g$. Then, \eqref{e:dec} follows at once since $y\in V$ is arbitrary. 
By this, \cite[Lemma~6.1]{Giusti}, and the compactness of $\overline{V}$ we finally infer
\[
\int_V|\triangle_{s,h}\nabla^2 u|^2dx\leq c,
\]
with $c$ independent from $h$, and therefore $u\in W^{3,2}_{\loc}(B_{2r},\R^n)$.
 
Let us now prove \eqref{e:st7}. Using $u\in W^{3,2}_\loc(B_{2r},\R^n)$ and the fact that 
$e(\varphi)$ has average zero for every $\varphi\in W^{1,2}_0(B_{2r},\R^n)$, we 
can rewrite \eqref{e:ulk} as
\begin{multline}\label{eqEL32}
\frac1L\int_{B_{2r}} \langle \nabla \Delta u, \nabla \varphi \rangle dx \\
= \int_{B_{2r}} \langle \nabla f_\mu(e(u))-\nabla f_\mu(e(Qx)), e(\varphi)\rangle dx
+\cost\,p\int_{B_{2r}}|u-g|^{p-2}\langle u-g,\varphi\rangle dx\,,
\end{multline}
for any $Q\in \R^{n\times n}$.

Let  now $\psi:=\sum_{s=1}^n\partial_s (\zeta^q \partial_s (u-Qx))$, where 
$q\ge 4$, let $\zeta\in C^\infty_c(B_{3r/2};[0,1])$ obey $\zeta=1$ on $B_r$ 
and $|\nabla \zeta|\le c/r$. Since $u\in W^{3,2}_\loc(B_{2r},\R^n)$, $\psi$ 
can be strongly approximated in $W^{1,2}(B_{2r},\R^n)$ by smooth functions 
supported in $B_{3r/2}$; therefore we can use $\varphi=\psi$ as a trial 
function in (\ref{eqEL32}).

We now estimate the three terms in  (\ref{eqEL32}).
We start from the second one, which we write as
\begin{equation*}
I_2:= \int_{B_{2r}} \langle B, e(\psi)\rangle dx
\end{equation*}
where
\begin{equation*}
B(x):=\nabla f_\mu(e(u))-\nabla f_\mu(e(Qx))
=\int_0^1\nabla^2 f_\mu(e(Qx)+t e(\tilde u)(x)) e(\tilde u)(x) dt
\end{equation*}
and $\tilde u(x):=u(x)-Qx$.
We estimate, using (\ref{e:bdhessf}) and Lemma \ref{l:tecnico},
\begin{align*}
|B|&\le \int_0^1(\mu+ |e(Qx)+t e(\tilde u)|^2)^{p/2-1} | e(\tilde u)|dt\\
&\le c(\mu+ |e(Qx)|^2+ |e(u)|^2)^{p/2-1} | e(\tilde u)|\\
&\le c(\mu+ (|e(Qx)|+ |e(\tilde u)|)^2)^{p/2-1} | e(\tilde u)|
=c\, \phi'_{|e(Qx)|} (| e(\tilde u)|)
\end{align*}
for $\mu>0$. By continuity, $|B|\le c\, \phi'_{|e(Qx)|} (| e(\tilde u)|)$ 
holds also for $\mu=0$. We compute 
\begin{equation*}
e(\psi)=\sum_{s=1}^n(\partial_s \nabla \zeta^q)\odot \partial_s \tilde u
+\nabla \zeta^q \odot \partial_s^2 \tilde u
+ \partial_s (\zeta^q \partial_s e(\tilde u))\,.
\end{equation*}
We estimate the three contributions to $I_2$ separately. Recalling the estimate for $B$, we obtain
\begin{align*}
|I_{2,1}|\le& \int_{B_{2r}\setminus B_r}  |B| \frac{c {q^2}}{r^2} | \nabla \tilde u| dx
\le  \frac{c{q^2}}{r^2} \int_{B_{2r}\setminus B_r}  \phi'_{|e(Qx)|} (| e(\tilde u)|) | \nabla \tilde u| dx\\
\le & \frac{c{q^2}}{r^2} \int_{B_{2r}\setminus B_r}  \phi'_{|e(Qx)|} (| \nabla \tilde u|) | \nabla \tilde u| dx
\le  \frac{c{q^2}}{r^2} \int_{B_{2r}\setminus B_r}  \phi_{|e(Qx)|} (| \nabla \tilde u|)  dx
\end{align*}
where we used monotonicity of $\phi_a'$ and (\ref{a:e}).
Using Korn's inequality for shifted N-functions (cf. \cite[Lemma~2.9]{DieningKaplicky} {or \eqref{e:Kornsuper}})
and choosing $Q:=(\nabla u)_{B_{2r}}$ we conclude
\begin{align*}
|I_{2,1}| \le  \frac{c{q^2}}{r^2} \int_{{B_{2r}}}  \phi_{|e(Qx)|} (| e( \tilde u)|)  dx
\le   \frac{c{q^2}}{r^2} \int_{{B_{2r}}}  |V_\mu(e( u))-V_\mu(e(Qx))|^2  dx
\end{align*}
{where in the last step we used 
(\ref{eqphixivmu}).}

For the second one, we use that for any function $v$ in $W^{2,2}_\loc$ one has
\begin{equation}\label{e:trick}
\partial_s^2 v_j = 2\partial_s ([e(v)]_{sj}) -\partial_j ([e(v)]_{ss}), 
\end{equation}
here $[e(v)]_{hk}$ denotes the entry of position $(h,k)$ of the matrix $e(v)$, 
to obtain
\begin{align*}
|I_{2,2}|\le& \frac{cq}{r} \int_{B_{2r}\setminus B_r} |B| \zeta^{q-1} |\Delta \tilde u| dx\\
\le &\frac{cq}{r}\int_{B_{2r}\setminus B_r}   \phi'_{|e(Qx)|} (| e(\tilde u)|) \zeta^{q-1} |\nabla e(\tilde u)| dx.
\end{align*}
Recalling (\ref{eqphiprimophisecondo}), choosing $q\ge 2$ and since 
$0\leq \zeta\le 1$, we deduce by Young's inequality
\begin{align*}
|I_{2,2}|\le& \frac{cq}{r}\int_{B_{2r}\setminus B_r}   \phi''_{|e(Qx)|} (|e(\tilde u)|) \zeta^{q-1} | e(\tilde u)| \, |\nabla e(\tilde u)| dx\\
\le &  \delta \int_{B_{2r}\setminus B_r} \zeta^{q}  \phi''_{|e(Qx)|} (| e(\tilde u)|)   |\nabla e(\tilde u)| ^2dx\\
&+ \frac{c}{r^2}\int_{B_{2r}\setminus B_r}   \phi''_{|e(Qx)|} (| e(\tilde u)|) | e(\tilde u)| ^2dx\,,
\end{align*}
with $c=c(\delta,q)>0$ and $\delta\in(0,1)$ to be chosen below.
Hence, recalling 
$|\nabla V_\mu(\xi)|^2\ge c(\mu+|\xi|^2)^{\sfrac p2-1}$ and $\phi_a''(|t-a|)|t-a|^2\le c|V_\mu(t)-V_\mu(a)|^2$ 
(see Lemma~\ref{l:tecnico2} and the definition of $\phi_a$), we infer 
\[
|I_{2,2}|\le  \delta \int_{B_{2r}\setminus B_r}   \zeta^{q} |\nabla( V_\mu(e( u)))|^2dx+ \frac{c}{r^2}\int_{B_{2r}\setminus B_r}   |V_\mu(e(u))-V_\mu(e(Qx))|^2dx\,.
\]
Finally, to deal with the last term $I_{2,3}$ we integrate by parts. Since 
$\partial_s B=\nabla^2 f_\mu(e(u))\partial_s e(u)$, recalling (\ref{e:bdhessf})
and the definition of $V_\mu$
\begin{align*}
-I_{2,3}=&\int_{B_{2r}}\zeta^q \sum_{s=1}^n\langle \nabla^2 f_\mu(e(u))\partial_s e(u), \partial_s e(u)\rangle dx\\
\ge& c\int_{B_{2r}} \zeta^q(\mu+|e(u)|^2)^{p/2-1} |\nabla e(u)|^2dx 
\ge c\int_{B_{2r}} \zeta^q|\nabla (V_\mu(e(u)))|^2dx \,,
\end{align*}
with $c=c(p)>0$.

We now turn to the first term in \eqref{eqEL32}, 
\begin{equation*}
I_1:=\int_{B_{2r}} \langle\nabla \Delta u, \nabla \psi\rangle dx\,.
\end{equation*}
Again we consider separately the contributions of the different components of $\nabla\psi$,
\begin{equation*}
\nabla\psi=\sum_{s=1}^n\partial_s \tilde u\otimes (\partial_s \nabla \zeta^q)
+ \partial_s^2 \tilde u \otimes \nabla \zeta^q
+ (\partial_s \zeta^q) \partial_s \nabla \tilde u
+  \zeta^q \partial_s^2 \nabla \tilde u\,.
\end{equation*}
The first term is controlled by
\begin{align*}
|I_{1,1}|\le &  \frac{c}{r^2} 
\int_{B_{2r}\setminus B_r}|\nabla \Delta u|\zeta^{q-2} |\nabla \tilde u|dx\\
\le& \delta \int_{B_{2r}\setminus B_r}\zeta^q |\nabla \Delta u|^2 dx
+\frac{c}{r^4}\int_{B_{2r}\setminus B_r} |\nabla \tilde u|^2dx
\end{align*}
for some $c=c(q,\delta)>0$, provided that $q-2\ge q/2$, namely 
$q\geq 4$.
The second and the third terms are controlled, for some $c=c(q,\delta)>0$, by
\begin{align*}
|I_{1,2}+I_{1,3}|\le & \frac{c}{r}\int_{B_{2r}\setminus B_r}|\nabla \Delta u|  \zeta^{q-1} |\nabla^2 \tilde u|dx\\
\le& \delta \int_{B_{2r}\setminus B_r}\zeta^q |\nabla \Delta u|^2 dx
+\frac{c}{r^2}\int_{B_{2r}\setminus B_r} |\nabla^2 \tilde u|^2dx\,.
\end{align*}
The fourth summand in $I_1$ is
\begin{equation*}
I_{1,4}:=\int_{B_{2r}}{\zeta^q} \nabla\Delta u\cdot  \nabla\Delta udx\,.
\end{equation*}
We deal with the remaining term in \eqref{eqEL32}
\[
I_3:=\cost\,p\int_{B_{2r}}|u-g|^{p-2}\langle u-g,
\sum_{s=1}^n\partial_s(\zeta^q\partial_s\tilde{u})\rangle dx.
\]
H\"older's and Young's inequalities together with \eqref{e:trick} yield for some
constant $c=c(p,q)>0$ 
\[
\cost^{-1}I_3\leq 
{r^{\frac{\lambda p}{p-1}}}\int_{B_{2r}}|u-g|^pdx+\frac c{{r^{(\lambda+1)p}}}\int_{B_{2r}\setminus B_r}
|\nabla\tilde u|^pdx
+\frac c{{r^{\lambda p}}}\int_{B_{2r}}\zeta^q|\nabla e(u)|^pdx.
\]
Recalling that we have chosen $Q=(\nabla u)_{B_{2r}}$,
 apply 
Korn's inequality {to obtain
\begin{equation*}
\cost^{-1}I_{3,2} \leq\frac c{{r^{(\lambda+1)p}}}\int_{B_{2r}}|e(u)-e(Qx)|^pdx .
\end{equation*}
From} Lemma~\ref{l:tecnico2} 
{we obtain
\begin{equation*}
 |\xi-\eta|^p \le c |V(\xi)-V(\eta)|^p(\mu+|\xi|^2+|\eta|^2)^{\sfrac{p(2-p)}4} ,
\end{equation*}
using Young's inequality }
and 
Remark~\ref{r:Vmu} we conclude that
\begin{align*}
\cost^{-1}I_{3,2}  & \leq\frac 1{r^2}\int_{B_{2r}}|V_\mu(e(u))-V_\mu(e(Qx))|^2dx+
{\frac c{{r^{\frac{2\lambda p}{2-p}}}}}\int_{B_{2r}}(\mu+|e(u)|^2)^{\sfrac p2}dx\\
& \leq\frac 1{r^2}\int_{B_{2r}}|V_\mu(e(u))-V_\mu(e(Qx))|^2dx+
\frac c{{r^{\frac{2\lambda p}{2-p}}}}\Big(\int_{B_{2r}}|V_\mu(e(u))|^2dx
+\mu^{\sfrac p2}r^n\Big).
\end{align*}
with $c=c(p)>0$. Furthermore, again by Lemma \ref{l:tecnico2},
Young's inequality and Remark \ref{r:Vmu} we have that
\[
\cost^{-1}I_{3,3}\leq \delta\int_{B_{2r}}\zeta^q
|\nabla\big(V_\mu(e(u))\big)|^2dx+
{\frac c{{r^{\frac{2\lambda p}{2-p}}}}}\Big(\int_{B_{2r}}|V_\mu(e(u))|^2dx
+\mu^{\sfrac p2}r^n\Big).
\]
for some $c=c(\delta, p)>0$. Therefore, we deduce that
\begin{align*}
\cost^{-1}I_3 \leq & \,{r^{\frac{\lambda p}{p-1}}}\int_{B_{2r}}|u-g|^pdx
+\frac 1{r^2}\int_{B_{2r}}
|V_\mu(e(u))-V_\mu(e(Qx))|^2dx+\\
&+\delta\int_{B_{2r}}\zeta^q|\nabla\big(V_\mu(e(u))\big)|^2dx
+{\frac c{{{r^{\frac{2\lambda p}{2-p}}}}}}\Big(\int_{B_{2r}}|V_\mu(e(u))|^2dx
+\mu^{\sfrac p2}r^n\Big).
\end{align*}

Finally, we rewrite (\ref{eqEL32}) as
\begin{align*}
\frac1L I_{1,4}-I_{2,3}\le \frac1L|I_{1,1}|+\frac1L|I_{1,2}+I_{1,3}|+I_{2,1}+I_{2,2}+I_3\,.
\end{align*}
Choosing $q\geq 4$ and $\delta\in(0,\sfrac 14]$, 
for some constant $c=c(n,p)>0$ we have that 
\begin{align*} 
&\frac1L \int_{B_{2r}} \zeta^q |\nabla\Delta u|^2dx
+ \int_{B_{2r}}\zeta^q |\nabla\big(V_\mu(e(u))\big) |^2dx \\
\le& 
\frac{c}{Lr^4}\int_{B_{2r}\setminus B_r} |\nabla \tilde u|^2dx
+\frac{c}{Lr^2}\int_{B_{2r}\setminus B_r} |\nabla^2 \tilde u|^2dx
\\
&+c\,\frac{1+\cost}{r^2}
\int_{B_{2r}}|V_\mu(e(u))-V_\mu(e(Qx))|^2dx
\\
&+\cost\,{r^{\frac{\lambda p}{p-1}}}\int_{B_{2r}}|u-g|^pdx
+c\,{\frac{\cost}{{{r^{\frac{2\lambda p}{2-p}}}}}}\Big(\int_{B_{2r}}|V_\mu(e(u))|^2dx
+\mu^{\sfrac p2}r^n\Big),
\end{align*}
and \eqref{e:st7} follows at once from Lemma~\ref{l:DieningKaplicky}.
\end{proof}
We are now ready to prove the Caccioppoli's inequality in the sub-quadratic case.
\begin{proposition}\label{p:regVmusub}
Let $n\ge 2$, $p\in (1,2]$, $\cost$ and $\mu\geq 0$ and $g\in L^p(\Omega;\R^n)$.
Let $u\in W^{1,p}(\Omega;\R^n)$ be a local minimizer of $\FF$ defined in \eqref{e:FF}.

Then, $V_\mu(e(u))\in W^{1,2}_\loc(\Omega;\R^{n\times n}_\sym)$ and 
$u\in W^{2,p}_\loc(\Omega;\R^n)$.
More precisely, there is a constant $c=c(n,p)>0$ 
such that if $B_{2r}(x_0)\subset\Omega$ 
{and $\lambda\geq 0$}
\begin{multline}\label{e:CaccioppoliVmusub}
\int_{B_{r}(x_0)}|\nabla \big(V_\mu(e(u))\big)|^2dx
\leq c\frac{1+\cost}{r^2}\int_{B_{2r}(x_0)}
|V_\mu(e(u))-(V_\mu(e(u)))_{B_{2r}(x_0)}|^2dx\\
+\cost\,{r^{\frac{\lambda p}{p-1}}}\int_{B_{2r}(x_0)}|u-g|^pdx+
c\,\frac{\cost}{{r^{\frac{2\lambda p}{2-p}}}}\Big(
\int_{B_{2r}(x_0)}|V_\mu(e(u))|^2dx+
\mu^{\sfrac p2}r^n\Big),
\end{multline}
and
\begin{equation}\label{e:eureg}
\int_{B_r(x_0)}|\nabla(e(u))|^pdx
\leq c\Big(\int_{B_r(x_0)}|\nabla(V_\mu(e(u)))|^2dx\Big)^{\frac p2}
\Big(\int_{B_r(x_0)}(\mu+|e(u)|^2)^{\frac p2}dx\Big)^{1-\frac p2}.
\end{equation}
\end{proposition}
\begin{proof}
By a simple translation argument we can assume $x_0=0$ without loss of generality. 
We consider the functionals $F_L$ defined in \eqref{e:Flk}
and correspondingly we define 
\[
{\mathscr{F}_\infty}(v):=\int_{B_{2r}} f_\mu(e(v)) dx+\cost\int_{B_{2r}}|v-g|^pdx\,.
\]
Fix a sequence $u_l\in C^\infty(\overline B_{2r};\R^n)$ which converges strongly in $W^{1,p}$ to $u$, 
and let $u_{l,L}$ be the minimizer of $F_L$ over the set of $W^{1,p}(B_{2r};\R^n)$ functions
which coincide with $u_l$ on the boundary, correspondingly  $u_l^*$ for 
$\mathscr{F}_\infty$.

For a fixed $l$, let $v$ be a smooth approximation to $u_l^*$ with the same boundary data. 
Then $F_L(v)\to \mathscr{F}_\infty(v)$ as $L\uparrow\infty$. Since $F_L\ge \mathscr{F}_\infty$, 
this implies that $F_L(u_{l,L})\to \mathscr{F}_\infty(u_l^*)$ as $L\uparrow\infty$.
In particular, the sequence $u_{l,L}$ is a minimizing sequence for $\mathscr{F}_\infty$, 
and since this functional is strictly convex it converges strongly in $W^{1,p}$ to the unique 
minimizer $u_l^*$ of $\mathscr{F}_\infty$. Further, $L^{-1}\int_{B_{2r}} |\nabla^2u_{l,L}|^2dx\to0$.

Using Lemma \ref{l:regVmusub} with $w=u_l$ and taking the limit $L\uparrow\infty$ in \eqref{e:st7} we obtain
\begin{multline*}
 \int_{B_r} |\nabla(V_\mu(e(u_{l}^*)))|^2dx  
 \le c\,\frac{1+\cost }{r^2}
 \int_{{B_{2r}}}   |V_\mu(e(u_{l}^*))-(V_\mu(e(u_l^*)))_{{B_{2r}}}|^2dx
\\
+\cost\,{r^{\frac{\lambda p}{p-1}}}\int_{B_{2r}}|u_{l}^*-g|^pdx
+c\,\frac{\cost}{{r^{\frac{2\lambda p}{2-p}}}}\Big(\int_{B_{2r}}|V_\mu(e(u_{l}^*))|^2dx
+\mu^{\sfrac p2}r^n\Big).
\end{multline*}
Finally, since $u_l\to u$ strongly the sequence $u_l^*+u-u_l$ is also a minimizing sequence for 
$\mathscr{F}_\infty$, and by strict convexity it converges strongly to the unique minimizer $u$. 

We deduce that
$u_l^*\to u$ strongly in $W^{1,p}(B_{2r};\R^n)$ and in the limit as 
$l\uparrow\infty$ we conclude the proof of \eqref{e:CaccioppoliVmusub}. 
Eventually, \eqref{e:eureg} follows by H\"older's inequality and Lemma~\ref{l:tecnico2}.
\end{proof}

\subsection{Decay Estimates}\label{s:decay}

As a first corollary of Propositions~\ref{p:regVmusuper} and \ref{p:regVmusub} we 
establish a decay property of the $L^2$-norm of $V_\mu(e(u))$ needed to prove the 
density lower bound inequality in \cite{ContiFocardiIurlano17}
in the two dimensional setting. The result shall be improved as a consequence of 
the higher integrability property in the next section (cf. Corollary~\ref{c:decayVmu bis}). 
\begin{proposition}\label{p:decayVmu}
Let $n\geq 2$, $p\in(1,\infty)$, $\cost$ and $\mu\geq 0$.
Let $u\in W^{1,p}(\Omega;\R^n)$ be a local minimizer of $\FF$ defined in \eqref{e:FF}
with $g\in L^{p}(\Omega;\R^n)$ if $p\in (1,2]$ and  $g\in W^{1,p}(\Omega;\R^n)$ if 
$p\in (2,\infty)$.

Then, for all $\gamma\in(0,2)$ there is a constant {$c=c(\gamma,p,n,\cost)>0$} such that if 
$B_{R_0}(x_0)~\subset\hskip-0.125cm\subset\Omega$, then for all $\rho<R\leq R_0\leq 1$ 
if $p\geq 2$ it holds
\begin{equation}\label{edecayVmu super}
 \int_{B_{\rho}(x_0)}|V_\mu(e(u))|^2dx\leq c\rho^\gamma\Big(\frac{1}{R^\gamma} 
 \int_{B_{R}(x_0)}|V_\mu(e(u))|^2dx+c\cost \|u-g\|^p_{W^{1,p}(\Omega;\R^n)}\Big),
\end{equation}
and if $p\in(1,2)$ it holds
\begin{equation}\label{edecayVmu sub}
 \int_{B_{\rho}(x_0)}|V_\mu(e(u))|^2dx\leq 
c\rho^\gamma\Big(\frac{1}{R^\gamma} 
 \int_{B_{R}(x_0)}|V_\mu(e(u))|^2dx+c\cost\|u-g\|^p_{L^p(\Omega;\R^n)}+c\cost\mu^{\sfrac p2}\Big), 
 \end{equation}
 \end{proposition}
\begin{proof}
Let $4r\leq R_0$, and $\zeta\in C^\infty_c(B_{2r}(x_0))$ be such that $0\leq\zeta\leq 1$, 
$\zeta|_{B_{r}(x_0)}\equiv 1$, $|\nabla\zeta|\leq 2/r$. Note that $\zeta^2 V_\mu(e(u))\in 
W^{1,2}_0(B_{2r}(x_0),\R^{n\times n}_\sym)$, therefore 
 \begin{multline}\label{e:est1}
  \int_{B_{2r}(x_0)}|\nabla\big(\zeta^2 V_\mu(e(u))\big)|^2dx\\
  \leq 2\int_{B_{2r}(x_0)}\zeta ^4|\nabla V_\mu(e(u))|^2dx
 +8\int_{B_{2r}(x_0)}\zeta^2|\nabla\zeta|^2|V_\mu(e(u))|^2dx\\
  \leq 2\int_{B_{2r}(x_0)}|\nabla V_\mu(e(u))|^2dx
 +\frac{32}{r^2}\int_{B_{2r}(x_0){\setminus B_r(x_0)}}|V_\mu(e(u))|^2dx.
 \end{multline}
If $p\geq 2$ by means of Proposition~\ref{p:regVmusuper} with $\lambda=1$ we further 
estimate as follows
 \begin{multline}\label{e:est1 super}
  \int_{B_{2r}(x_0)}|\nabla\big(\zeta^2 V_\mu(e(u))\big)|^2dx
 \leq\frac{c{(1+\cost)}}{r^2}\int_{B_{4r}(x_0)}|V_\mu(e(u))-(V_\mu(e(u)))_{{B_{4r}(x_0)}}|^2dx\\
 +c\cost r^{{\frac2{p-1}}}\int_{B_{4r}(x_0)}\big(|u-g|^{\tilde p(1)}+|\nabla(u-g)|^p\big)dx
 +\frac{c}{r^2}\int_{B_{2r}(x_0)\setminus B_r(x_0)}|V_\mu(e(u))|^2dx\\
 \leq\frac{c{(1+\cost)}}{r^2}\int_{B_{4r}(x_0)}|V_\mu(e(u))|^2dx
 +c\cost r^{\frac2{p-1}}\|u-g\|_{W^{1,p}(\Omega;\R^n)}^p,
\end{multline}
with $c=c(p,n)>0$. Therefore, in view of Poincar\'e inequality and \eqref{e:est1 super} we get 
for any $\tau\in(0,1)$ and any $q\in (2,2^*)$, with $2^*=2n/(n-2)$ if $n>2$, 
$2^*=\infty$ if $n=2$,
\begin{multline}\label{e:poincare}
 \int_{B_{\tau r(x_0)}}|V_\mu(e(u))|^2dx\leq c\,(\tau\,r)^{n(1-\sfrac 2q)}
 \left(\int_{B_{2r}(x_0)}|\zeta^2V_\mu(e(u))|^qdx\right)^{\sfrac 2q}\\
 \leq c\,
 \tau^{n(1-\sfrac 2q)}r^2\int_{B_{2r}(x_0)}|\nabla\big(\zeta^2 V_\mu(e(u))\big)|^2dx\\
 \leq c\,{(1+\cost)} \tau^{n(1-\sfrac 2q)}\Big(\int_{B_{4r}(x_0)}|V_\mu(e(u))|^2dx
 +c \cost r^{\frac{2p}{p-1}}\|u-g\|_{W^{1,p}(\Omega;\R^n)}^p\Big),
\end{multline}
with $c=c(p,q,n)>0$. 
We choose $q\in (2,2^*)$ such that $n(1-2/q)>
{\frac{2+\gamma}2}$, which 
is the same as $q\in(\frac{4n}{2n-2-\gamma},2^*)$. 
This is possible since $\gamma\in(0,2)$. Then, for sufficiently 
small $\tau$, and for $\theta=\sfrac\tau 4$ 
\begin{equation}\label{e:est4}
 \int_{B_{4\theta r(x_0)}}|V_\mu(e(u))|^2dx
 \leq \theta^{\frac{2+\gamma}2} \int_{B_{4r}(x_0)}|V_\mu(e(u))|^2dx
 +c\cost r^{\gamma}\|u-g\|_{W^{1,p}(\Omega;\R^n)}^p\,\,.
\end{equation}
The decay formula \eqref{edecayVmu super} then follows from \cite[Lemma~7.3]{Giusti}.

Instead, if $p\in(1,2)$ by Proposition~\ref{p:regVmusub} 
{choosing $\lambda=\frac2p-1>0$}
we estimate \eqref{e:est1} as follows  
 \begin{multline*}
  \int_{B_{2r}(x_0)}|\nabla\big(\zeta^2 V_\mu(e(u))\big)|^2dx
 \leq\frac{c{(\cost+1)}}{r^2}\int_{B_{4r}(x_0)}|V_\mu(e(u))-(V_\mu(e(u)))_{{B_{4r}(x_0)}}|^2dx\\
 +c\cost {r^{\frac{2-p}{p-1}}}
 \int_{B_{4r}(x_0)}\big |u-g|^pdx
 +c\cost \mu^{\sfrac p2}r^{{n-2}}
 +c{\frac{\cost+1}{r^2}}
 \int_{B_{4r}(x_0)}|V_\mu(e(u))|^2dx\\
 \leq c{\frac{\cost+1}{r^2}}
 \int_{B_{4r}(x_0)}|V_\mu(e(u))|^2dx
 +c\,\cost\,{r^{\frac{2-p}{p-1}}}\|u-g\|_{L^{p}(\Omega;\R^n)}^p+c\,\cost \mu^{\sfrac p2}r^{{n-2}},
\end{multline*}
with $c=c(n,p)>0$. Then, arguing as to deduce \eqref{e:poincare} we conclude that 
\begin{multline*}
 \int_{B_{\tau r(x_0)}}|V_\mu(e(u))|^2dx \leq  c\,\tau^{n(1-\sfrac 2q)}\cdot\\
 \cdot\Big(\big({\cost+1}\big)\int_{B_{4r}(x_0)}|V_\mu(e(u))|^2dx
 +c\,\cost\,r^{{\frac{p}{p-1}}} \|u-g\|_{L^{p}(\Omega;\R^n)}^p+c\cost \mu^{\sfrac p2}r^{{n}}\Big),
\end{multline*}
with $c=c(p,q,n)>0$. 
By choosing $q\in (2,2^*)$ such that $n(1-2/q)>{\frac{2+\gamma}2}$, for sufficiently small $\tau$, and 
for $\theta=\sfrac \tau 4$
\begin{equation}\label{e:est4 sub}
 \int_{B_{4\theta r(x_0)}}|V_\mu(e(u))|^2dx
 \leq\theta^{\frac{2+\gamma}2} \int_{B_{4r}(x_0)}|V_\mu(e(u))|^2dx
 +c\cost r^{\gamma}\big(\|u-g\|_{L^p(\Omega;\R^n)}^p+\mu^{\sfrac p2}\big)\,.
\end{equation}
The decay formula \eqref{edecayVmu super} then follows from \cite[Lemma~7.3]{Giusti}.
\end{proof}

\section{Partial regularity results}\label{ss:enhanced}

In the quadratic case $p=2$ it is well-known that the minimizer $u$ is $C^2(\Omega;\R^n)$
in any dimension {if $g\in C^1$} (see for instance \cite[Theorem~10.14]{Giusti} or 
\cite[Theorem~5.13, Corollary~5.14]{GiaquintaMartinazzi2012}).

Below we establish $C^{1,\alpha}$ regularity in the two dimensional setting and partial regularity 
in $n$ dimensions together with an estimate on the Hausdorff dimension of the corresponding singular 
set. 
To our knowledge it is a major open problem in elliptic regularity to prove or disprove everywhere 
regularity for elasticity type systems in the nonlinear case if $n\geq 3$ and $p\neq 2$.

\subsection{Higher integrability}\label{ss:HI}

In this subsection we prove the first main ingredient for establishing both 
$C^{1,\alpha}$ regularity if $n=2$ and partial regularity if 
$n\geq 3$ with an estimate of the Hausdorff dimension 
of the singular set:
the higher integrability for the gradient of $V_\mu(e(u))$, $\mu\geq 0$. 
\begin{proposition}\label{p:HI}
Let $n\geq 2$, $p\in(1,\infty)$, $\cost$ and $\mu\geq 0$.
Let $u\in W^{1,p}(\Omega;\R^n)$ be a local minimizer of $\FF$ defined in \eqref{e:FF}
with $g\in L^{s}(\Omega;\R^n)$, $s>p$, if $p\in (1,2]$ and 
$g\in W^{1,p}(\Omega;\R^n)$ if $p\in (2,\infty)$.

Then, $V_\mu(e(u))\in W^{1,q}_\loc(\Omega;\R^{n\times n}_\sym)$ for some $q>2$.
More precisely, there exist $q=q(n,p,\cost)>2$ and $c=c(n,p,\cost)>0$ such that 
if $B_{2r}(x_0)\subset\hskip-0,125cm\subset\Omega$ and $p> 2$ 
\begin{multline}\label{e:HI super}
\Big(\fint_{B_{r}(x_0)}|\nabla \big(V_\mu(e(u))\big)|^qdx\Big)^{\sfrac 1q}
\leq c\Big(\fint_{B_{2r}(x_0)}|\nabla \big(V_\mu(e(u))\big)|^2dx\Big)^{\sfrac12}\\
+c\,\Big(\cost\fint_{B_{2r}(x_0)}\big(|u-g|^{\tilde{p}(\frac{1+\lambda_0}2)}
+|\nabla(u-g)|^{\frac{1+\lambda_0}2 p}\big)^{\frac{q}2}dx\Big)^{\sfrac1q},
\end{multline}
with the exponent {$\tilde p$ and} $\lambda_0\in[\frac 1{p-1},1)$ introduced 
in Section~\ref{ss:caccioppoli}, and if $p\in(1,2]$
\begin{multline}\label{e:HI sub}
\Big(\fint_{B_{r}(x_0)}|\nabla \big(V_\mu(e(u))\big)|^qdx\Big)^{\sfrac 1q}
\leq c\Big(\fint_{B_{2r}(x_0)}|\nabla \big(V_\mu(e(u))\big)|^2dx\Big)^{\sfrac 12}\\
+c\,\Big(\cost\fint_{B_{2r}(x_0)}\big(|u-g|^p+|V_\mu(e(u))|^2+\mu^{\sfrac p2}\big)^{\frac{q}2}dx\Big)^{\sfrac1q}
\end{multline}
{with  $q=q(n,p,\cost,s)>2$ and $c=c(n,p,\cost,s)>0$.}
\end{proposition}
\begin{proof} 
Recalling that $2$ is the Sobolev exponent of $\frac{2n}{n+2}$, we may use the Caccioppoli's 
type estimates \eqref{e:CaccioppoliVmusuper} and \eqref{e:CaccioppoliVmusub}, the former if 
$p> 2$ (with $\lambda=\frac{1+\lambda_0}2$) and the latter if $p\in(1,2]$ {(with $\lambda=0$)}, to deduce by 
Poincar\'e-Wirtinger inequality for some $c=c(n,p)>0$ 
\begin{multline*}
\fint_{B_{r}(x_0)}|\nabla \big(V_\mu(e(u))\big)|^2dx
\leq c(1+\cost)\Big(\fint_{B_{2r}(x_0)}
|\nabla \big(V_\mu(e(u))\big)|^{\frac{2n}{n+2}}dx\Big)^{\frac{n+2}{n}}\\
+c\,\cost r^{\frac2{p-1}}\fint_{B_{2r}(x_0)}
\big(|u-g|^{\tilde p(\frac{1+\lambda_0}2)}+|\nabla(u-g)|^{\frac{1+\lambda_0}2 p}\big)dx,
\end{multline*}
if $p> 2$, and
\begin{multline*}
\fint_{B_{r}(x_0)}|\nabla \big(V_\mu(e(u))\big)|^2dx
\leq c(1+\cost)\Big(\fint_{{B_{2r}(x_0)}}
|\nabla \big(V_\mu(e(u))\big)|^{\frac{2n}{n+2}}dx\Big)^{\frac{n+2}{n}}\\
+c\,\cost\fint_{B_{2r}(x_0)}\big(|u-g|^p+|V_\mu(e(u))|^2+\mu^{\sfrac p2}\big)dx
\end{multline*}
if $p\in(1,2]$. By taking into account that $u\in W^{1,p}(\Omega;\R^n)$, {$\lambda_0\in(\frac1{p-1},1)$}
and $\tilde p(\frac{1+\lambda_0}2)<p^\ast$ and that
$V_\mu(e(u))\in W^{1,2}_\loc(\Omega;\R^{n\times n}_\sym)$ (cf.~Propositions~\ref{p:regVmusuper} 
and \ref{p:regVmusub}), Gehring's lemma with increasing support (see for instance 
\cite[Theorem~6.6]{Giusti}) yields higher integrability together with estimates 
\eqref{e:HI super} and \eqref{e:HI sub}. A covering argument provides the conclusion.
\end{proof}
\begin{remark}
 To apply Gehring's lemma with increasing support in order to deduce higher integrability in 
 case $p>2$ it is instrumental that we may choose $\lambda\in({\lambda_0},1)$ and 
 the corresponding exponent $\tilde p(\lambda)\in(p,{p^\ast})$ 
 in \eqref{e:CaccioppoliVmusuper}
(cf. the definition of $\tilde p(\cdot)$ in \eqref{e:ptilde}).
 \end{remark}
We improve next the decay estimates in 
Proposition \ref{p:decayVmu}. This version is useful to prove 
the density lower bound in \cite{ContiFocardiIurlano17} in the 
three dimensional setting. We do not provide the details since
the proof is the same of Proposition \ref{p:decayVmu} and only 
takes further advantage of Proposition \ref{p:HI}.
\begin{corollary}\label{c:decayVmu bis}
Let $n\geq 3$, $p\in(1,\infty)$, $\cost$ and $\mu\geq 0$.
Let $u\in W^{1,p}(\Omega;\R^n)$ be a local minimizer of $\FF$ defined in \eqref{e:FF}
with $g\in L^s(\Omega;\R^n)$ with $s>p$ if $p\in (1,2]$ and $g\in W^{1,p}(\Omega;\R^n)$ 
if $p\in (2,\infty)$.

Then, there exists $\gamma_0=\gamma_0(n,p,\cost)$, with $\gamma_0>2$, 
such that for all $\gamma\in(0,\gamma_0]$ there is a constant $c=c(\gamma,p,n)>0$ 
such that if $B_{R_0}(x_0)~\subset\hskip-0.125cm\subset\Omega$, then for all $\rho<R\leq R_0\leq 1$ 
\begin{equation*}
 \int_{B_{\rho}(x_0)}|V_\mu(e(u))|^2dx\leq c\rho^\gamma\Big(\frac{1}{R^\gamma} 
 \int_{B_{R}(x_0)}|V_\mu(e(u))|^2dx+c\cost\|u-g\|^p_{W^{1,p}(\Omega;\R^n)}\Big)
\end{equation*}
{if $p>2$, and if $p\in(1,2]$,}
\begin{equation*}
 \int_{B_{\rho}(x_0)}|V_\mu(e(u))|^2dx\leq 
c\rho^\gamma\Big(\frac{1}{R^\gamma} 
 \int_{B_{R}(x_0)}|V_\mu(e(u))|^2dx+c\cost\|u-g\|^p_{L^p(\Omega;\R^n)}+c\cost\mu^{\sfrac p2}\Big)
 \end{equation*}
 {with  $c=c(\gamma,p,n,s)>0$.}
\end{corollary}

\subsection{The $2$-dimensional case}\label{ss:2d}

$C^{1,\alpha}$ regularity in $2$d readily follows from Proposition~\ref{p:HI} 
(see also Remark~\ref{r:furthereg}).
\begin{proposition}\label{p:Calfa}
Let $n=2$, $p\in(1,\infty)$, $\cost$ and $\mu\geq 0$. 
Let $u\in W^{1,p}(\Omega;\R^2)$ be a local minimizer of $\FF$ defined in \eqref{e:FF}
with {$g\in L^{s}(\Omega;\R^2)$, $s>p$,} if $p\in (1,2]$
and $g\in W^{1,p}(\Omega;\R^2)$ if $p\in (2,\infty)$.

Then, $u\in C^{1,\alpha}_\loc(\Omega;\R^2)$ for all $\alpha\in(0,1)$ if $1<p<2$ and $\mu>0$ or if $p\geq 2$, and for 
some $\alpha(p)\in(0,1)$ if $1<p<2$ and $\mu=0$.
\end{proposition}
\begin{proof}
We recall that $V_\mu(e(u))\in W^{1,q}_\loc(\Omega;\R^{2\times 2}_\sym)$ 
for some $q>2$ in view of Proposition~\ref{p:HI}. Therefore, by Morrey's theorem 
$V_\mu(e(u))\in C^{0,1-\frac{2}{q}}_\loc(\Omega;\R^{2\times 2}_\sym)$. 

Furthermore, being $V_\mu$ an homeomorphism with inverse of class 
$C^1(\R^{2\times 2};\R^{2\times 2})$ if $p\in(1,2]$ and $\mu\geq 0$ or if $p>2$ and $\mu>0$, and of class 
$C^{0,\frac2{p}}_\loc(\R^{2\times 2};\R^{2\times 2})$ if $p>2$ and $\mu=0$, we conclude by Korn's inequality that 
$u\in C^{1,\alpha_p}_\loc(\Omega;\R^2)$ for some $\alpha_p=\alpha(p)\in(0,1)$.

To conclude the claimed $C^{1,\alpha}$ regularity for all $\alpha\in(0,1)$, 
we recall first that $u\in W^{2,p\wedge 2}_\loc(\Omega;\R^2)$ (cf. Propositions~\ref{p:regVmusuper} 
and \ref{p:regVmusub}). Actually, in the $2$-dimensional setting $u\in W^{2,2}_\loc(\Omega;\R^2)$ 
in case $p\in(1,2)$, as well. Indeed, $|e(u)|\in L^{\infty}_\loc(\Omega)$ by Corollary~\ref{c:decayVmu bis}, 
therefore we conclude at once from Lemma~\ref{l:tecnico2} (cf. the argument leading to \eqref{e:eureg}).
Hence, since $f_\mu\in C^2(\R^{2\times 2}_\sym)$ if $1<p<2$ and $\mu>0$ or if $p\geq 2$, 
one can differentiate \eqref{e:ELpde} and deduce 
that the weak gradient of $e(u)$ is a weak solution to a linear 
uniformly elliptic system with continuous coefficients. Schauder's 
theory provides the conclusion (cf. \cite[Theorem~5.6 and 5.15]{GiaquintaMartinazzi2012}).
\end{proof}

\begin{remark}\label{r:furthereg}
Actually, $u\in C^k(\Omega;\R^2)$ if ${g}\in C^k(\R^{2\times 2}_\sym)$ and $\mu>0$ bootstrapping 
the previous argument. 
\end{remark}

\subsection{Partial regularity in the non-degenerate autonomous case}\label{ss:partial autonomous}

In this section we deal with the non-degenerate autonomous case, corresponding to 
$\mu>0$ and $\cost=0$, 
by following the so called indirect methods for proving partial regularity (see \cite{Giaquinta83}). 
Therefore, the other main ingredient besides higher integrability of the gradient, is the following 
excess decay lemma.  We introduce the notation
 \begin{equation}\label{e:excess}
  \Exc_v(x,r):=\fint_{B_r(x)}\left|V_\mu\Big(e(v)-\big(e(v)\big)_{B_r(x)}\Big)\right|^2dy
  \end{equation}
for the excess of any $v\in W^{1,p}(\Omega;\R^n)$. Recall that $\big(e(v)\big)_{B_r(x)}=\fint_{B_r(x)}e(v)dy$. 

Technical tools exploited in the proof of the excess decay are postponed to the 
Appendix~\ref{a:technical}. For a linearization argument there, the assumption 
$\mu>0$ is crucial (cf. Theorem~\ref{t:Gammasup}).
\begin{proposition}\label{p:excessdecay}
Let $n\geq 2$, $p\in(1,\infty)$ and $\mu> 0$. Let $u\in W^{1,p}(\Omega;\R^n)$ be a 
local minimizer of $\Fz$ defined in \eqref{e:FF}.

Then, for every $L>0$ there exists $C=C(L)>0$ such that for every $\tau\in(0,\sfrac 14)$ 
there exists $\varepsilon=\varepsilon(\tau,L)>0$ such that if $B_r(x)\subseteq\Omega$, 
\[
{\left|\big(e(u)\big)_{B_r(x)}\right|}\leq L\qquad\text{and}\qquad
\Exc_u(x,r)\leq\varepsilon\,,
\]
then 
\begin{equation}\label{e:excessdecay}
\Exc_u(x,\tau\,r)\leq C\,\tau^2 \Exc_u(x,r).
\end{equation}
\end{proposition}
\begin{proof}
 Suppose by contradiction that there is $L>0$ such that for all constants $C>0$
 we can find $\tau\in(0,{\sfrac 14})$ for which there {exist} 
 $B_{r_h}(x_h)\subset \Omega$ such that
 \[
  {\left|\big(e(u)\big)_{B_{r_h}(x_h)}\right|}\leq L,\quad \Exc_u(x_h,r_h)=\lambda_h^2\downarrow 0,
 \]
and
\begin{equation}\label{e:contra1}
 \Exc_u(x_h,\tau\,r_h)> C\,\tau^2 \Exc_u(x_h,r_h).
\end{equation}
We shall conveniently fix the value of $C$ at the end of the proof to reach a contradiction.

Consider the field $u_h:B_1\to\R^n$ defined by
\[
 u_h(y):=\frac{1}{\lambda_h r_h}\Big(u(x_h+r_h y)-{(u)_{B_{r_h}(x_h)}}
 -r_h 
 {\big(\nabla u\big)_{B_{r_h}(x_h)}}\cdot y\Big),
\]
and set $\mathbb{A}_h:=
{\big(e(u)\big)_{B_{r_h}(x_h)}}$. Then, up to a subsequence we may assume 
that $\mathbb{A}_h\to\mathbb{A}_\infty$ and 
\begin{equation}\label{e:Vmucptness}
 \fint_{B_1}|V_\mu(\lambda_he(u_h))|^2dx=\fint_{B_{r_h}(x_h)}|V_\mu(e(u)-\mathbb{A}_h)|^2dx
 =\Exc_u(x_h,r_h)=\lambda_h^2.
\end{equation}
Being $u$ a local minimizer {of $\Fz$} defined in \eqref{e:FF},  
$u_h$ is in turn a local minimizer of 
\[
\FFh(v)=\int_{B_1}\Fh(e(v))dx, 
\]
with integrand
\[
\Fh(\xi):=\lambda_h^{-2}\big(f_\mu(\mathbb{A}_h+\lambda_h\xi)-f_\mu(\mathbb{A}_h)
-\lambda_h\langle\nabla f_\mu(\mathbb{A}_h),\xi\rangle\big).
\]
Note that $\FFh(u_h)\leq c\,\calL^n(B_1)$ by (iii) in Lemma~\ref{l:Fhsup} and \eqref{e:Vmucptness}, 
thus by Theorem~\ref{t:Gammasup}, $(u_h)_h$ converges weakly to some function $u_\infty\in W^{1,2}(B_1,\R^n)$ 
in $W^{1,p\wedge 2}(B_1,\R^n)$, and actually, by Corollary~\ref{c:strongconv} 
we have for all $r\in(0,1)$
\begin{equation}\label{e:cptness}
  \lim_{h\uparrow\infty}\int_{B_r}\lambda_h^{-2}|V_\mu(\lambda_h e(u_h-u_\infty){)}|^2dx=0.
\end{equation}
Therefore, item (iii) in Lemma~\ref{l:Vmu} and a scaling argument give
for some constant $c=c(p)>0$ 
\begin{align*}
\lambda_h^{-2} &\Exc_u(x_h,\tau r_h)
=\lambda_h^{-2}\fint_{B_{\tau}}\left|V_\mu\Big(\lambda_h(e(u_h)-\big(e(u_h)\big)_{B_\tau}{)}\Big)\right|^2dx\\
\leq &c\,\lambda_h^{-2}\fint_{B_{\tau}}\left|V_\mu\Big(\lambda_h e(u_h-u_\infty)\Big)\right|^2dx
+c\,\lambda_h^{-2}\fint_{B_{\tau}}\left|V_\mu\Big(\lambda_h(e(u_\infty)-\big(e(u_\infty)\big)_{B_\tau})\Big)\right|^2dx\\
&+c\,\lambda_h^{-2}\fint_{B_{\tau}}
\left|V_\mu\Big(\lambda_h\big(\big(e(u_h)\big)_{B_\tau}-\big(e(u_\infty)\big)_{B_\tau}\big)\Big)\right|^2dx.
\end{align*}
The very definition of $V_\mu$, item (v) in Lemma~\ref{l:Vmu} and \eqref{e:cptness} 
yield 
\[
 \limsup_{h\uparrow\infty}\lambda_h^{-2} \Exc_u(x_h,\tau r_h)\leq c\mu^{\sfrac p2-1}
 \fint_{B_{\tau}}\left|e(u_\infty)-\big(e(u_\infty)\big)_{B_\tau}\right|^2dx.
\]
In particular, 
\[
 \limsup_{h\uparrow\infty}\lambda_h^{-2} \Exc_u(x_h,\tau r_h)\leq \widetilde{C}\tau^2\,,
\]
as $u_\infty$ is the solution of a linear elliptic system (cf. Corollary~\ref{c:pde}). 
Thus, by taking the constant $C>\widetilde{C}$, we reach a contradiction to \eqref{e:contra1}.
\end{proof}

We are finally ready to establish partial regularity and an estimate on the Hausdorff 
dimension of the singular set in the non-degenerate autonomous case. The degenerate case,
namely $\mu=0$, corresponding to the symmetrized $p$-laplacian, $p\neq 2,$ is not included 
in our results. 
The non-autonomous case will be treated next via a perturbation argument.
We recall that in case $p=2$ the solutions are actually smooth.

Before proceeding with the proof, we introduce some notation: 
for $v\in W^{1,p}(\Omega;\R^n)$ let 
\begin{equation}\label{e:sigma1}
\Sigma_v^{(1)}:=\left\{x\in\Omega:\,\liminf_{r\downarrow 0}
\fint_{B_r(x)}\left|V_\mu(e(v(y)))-\big(V_\mu(e(v))\big)_{B_r(x)}\right|^2dy>0\right\},
\end{equation}
and
\begin{equation}\label{e:sigma2}
\Sigma_v^{(2)}:=\left\{x\in\Omega:\,\limsup_{r\downarrow 0}\left|\big(V_\mu(e(v))\big)_{B_r(x)}\right|
=\infty\right\}.
\end{equation}
\begin{theorem}\label{t:partial autonomous}
Let $n\geq 3$, $p\in(1,\infty)$ and $\mu>0$. 
Let $u\in W^{1,p}(\Omega;\R^n)$ be a local minimizer of $\Fz$ defined in \eqref{e:FF}.

Then, there exists an open set $\Omega_u\subseteq\Omega$ such that 
$u\in C^{1,\alpha}_{loc}(\Omega_u;\R^n)$ for all $\alpha\in(0,1)$. 
Moreover, 
\[
 \dim_{\mathcal{H}}(\Omega\setminus\Omega_u)\leq (n-q)\vee 0,
\]
where $q>2$ is the exponent in Proposition~\ref{p:HI}.
\end{theorem}
\begin{proof} We shall show in what follows that under the standing assumptions the singular and 
regular sets are given respectively by 
\begin{equation}\label{e:singular regular sets autonomous}
\Sigma_u:=\Sigma_u^{(1)}\cup\Sigma_u^{(2)},\qquad \Omega_u:=\Omega\setminus\Sigma_u.
\end{equation}

By the higher integrability property established in Proposition~\ref{p:HI}, 
we know that $V_\mu(e(u))\in W^{1,q}_{loc}(\Omega;\R^{n\times n}_\sym)$ for some $q>2$. 
Therefore, $\Sigma_u=\emptyset$ if $q>n$ by Morrey's theorem.
Otherwise, if $B_{r}(x_0)\subseteq\Omega$, by Poincar\`e's inequality for all 
$r\in(0,\mathrm{dist}(x_0,\partial\Omega))$
\[
 \fint_{B_r(x_0)}\left|V_\mu(e(u))-\big(V_\mu(e(u))\big)_{B_{r}(x_0)}\right|^2dx\leq 
 c\Big(r^{q-n}\int_{B_r(x_0)}|\nabla V_\mu(e(u))|^qdx\Big)^{\sfrac 2q}.
\]
Therefore, $\mathcal{H}^{n-q}\big(\Sigma_u^{(1)})=0$ by 
standard density estimates (cf. \cite[Proposition~2.7]{Giusti} or \cite[Theorem~2.56]{AmbrosioFuscoPallara}), 
 and $\dim_{\mathcal{H}}\Sigma_u^{(2)}\leq n-q$ by 
standard properties of Sobolev functions (cf. \cite[Theorem~3.22]{Giusti}). 
In conclusion, $\dim_{\mathcal{H}}\Sigma_u\leq n-q$.

Let us prove that $\Omega_u$ is open and that $u\in C^{1,\alpha}(\Omega_u;\R^n)$ for all 
$\alpha\in(0,1)$. Let $x_0\in\Omega_u$. 
First note that {$\sup_{r}|\big(V_\mu(e(u))\big)_{B_r(x_0)}|<\infty$ being 
$x_0\in\Omega\setminus\Sigma_u^{(2)}$. Additionally, since
\begin{align*}
 \fint_{B_\rho(x_0)}&|V_\mu(e(u))|^2dy
 \leq 
 c(p) \fint_{B_\rho(x_0)}\left|V_\mu(e(u))-\big(V_\mu(e(u))\big)_{B_\rho(x)}\right|^2dy+
 c(p) \left|\big(V_\mu(e(u))\big)_{B_\rho(x)}\right|^2,
 \end{align*}
being $x_0\in\Omega\setminus\Sigma_u^{(1)}$ we conclude  
 \begin{align}\label{e:equivalence0}
\liminf_{\rho\downarrow 0} \fint_{B_\rho(x_0)}&|V_\mu(e(u))|^2dy<\infty.
 \end{align}
The last inequality
and item (v) in Lemma~\ref{l:Vmu} yield for some $L>0$ 
\[
\liminf_{\rho\downarrow 0}\left|\big(e(u)\big)_{B_\rho(x_0)}\right|<L.
\]
In view of this, Lemma~\ref{l:Vmu} (item (i) if $p\geq 2$ and item (ii) if $p\in(1,2)$, 
respectively) and Lemma~\ref{l:DieningKaplicky} yield that $\liminf_{\rho\downarrow 0} \Exc_u(x_0,\rho)=0$.
Therefore, for all $\eta>0$, $x_0$ belongs to the set 
\[
\Omega_u^{L,\eta}:=\left\{x\in\Omega:\,
\left|\big(e(u)\big)_{B_r(x)}\right|<L,\quad
\Exc_u(x,r)<\eta
\quad\text{for some $r\in(0,\textrm{dist}(x,\partial\Omega))$}\right\}.
\]
In particular, $\Omega_u{\subseteq}\cup_{L\in\N}
\Omega_u^{L,\eta(L)}$, for every $\eta(L)>0$, and clearly 
each $\Omega_u^{L,\eta(L)}\subseteq\Omega$ is open.
We claim that actually $\Omega_u=\cup_{L\in\N}\Omega_u^{L,\overline{\eta}(L)}$ 
for some $\overline{\eta}(L)=\overline{\eta}(L,n,p,\alpha)$ conveniently defined in what follows.
To this aim we distinguish the super-quadratic and sub-quadratic cases.

We start with the range of exponents $p\geq 2$. 
To check the claim fix any $L\in\N$ and $x_0\in \Omega_u^{L,\eta}$, with corresponding
radius $r$, then we have for all $\tau\in(0,\sfrac 14)$ 
\begin{align}\label{e:mean eu}
\left|\big(e(u)\big)_{B_{\tau r}(x_0)}\right| 
&\leq \left|\big(e(u)\big)_{B_r(x_0)}
\right|+\left|\big(e(u)\big)_{B_{\tau r}(x_0)}-\big(e(u)\big)_{B_r(x_0)}\right|\notag\\
&\leq\left|\big(e(u)\big)_{B_r(x_0)}\right|
+\tau^{-n}\fint_{B_{r}(x_0)}\left|e(u)-\big(e(u)\big)_{B_r(x_0)}\right|dy\notag\\
& \leq \left|\big(e(u)\big)_{B_r(x_0)}\right|+\tau^{-n}(\Exc_u(x_0,r))^{\sfrac 1p},
 \end{align}
where for the last inequality we have used item (iv) of Lemma~\ref{l:Vmu} 
for $p\geq 2$. Moreover, if $\varepsilon(\tau,L)$ is the parameter provided by 
Proposition~\ref{p:excessdecay}, and $0<\eta\leq \varepsilon(\tau,L)$ we infer that 
\begin{equation}\label{e:decay bis}
 \Exc_u(x_0,\tau r)\leq C\tau^2 \Exc_u(x_0,r).
\end{equation}
Having fixed any $\alpha\in(0,1)$ we choose $\tau=\tau(\alpha,L)\in(0,{\sfrac 14})$ 
such that $C\tau^{2\alpha}<1$, with $C=C(L)>0$ the constant in \eqref{e:excessdecay}. 
Therefore, choosing $0<\eta\leq \varepsilon(\tau,L)\wedge \tau^{np}$ we infer from \eqref{e:mean eu} and 
\eqref{e:decay bis}
\[
|(e(u))_{B_{\tau r}(x_0)}|<L+1,\quad
 \Exc_u(x_0,\tau r)<\tau^{2(1-\alpha)}{\Exc_u(x_0,r)}.
\]
The latter is the basic step of an induction argument leading to 
\begin{equation}\label{e:excess iteration}
|(e(u))_{B_{\tau^j r}(x_0)}|<L+1,\quad
 \Exc_u(x_0,\tau^j r)<\tau^{2(1-\alpha)j}{\Exc_u(x_0,r)}
\end{equation}
for all $j\in \N$. Note that from the last two inequalities we conclude readily that $x_0\in\Omega_u$.

Hence we are left with showing \eqref{e:excess iteration}. To this aim fix $j\in\N$,
$j\geq 2$, and assume \eqref{e:excess iteration} true for all $0\leq k\leq j-1$ (as noticed
the first induction step corresponding to $j=1$ has already been established above). Then, by \eqref{e:mean eu} we get
\[
 |(e(u))_{B_{\tau^j r}(x_0)}|\leq |(e(u))_{B_{r}(x_0)}|
 +\tau^{-n}\sum_{k=0}^{j-1}\big(\Exc_u(x_0,\tau^{k}r)\big)^{\sfrac 1p}<L
 +\frac{\tau^{-n}}{1-\tau^{\sfrac2p(1-\alpha)}}{(\Exc_u(x_0,r))^{\sfrac1p}.}
\]
We get the first estimate in \eqref{e:excess iteration} provided 
$0<\eta\leq\varepsilon(\tau,L)\wedge\tau^{np}(1-\tau^{\sfrac2p(1-\alpha)})^p$. 
Finally, to get the second inequality in \eqref{e:excess iteration} it suffices to 
assume in addition $0<\eta<\varepsilon(\tau,L+1)$ and apply Proposition~\ref{p:excessdecay}.
In conclusion, we set 
\begin{equation}\label{eta0_magg2}
\overline{\eta}(L):=\varepsilon(\tau,L+1)\wedge\varepsilon(\tau,L)\wedge\tau^{np}(1-\tau^{\sfrac2p(1-\alpha)})^p
\end{equation}
(recall that $\tau=\tau(\alpha,L)$).

If $p\in(1,2)$ we only highlight the needed changes since the strategy of proof is 
completely analogous. We start off noting that we have for some $c=c(p)>0$ 
(which may vary from line to line)
\begin{align}\label{e:mean eu 0}
&\fint_{B_{\rho}(x_0)}|e(u)-(e(u))_{B_{\rho}(x_0)}|^pdy 
\leq c \fint_{B_{\rho}(x_0)}\frac{\left|V_\mu(e(u))-V_\mu\big((e(u))_{B_{\rho}(x_0)}\big)\right|^p}
{\big(\mu+|e(u)|^2+|(e(u))_{B_{\rho}(x_0)}|^2\big)^{\frac{p(p-2)}{4}}}dy\notag\\
&\leq c \Big(\fint_{B_{\rho}(x_0)}\left|V_\mu(e(u))-V_\mu\big((e(u))_{B_{\rho}(x_0)}\big)\right|^2dy\Big)^{\frac p2}
\Big(\fint_{B_{\rho}(x_0)}\big(\mu+|e(u)|^2+|(e(u))_{B_{\rho}(x_0)}|^2\big)^{\frac{p}{2}}dy\Big)^{1-\frac p2}\notag\\
&\leq c(\Exc_u(x_0,\rho))^{\frac p2}
\Big(\mu^{{\sfrac p2}}+|(e(u))_{B_{\rho}(x_0)}|^p+
\fint_{B_{\rho}(x_0)}|e(u)-(e(u))_{B_{\rho}(x_0)}|^pdy\Big)^{1-\frac p2}\notag\\
&\leq c\big(\mu^{{\sfrac p2}}+|\big(e(u)\big)_{B_{\rho}(x_0)}|^p\big)^{1-\frac p2}(\Exc_u(x_0,\rho))^{\frac p2}
+c\,\Exc_u(x_0,\rho)+
\frac12\fint_{B_{\rho}(x_0)}|e(u)-(e(u))_{B_{\rho}(x_0)}|^pdy,
\end{align}
where we have used Lemma~\ref{l:tecnico2} in the first inequality, H\"older's inequality 
in the second, item~(ii) of Lemma~\ref{l:Vmu} in the third, and Young's inequality in 
the fourth. Therefore, we get
\[
 \fint_{B_{\rho}(x_0)}|e(u)-(e(u))_{B_{\rho}(x_0)}|dy \leq
 c\big(\mu^{{\sfrac p2}}+|(e(u))_{B_{\rho}(x_0)}|^p
 +\Exc_u(x_0,\rho)\big)^{\frac1p-\frac 12}
 (\Exc_u(x_0,\rho))^{\frac 12}
\]
for some constant $c=c(p)>0$. In turn, with fixed $L\in\N$ and $x_0\in \Omega_u^{L,\eta}$, 
for all $\tau\in(0,\sfrac 14)$ we have instead of \eqref{e:mean eu} 
\begin{equation}\label{e:mean eu bis}
 |(e(u))_{B_{\tau r}(x_0)}|\leq|(e(u))_{B_r(x_0)}|
 +c\tau^{-n} \big(\mu^{{\sfrac p2}}+L^p+\Exc_u(x_0,r)\big)^{\frac1p-\frac 12}(\Exc_u(x_0,r))^{\frac 12}.
\end{equation}
Having fixed any $\alpha\in(0,1)$ and choosing $\tau=\tau(\alpha,L)\in(0,{\sfrac 14})$ 
such that $C\tau^{2\alpha}<1$, with $C=C(L)>0$ the constant in \eqref{e:excessdecay}, we can 
establish inductively \eqref{e:excess iteration} provided we choose
\begin{equation}\label{eta0_min2}
\overline{\eta}(L):=\varepsilon(\tau,L+1)\wedge\varepsilon(\tau,L)
\wedge 1\wedge c\tau^{2n}\big(\mu^{\sfrac p2}+L^p+1\big)^{1-\frac 2p}(1-\tau^{1-\alpha})^2,
\end{equation}
with $c=c(p)>0$.

Eventually, for any $p\in(1,\infty)$, 
$V_\mu(e(u))\in C^{0,1-\alpha}_{loc}(\Omega_u^{L,\eta(L)};\R^{n\times n}_\sym)$
for all $\alpha\in(0,1)$ by Campanato's theorem and \eqref{e:excess iteration}.
The conclusion for $e(u)$ then follows at once from 
the fact that $V_\mu$ is an homeomorphism with 
inverse of class $C^1(\R^{n\times n};\R^{n\times n})$ if $p>2$ and $\mu>0$ or if $p\in(1,2]$.
}\end{proof}

\subsection{Partial regularity in the non-degenerate case}\label{ss:partial non-autonomous}

In this section we prove partial regularity in the general non-degenerate case by following
the so called direct methods for regularity. To this aim, with given $\cost,\,\mu>0$ and 
a local minimizer $u$ on $W^{1,p}(\Omega;\R^n)$ of the energy $\FF(\cdot)$, with fixed 
$B_r(x_0)\subseteq\Omega$, we consider {the} minimizer $w$ of the corresponding autonomous functional 
(on the ball $B_r(x_0)$)
\begin{equation}\label{e:F0}
\Fz\big(v,B_r(x_0)\big):=\int_{B_r(x_0)}f_\mu(e(v))dx
\end{equation}
on $u+W^{1,p}_0(B_r(x_0),\R^n)$. 
{This implies $\big(e(u)\big)_{B_r(x_0)}=\big(e(w)\big)_{B_r(x_0)}$.}

\begin{lemma}\label{l:energy comparison}
 Let $n\geq 3$, $p\in(1,\infty)$, $\cost$ and $\mu>0$, $B_r(x_0)\subseteq\Omega$.
 Let $u$ be a local minimizer of $\FF$ in \eqref{e:FF} and $w$ be defined as above.
 
 Then, there exists a constant $c=c(n,p)>0$ such that for all symmetric matrices $\xi\in\R^{n\times n}_\sym$
 \begin{equation}\label{e:energy comparison1}
 \int_{B_r(x_0)}|V_\mu(e(w))-V_\mu(\xi)|^2dx\leq c\int_{B_r(x_0)}|V_\mu(e(u))-V_\mu(\xi)|^2dx,
 \end{equation}
 and
 \begin{equation}\label{e:energy comparison2}
 \int_{B_r(x_0)}|V_\mu(e(u))-V_\mu(e(w))|^2dx\leq c\big(\Fz(u,B_r(x_0))-\Fz(w,B_r(x_0))\big).
 \end{equation}
Moreover, we have
\begin{equation}\label{e:excessdecay5}
  \Exc_w(x_0,r)\leq c_0\,\Exc_u(x_0,r)
 \end{equation}
for some constant $c_0=c_0\big(n,p,\mu, M)>0$, provided that $|(e(u))_{B_r(x_0)}|\leq M$.
\end{lemma}

\begin{proof}
Note that for all symmetric matrices $\xi,\,\eta \in\R^{n\times n}_\sym$
\[
 f_\mu(\eta)-f_\mu(\xi)-\langle\nabla f_\mu(\xi),\eta-\xi\rangle
 =\int_0^1\langle\nabla^2f_\mu(\xi+t(\eta-\xi))(\eta-\xi),\eta-\xi)\rangle(1-t)dt.
\]
Therefore, from \eqref{e:bdhessf} and Lemmata~\ref{l:tecnico}, \ref{l:tecnico2} we infer 
for some constant $c=c(p)>0$
\begin{equation}\label{e:linear}
c^{-1}|V_\mu(\eta)-V_\mu(\xi)|^2\leq 
f_\mu(\eta)-f_\mu(\xi)-\langle\nabla f_\mu(\xi),\eta-\xi\rangle\leq c|V_\mu(\eta)-V_\mu(\xi)|^2.
\end{equation}
Since for all $\varphi\in W^{1,p}_0(B_r(x_0),\R^n)$
\[
 \int_{B_r(x_0)}\langle\nabla f_\mu(\xi),e(\varphi)\rangle dx=0,
\]
from the minimality of $w$ for $\Fz(\cdot,B_r(x_0))$ and since $u-w\in W^{1,p}_0(\Omega;\R^n)$ 
we get that
\begin{multline*}
 \int_{B_r(x_0)} \big(f_\mu(e(w))-f_\mu(\xi)-\langle\nabla f_\mu(\xi),e(w)-\xi\rangle\big)dx\\ 
 \leq
 \int_{B_r(x_0)} \big(f_\mu(e(u))-f_\mu(\xi)-\langle\nabla f_\mu(\xi),e(u)-\xi\rangle\big)dx,
 \end{multline*}
and \eqref{e:energy comparison1} follows at once from \eqref{e:linear}.

For \eqref{e:energy comparison2} we argue analogously: we use the minimality of $w$ and
the condition $u-w\in W^{1,p}_0(\Omega;\R^n)$, 
to infer for all $\varphi\in W^{1,p}_0(B_r(x_0);\R^n)$
\[
 \int_{B_r(x_0)}\langle\nabla f_\mu(e(w)),e(\varphi)\rangle dx=0.
\]
The conclusion follows at once by \eqref{e:linear}. 

Finally, to prove \eqref{e:excessdecay5} we use Lemma~\ref{l:Vmu} (item (i) if $p\geq 2$, 
item (ii) if $p\in(1,2)$) 
and \eqref{e:energy comparison1} 
with $\xi=(e(u))_{B_r(x_0)}=(e(w))_{B_r(x_0)}$ to conclude that 
\[
  \Exc_w(x_0,r)\leq c\fint_{B_r(x_0)}|V_\mu(e(w))-V_\mu(\xi)|^2dx\leq 
  c\fint_{B_r(x_0)}|V_\mu(e(u))-V_\mu(\xi)|^2dx\leq c\,\Exc_u(x_0,r).
\]
{for some constant $c=c\big(n,p,\mu, |\big(e(u)\big)_{B_r(x_0)}|\big)>0$}.
\end{proof}

We are now ready to extend the result of Section~\ref{ss:partial autonomous} to the 
non-autonomous case. 
Besides the sets $\Sigma_v^{(1)}$ introduced in \eqref{e:sigma1} and 
$\Sigma_v^{(2)}$ in \eqref{e:sigma2}, in the framework under examination 
it is necessary to consider additionally the sets 
 \begin{multline}\label{e:sigma3}
 \Sigma_v^{(3)}:=
  \left\{x\in\Omega:\,{\limsup_{r\downarrow 0}}\fint_{B_r(x)}|v(y)-(v)_{B_r(x)}|^pdy>0\right\}\\
  \cup\left\{x\in\Omega:\,\limsup_{r\downarrow 0}
  |(v)_{B_r(x)}|=\infty\right\},
 \end{multline}
 and 
 \begin{align}\label{e:sigma4}
 \Sigma_v^{(4)}:=\left\{x\in\Omega:\,\limsup_{r\downarrow 0}
  |(\nabla v)_{B_r(x)}|=\infty\right\}
   \end{align}
for all $v\in W^{1,p}(\Omega;\R^n)$. 
Note that $\Sigma_v^{(3)}$ is actually empty for exponents $p>n$. More generally
we shall carefully estimate the Hausdorff dimension of such a set using Sobolev 
embedding and the results in Propositions~\ref{p:regVmusuper} and \ref{p:regVmusub}.
\begin{theorem}\label{t:partial}
Let $n\geq 3$, $p\in(1,\infty)$, $\cost$ and $\mu>0$, 
$g\in W^{1,p}\cap L^\infty(\Omega;\R^n)$ if $p\in(2,\infty)$ and 
$g\in L^\infty(\Omega;\R^n)$ if $p\in(1,2]$. 
Let $u$ be a local minimizer on $W^{1,p}(\Omega;\R^n)$ of $\FF$ in \eqref{e:FF}.

Then, there exists an open set $\Omega_u\subseteq\Omega$ such that $u\in C^{1,\beta}_\loc(\Omega_u;\R^n)$ 
for all $\beta\in(0,\sfrac12)$. Moreover,  
\[
 \dim_{\mathcal{H}}(\Omega\setminus\Omega_u)\leq (n-\widetilde{q})\vee 0,
\]
where $\widetilde{q}:=q\wedge p^\ast\wedge2^\ast$, $q>2$ being the exponent in 
Proposition~\ref{p:HI}.
\end{theorem}
\begin{proof}[Proof of Theorem~\ref{t:partial}]
In the current setting the singular and regular sets are defined respectively by 
$\Sigma_u:=\Sigma_u^{(1)}\cup \Sigma_u^{(2)}\cup\Sigma_u^{(3)}\cup \Sigma_u^{(4)}$ and 
$\Omega_u:=\Omega\setminus\Sigma_u$. 

For the details of the estimation of the Hausdorff measures of the sets $\Sigma_u^{(i)}$'s, 
$i\in\{1,2\}$, we refer to the discussion in Theorem~\ref{t:partial autonomous}. Here we simply 
recall that by taking into account that $V_\mu(e(u))\in W^{1,q}_{loc}(\Omega;\R^{n\times n}_\sym)$
for some $q>2$ (cf. Proposition~\ref{p:HI}), we get $\dim_\calH(\Sigma_u^{(1)}\cup \Sigma_u^{(2)})
\leq (n-q)\vee 0$.
Instead, for what concerns $\Sigma_u^{(i)}$'s, $i\in\{3,4\}$, as
$u\in W^{2,p\wedge 2}(\Omega;\R^n)$ (see Propositions~\ref{p:regVmusuper} and \ref{p:regVmusub}), by Sobolev embedding 
$u\in W^{1,p^\ast\wedge 2^\ast}(\Omega;\R^n)$, and then we deduce that 
$\dim_\calH(\Sigma_u^{(3)}\cup\Sigma_u^{(4)})\leq 
\big(n-(p^\ast\wedge 2^\ast)\big)\vee 0$ 
(cf. \cite[Theorem~3.22]{Giusti}). In conclusion, the inequality
$\dim_{\mathcal{H}}(\Omega\setminus\Omega_u)\leq (n-\widetilde{q})\vee 0$ follows.

Next, we claim that the set $\Omega_u$ is open and that 
$u\in C^{1,\beta}_\loc(\Omega_u;\R^n)$ for all $\beta\in(0,\sfrac 12)$.
Let $x_0\in\Omega_u$, then we may find an infinitesimal sequence of radii 
$r_i$ and $M>0$ 
such that 
\begin{equation}\label{e:equivalence}
\limsup_{i\uparrow \infty}\Big(|(\nabla u)_{B_{r_i}(x_0)}|\vee 
\Big(\fint_{B_{r_i}(x_0)}|u|^p\,dy\Big)^{\sfrac 1p}\vee
\Big(\fint_{B_{r_i}(x_0)}|V_\mu(e(u))|^2\,dy\Big)^{\sfrac12}
\Big)<M<\infty,
\end{equation}
and that
\begin{equation}\label{e:excessdecay0}
\liminf_{i\uparrow \infty}\Exc_u(x_0,r_i)=0.
\end{equation}
Given $j\in\N$, $\varepsilon$, $\rho\in(0,1)$, and setting
\begin{multline*}
\Omega_u^{j,\varepsilon,\rho}
:=\Big\{x\in\Omega:\,
|(\nabla u)_{B_r(x)}|\vee
\Big(\fint_{B_r(x)}|u|^pdy\Big)^{\sfrac 1p}\vee
\Big(\fint_{B_r(x)}|V_\mu(e(u))|^2dy\Big)^{\sfrac12}<j,\\
\quad\Exc_u(x,r)<\varepsilon\quad
\text{for some $r\in(0,\rho\wedge\dist(x,\partial\Omega))$}\Big\},
\end{multline*}
we conclude that $x_0\in \Omega_u^{M,\varepsilon,\rho}$ for all choices of $\varepsilon$ and $\rho$ as above. 
Clearly, each $\Omega_u^{j,\varepsilon,\rho}$ is open and 
$\Omega_u\subseteq\cup_{j\in\N}\Omega_u^{j,\varepsilon(j),\rho(j)}$ for every choice of $\varepsilon(j)$, $\rho(j)\in(0,1)$.
The rest of the proof is devoted to establish that 
$\cup_{j\in\N}\Omega_u^{j,\overline{\varepsilon}(j),\overline{\rho}(j)}\subseteq\Omega_u$,
for suitable values of $\overline{\varepsilon}(j)$ and 
$\overline{\rho}(j)$ to be defined in 
what follows, and the claimed regularity for $u$ on $\Omega_u$.

To this aim let $x_0\in\Omega_u^{M,\varepsilon,\rho}$,
for some $M\in\N$, $\varepsilon$, $\rho\in(0,1)$, 
and $r\in(0,\rho\wedge\dist(x,\partial\Omega))$ be a radius corresponding 
to $x_0$ in the definition of $\Omega_u^{M,\varepsilon,\rho}$, i.e.~such that 
\begin{equation}\label{e:supestimate2}
|(\nabla u)_{B_r(x_0)}|\vee
\Big(\fint_{B_r(x_0)}|u|^pdy\Big)^{\sfrac 1p}\vee
\Big(\fint_{B_r(x_0)}|V_\mu(e(u))|^2dy\Big)^{\sfrac12}<M,\quad
\Exc_u(x_0,r)<\varepsilon.
\end{equation} 
Consider the minimizer $w$ of 
$\Fz(\cdot,B_r(x_0))$ on $u+W^{1,p}_0(B_r(x_0);\R^n)$. 
Since $(e(w))_{B_r(x_0)}=(e(u))_{B_r(x_0)}$, we get that
\begin{equation}\label{e:medieew}
|(e(w))_{B_r(x_0)}|=|(e(u))_{B_r(x_0)}|<
|(\nabla u)_{B_r(x_0)}|<M.
 \end{equation}
Moreover, from the proof of Theorem~\ref{t:partial autonomous} 
we know that there exists $\eta(M)>0$ for which if 
 \begin{equation}\label{e:excessdecay2}
  \Exc_w(x_0,r)<\eta(M)\,,
 \end{equation}
  then 
$V_\mu(e(w))\in C^{0,\alpha}(B_r(x_0);\R^{n\times n}_\sym)$ 
for every $\alpha\in(0,1)$, with 
 \begin{equation}\label{e:excessdecay3}
 \Exc_w(x_0,\rho)\leq c_1\Big(\frac{\rho}{r}\Big)^{2\alpha} \Exc_w(x_0,r)
 \end{equation}
for every $\rho\in(0,r)$, with $c_1=c_1(p,\alpha,M)>0$.
Denoting by $c_0=c_0(n,\mu,p,M)>0$ the constant in \eqref{e:excessdecay5} we first choose 
$\varepsilon<\frac{1}{c_0}(\eta(M)\wedge \eta(M+1))$.

Let us first check that for any $\alpha\in(0,1)$ there exist 
a constant $c=c(n,p,\alpha,\mu,M,\|g\|_{L^\infty(\Omega;\R^n)})>0$,
and a radius $\rho_0=\rho_0(n,p)\in(0,1)$ 
satisfying the following: if $r\in(0,\rho_0\wedge\textrm{dist}(x_0,\partial\Omega)))$ we have for all $\tau>0$
\begin{equation}\label{e:decay superquadratic}
 \Exc_u(x_0, \tau r) 
\leq c\,\tau^{2\alpha}\Exc_u(x_0,r)+c\tau^{-n}r
\end{equation}
provided that $\varepsilon<\varepsilon_0=\varepsilon_0(p,\mu,\tau,M)\leq\frac1{c_0}(\eta(M)\wedge \eta(M+1))$ 
(actually $\varepsilon_0:=\frac1{c_0}(\eta(M)\wedge \eta(M+1))$ for $p\geq2$).
Note that from \eqref{e:supestimate2} and from the choice 
$\varepsilon<\varepsilon_0$, inequalities \eqref{e:excessdecay2} and \eqref{e:excessdecay3} hold.

We divide the proof in different steps for ease of readability.
We shall always distinguish the case $p\geq 2$ from $p\in(1,2)$. 
\medskip

\noindent{\bf Step 1.} Proof of \eqref{e:decay superquadratic}
for $p\geq 2$.

If $p\geq 2$, by item (iii) in Lemma~\ref{l:Vmu} we obtain
\begin{align*}
 |V_\mu\big(e(u)-(e(u))_{B_{\tau r}(x_0)}\big)|\le &
 c|V_\mu\big(e(w)-(e(w))_{B_{\tau r}(x_0)}\big)|\\
 &+c |V_\mu(e(u)-e(w))|+c |V_\mu\big((e(u))_{B_{\tau r}(x_0)}-(e(w))_{B_{\tau r}(x_0)}\big)|
\end{align*}
for some $c=c(p)>0$. Thus, by items (i) and (v) in Lemma~\ref{l:Vmu} 
we infer 
\begin{align}\label{e:decay u}
 \Exc_u(x_0,\tau r) &\leq c\, \Exc_{w}(x_0,\tau r)
 +c\,\fint_{B_{\tau r}(x_0)}|V_\mu(e(u-w))|^2dx\notag\\
 &\stackrel{\eqref{e:excessdecay3}}{\leq} c\,c_1\,\tau^{2\alpha}
 \Exc_{w}(x_0,r)
 +c\,\fint_{B_{\tau r}(x_0)}|V_\mu(e(u-w))|^2dx
\notag\\
 &\stackrel{\eqref{e:excessdecay5}}{\leq} c_2\,\tau^{2\alpha}\Exc_u(x_0,r)
 +c_2\,\tau^{-n}
 \fint_{B_r(x_0)}|V_\mu(e(u-w))|^2dx\,,
 \end{align}
with {$c_2=c_2(n,p,\mu,M)>0$}. To estimate the last term we use \eqref{e:energy comparison2} 
and the local minimality of $u$ for $\FF$ to find for some $c_3=c_3(n,p)>0$ that 
\begin{align}\label{e:decay u1}
\int_{B_r(x_0)}&|V_\mu(e(u))-V_\mu(e(w))|^2dx 
\leq c_3\big(\Fz(u,B_r(x_0))-\Fz(w,B_r(x_0))\big)\notag\\
&= c_3\big(\FF(u,B_r(x_0))-\FF(w,B_r(x_0))\big)
+c_3\int_{B_r(x_0)}\big(|w-g|^p-|u-g|^p\big)dx\notag\\
&\leq c_3\int_{B_r(x_0)}\big(|w-g|^p-|u-g|^p\big)dx.
\end{align}
In view of the elementary inequality 
\begin{equation}\label{e:elementary inequality}
\big||z_1|^p-|z_2|^p\big|\leq p(|z_1|^{p-1}+|z_2|^{p-1})|z_1-z_2| 
\end{equation}
for all $z_i\in\R^n$, together with H\"older's, 
Korn's and Young's inequalities, we may proceed as follows 
(in all the $L^p(B_\rho(x_0);\R^k)$ norms in the ensuing formula $k\in\{n,n\times n\}$, 
for the sake of notational simplicity we write only $L^p$):
\begin{align}\label{e:decay u2}
\int_{B_r(x_0)}&\big(|w-g|^p-|u-g|^p\big)dx\notag\\
&\leq
{c(p)}\int_{B_r(x_0)}\big(|w|^{p-1}+|u|^{p-1}+|g|^{p-1}\big)|u-w|dx\notag\\
 &\leq {c(p)}\big(\|w-u\|_{L^p}^{p-1}+\|u\|_{L^p}^{p-1}
 +\|g\|_{L^p}^{p-1}\big) \|u-w\|_{L^p}\notag\\ 
& \leq {c(p)}\, c_{\textrm{Korn}}\,r
\big(\|u\|_{L^p}^{p-1}+\|g\|_{L^p}^{p-1}\big)
 \|e(u-w)\|_{L^p}
 {+c(p) c_{\textrm{Korn}}r^p
 \|e(u-w)\|_{L^p}^p}
 \notag\\ 
& \leq {c(p)\, c_{\textrm{Korn}}\,
\Big(r^{n+1}
\big(\fint_{B_r(x_0)}|u|^pdy+\|g\|_{L^\infty}^{p}\big)
 +
 \,r \|e(u-w)\|_{L^p}^p
 +
 r^p
 \|e(u-w)\|_{L^p}^p}\Big)
 \notag\\
& \leq {c_4\, c_{\textrm{Korn}} \big(r^{n+1}
M^p +r\|e(u-w)\|_{L^p}^p}\big)
\end{align}
where $c_4=c_4(p)>0$, and we assumed without loss of generality that $M\ge \|g\|_{L^\infty}^p$ (recall that $r<1$).
Here $c_{\textrm{Korn}}=c_{\textrm{Korn}}(n,p)>0$  is the best constant in the first Korn's inequality 
on the unit ball. Then from \eqref{e:decay u1} and \eqref{e:decay u2} we find 
\begin{equation}\label{e:stima banana3}
\int_{B_r(x_0)}|V_\mu(e(u))-V_\mu(e(w))|^2dx
\leq c_3{c_4\, c_{\textrm{Korn}} \big(r^{n+1}
M^p +r \|e(u-w)\|_{L^p}^p}\big)\,.
\end{equation}
Next, recalling that $p\geq 2$, by item (iv) in Lemma~\ref{l:Vmu} we have
\begin{equation}\label{e:stima banana2}
\int_{B_r(x_0)}|e(u-w)|^pdx\leq 
\int_{B_r(x_0)}|V_\mu\big(e(u)-e(w)\big)|^2dx\,,
\end{equation}
and moreover by item (i) in the same Lemma~\ref{l:Vmu} 
\[
\int_{B_r(x_0)}|V_\mu\big(e(u)-e(w)\big)|^2dx\leq
c_5\int_{B_r(x_0)}|V_\mu(e(u))-V_\mu(e(w))|^2dx
\]
for some constant $c_5=c_5(p)>0$. Hence, from the latter inequality,
\eqref{e:stima banana3} and \eqref{e:stima banana2},
{if $r\leq \rho_0\leq (2c_3c_4c_5c_{\textrm{Korn}})^{-1}$}, we find
\[
\fint_{B_r(x_0)}|V_\mu\big(e(u)-e(w)\big)|^2dx
\leq \frac2{\omega_n}c_3c_4c_5c_{\textrm{Korn}}M^pr\,. 
\]
In turn, from this and \eqref{e:decay u} we get
\begin{equation}\label{base}
 \Exc_u(x_0,\tau r) 
\leq c_2\,\tau^{2\alpha}\Exc_u(x_0,r)+
\frac2{\omega_n}c_2c_3c_4c_5c_{\textrm{Korn}}
\tau^{-n}M^pr,
\end{equation}
for every $\tau\in(0,1)$, provided $\varepsilon<\frac 1{c_0}(\eta(M)\wedge\eta(M+1))$. 
Inequality \eqref{e:decay superquadratic} then follows at once.
\medskip

\noindent{\bf Step 2.} Proof of \eqref{e:decay superquadratic} for $p\in(1,2)$.
 
First, we have for some constant $c=c(p)$
(cf. \eqref{e:mean eu bis})
\begin{align}\label{e:mean nonautsub}
|(e(u))_{B_{\tau r}(x_0)}|\leq|(e(u))_{B_r(x_0)}|
 +c\tau^{-n} \big(\mu^{{\sfrac p2}}+ |(e(u))_{B_r(x_0)}|^p+\Exc_u(x_0,r)\big)^{\frac1p-\frac 12}
 \big(\Exc_u(x_0,r)\big)^{\frac 12}.
\end{align}
Thus if $\varepsilon<\varepsilon_0:=1\wedge \frac1{c^2}\tau^{2n}(\mu^{\sfrac p2}+M^p+1)^{1-\sfrac2p}\wedge\frac 1{c_0}(\eta(M)\wedge\eta(M+1))$ 
we conclude that 
\[
|(e(u))_{B_{\tau r}(x_0)}| <M+1.
\]
Hence, we may use item (ii) in Lemma~\ref{l:Vmu} to get 
for some constant $c_6=c_6(p,M)>0$
\[
\Exc_u(x_0,{\tau r}) \leq 
c_6\fint_{B_{\tau r}(x_0)}|V_\mu(e(u))-V_\mu((e(u))_{B_{\tau r}(x_0)})|^2dx.
\]
Thus, by item (ii) in Lemma~\ref{l:Vmu} and by Lemma~\ref{l:DieningKaplicky}, denoting by $c_7=c_7(n,p,\mu)>0$ 
the constant there, we infer (recall that since $\varepsilon<\frac1{c_0}\eta(M)$ inequalities \eqref{e:excessdecay2} 
and \eqref{e:excessdecay3} hold true)
\begin{align}\label{e:decay u sub}
 \Exc_u&(x_0,{\tau r}) \leq c_6c_7\fint_{B_{\tau r}(x_0)}|V_\mu(e(u))-\big(V_\mu(e(u))\big)_{B_{\tau r}(x_0)}|^2dx\notag\\
& \leq 3c_6c_7\fint_{B_{\tau r}(x_0)}|V_\mu(e(w))-\big(V_\mu(e(w))\big)_{B_{\tau r}(x_0)}|^2dx
+6c_6c_7\fint_{B_{\tau r}(x_0)}|V_\mu(e(u))-V_\mu(e(w))|^2dx\notag\\
& \leq  3c_6c_7\fint_{B_{\tau r}(x_0)}|V_\mu(e(w))-V_\mu((e(w))_{B_{\tau r}(x_0)}))|^2dx
+6c_6c_7\fint_{B_{\tau r}(x_0)}|V_\mu(e(u))-V_\mu(e(w))|^2dx\notag\\
& \leq  3c(p)c_6c_7\fint_{B_{\tau r}(x_0)}|V_\mu\big(e(w)-(e(w))_{B_{\tau r}(x_0)}\big)|^2dx
+6c_6c_7\fint_{B_{\tau r}(x_0)}|V_\mu(e(u))-V_\mu(e(w))|^2dx\notag\\
& =  3c(p)c_6c_7\Exc_{w}(x_0,{\tau r})+6c_6c_7\fint_{B_{\tau r}(x_0)}|V_\mu(e(u))-V_\mu(e(w))|^2dx\notag\\
&\stackrel{\eqref{e:excessdecay3}}{\leq}  3c(p)c_6c_7c_1\,\tau^{2\alpha}\Exc_{w}(x_0,r)
+6c_6c_7\fint_{B_{\tau r}(x_0)}|V_\mu(e(u))-V_\mu(e(w))|^2dx\notag\\
 &\stackrel{\eqref{e:excessdecay5}}{\leq} c_8\,\tau^{2\alpha}\Exc_u(x_0,r)
+c_8\tau^{-n}\fint_{B_r(x_0)}|V_\mu(e(u))-V_\mu(e(w))|^2dx\,,
 \end{align}
with $c_8=c_8(n,p,\mu,M)>0$.
The last term is bounded arguing exactly as in the superquadratic case: from \eqref{e:decay u1} and
\eqref{e:decay u2} we get \eqref{e:stima banana3} (recalling that $\|g\|_{L^\infty}<M$), i.e.,
\begin{equation}\label{e:stima banana3 sub}
\int_{B_r(x_0)}|V_\mu(e(u))-V_\mu(e(w))|^2dx
\leq c_3c_4c_{\mathrm{Korn}}\big(r^{1+n}M^p+r\|e(u-w)\|_{L^p}^p\big)\,.
\end{equation}
Next, Lemma~\ref{l:tecnico}, H\"older's and Young's inequalities imply for all $p\in(1,2)$
\begin{multline*}
\int_{B_r(x_0)}|e(u-w)|^pdx\\
\leq 
\int_{B_r(x_0)}|V_\mu(e(u))-V_\mu(e(w))|^2dx
+c_9\int_{B_r(x_0)}\big(\mu+|e(u)|^2+|e(w)|^2\big)^{\sfrac p2}dx\,,
\end{multline*}
where $c_9=c_9(p)>0$.
Hence from the latter inequality and \eqref{e:stima banana3 sub} we find for $r\leq \rho_0(n,p):=(2c_3c_4c_{\mathrm{Korn}})^{-1}$
\begin{align}\label{e:stima banana2 sub}
\int_{B_r(x_0)}&|V_\mu\big(e(u)\big)-V_\mu\big(e(w)\big)|^2dx
\notag\\
&\leq 2c_3c_4c_{\mathrm{Korn}}\Big(r^{1+n}M^p 
+c_9r\fint_{B_r(x_0)}\big(\mu+|e(u)|^2+|e(w)|^2\big)^{\sfrac p2}dx\Big)\,.
\end{align}
Being $u$ admissible to test the minimality of $w$, by \eqref{e:growth condition} 
we have for some $c_{10}=c_{10}(\mu)>0$ 
 \begin{equation*}
  c_{10}^{-1}\int_{B_r(x_0)}(|e(w)|^p-1)dx\leq
  \Fz(w,B_r(x_0))\leq\Fz(u,B_r(x_0))
  \leq c_{10}\int_{B_r(x_0)}(|e(u)|^p+1)dx\,.
 \end{equation*}
Since, if $p\in(1,2)$, item (iv) in Lemma~\ref{l:Vmu} yields for some $c_{11}=c_{11}(p)> 0$
\begin{equation*}
\int_{B_r(x_0)}|e(u)|^pdx\leq 
\int_{B_r(x_0)}|V_\mu(e(u))|^2dx+c_{11}\mu^{\sfrac p2}r^n
\end{equation*}
we infer for some $c_{12}=c_{12}(n,p,\mu)>0$
\[
\int_{B_r(x_0)}\big(\mu+|e(u)|^2+|e(w)|^2\big)^{\sfrac p2}dx\leq c_{12}
\Big(r^n+\int_{B_r(x_0)}|V_\mu(e(u))|^2dx\Big)\leq 
c_{12}r^n\big(1+M^2\big)\,.
\]
From this, \eqref{e:decay u sub} and \eqref{e:stima banana2 sub} we get
\[
\Exc_u(x_0,\tau r)\leq c_8\,\tau^{2\alpha}\Exc_u(x_0,r)+
\frac2{\omega_n}c_3c_4c_8c_{\mathrm{Korn}} \tau^{-n}
(M^p+c_9c_{12}(1+M^2))r
\]
provided $r<\rho_0\wedge 1\wedge\dist(x_0,\partial\Omega)$
with $\rho_0(n,p):=(2c_3c_4c_{\mathrm{Korn}})^{-1}$. 
 Inequality \eqref{e:decay superquadratic}
then follows at once.
\bigskip

\noindent Having established \eqref{e:decay superquadratic} for every $p\in(1,\infty)$, 
we proceed as follows. Fix $\alpha>\sfrac 12$, and let 
$0<\delta<\sfrac 12<\alpha$. Choose $\overline{\tau}=\overline{\tau}(\overline{c},\alpha)\in(0,1)$ 
such that $\overline{c}\,\overline{\tau}^{2\alpha-1}\leq 1$, where $\overline{c}$ denotes 
the maximum of the constants in \eqref{e:decay superquadratic} for the bounds $M$ and $M+1$ 
on the means. Thus, we have for all $\tau\in(0,\overline{\tau})$
\begin{equation}\label{e:decay superquadratic bis}
 \Exc_u(x_0, \tau r) 
\leq \tau\Exc_u(x_0,r)+\overline{c}\tau^{-n}r.
\end{equation}
We show next by induction that, with $\overline{\tau}$ as above, it is in fact possible to 
choose, in order, $\overline{\varepsilon}(M)$ and $\overline{\rho}(M)$ (here we highlight only 
the $M$ dependence, for more details see Steps 3 and 4) such that for every $j\in\N$ we have 
\begin{equation}\label{e:mediej}
|(\nabla u)_{B_{\tau^j r}(x_0)}|
\vee\Big(\fint_{B_{\tau^j r}(x_0)}|u|^pdy\Big)^{\sfrac 1p}
\vee\Big(\fint_{B_{\tau^j r}(x_0)}|V_\mu(e(u))|^2dy\Big)^{\sfrac 12}
<M+1,
\end{equation}
and
\begin{equation}\label{e:eccj}
\Exc_u(x_0,\tau^j r)<\tau^j\Exc_u(x_0,r)+\overline{c}\tau^{-n}(\tau^{j-1}r)^{2\delta}\sum_{k=0}^{j-1}\tau^{(1-2\delta) k},
\end{equation}
provided that $\varepsilon\leq\overline{\varepsilon}$, $\tau\leq\overline{\tau}$ and $r<\rho\leq\overline{\rho}$.

Given the the latter inequalities for granted we conclude the proof. Indeed, by \eqref{e:mediej} and \eqref{e:eccj} 
it follows that $x_0\in\Omega_u$, so that $\cup_{j\in\N}\Omega_u^{j,\overline{\varepsilon}(j),\overline{\rho}(j)}\subseteq\Omega_u$. 
Moreover, items (iii) and (v) in Lemma~\ref{l:Vmu}, 
\eqref{e:eccj} and an elementary argument yield that 
\[
\Exc_u(x_0,t)\leq
\frac{c(p)}{\tau^n}\Big(
\frac{\Exc_u(x_0,r)}{\tau\,r}\,t+
\frac{\overline{c}}{\tau^{n+4\delta}(1-\tau^{1-2\delta})}\,t^{2\delta}\Big)
\leq c\, t^{2\delta}\,,
\]
for all $t\in(0,r)$, since $\delta<\sfrac12$ and $r<1$, with $c=c(p,\tau,r,\overline{c},\delta,\overline{\eps})>0$. 
In addition, since by continuity \eqref{e:decay superquadratic bis} 
holds for all points in a ball $B_\lambda(x_0)$ with the same constants 
if $t\in(0,r\wedge \frac12 \textrm{dist}(x_0,\partial\Omega))$, we 
deduce that $u\in C^{1,\beta}(B_\lambda(x_0);\R^n)$ for all $\beta\in(0,\sfrac12)$. 
The result is thus proved.

Hence, to conclude we are left with showing the validity of \eqref{e:mediej} and of \eqref{e:eccj}.
As before we distinguish the superquadratic from the subquadratic case.
\medskip

\noindent{{\bf Step 3.}} Proof of \eqref{e:mediej} and \eqref{e:eccj}.

Let us first prove the case $p\geq 2$. 
We start off deriving some useful estimates on the different means in \eqref{e:mediej}.
Let $j\in\N$, $j\geq 1$, then by Korn's inequality (denoting by $c_K=c_K(n,p)>0$  
the best constant in such an inequality)
\begin{align*}
|(\nabla u)_{B_{\tau^j r}(x_0)}| 
&\leq |(\nabla u)_{B_{\tau^{j-1} r}(x_0)}|
+\big(\tau^{-n}\fint_{B_{\tau^{j-1} r}(x_0)}|\nabla u-(\nabla u)_{B_{\tau^{j-1} r}(x_0)}|^pdx\big)^{\sfrac 1p}\notag\\
&\leq|(\nabla u)_{B_{\tau^{j-1} r}(x_0)}|
+\Big(c_{K}\tau^{-n}
\fint_{B_{\tau^{j-1}r}(x_0)}|e(u)-(e(u))_{B_{\tau^{j-1}r}(x_0)}|^pdx\Big)^{\sfrac 1p}\\
&\leq|(\nabla u)_{B_{\tau^{j-1} r}(x_0)}|+\big(c_{K}\tau^{-n}\Exc(x_0,\tau^{j-1}r)\big)^{\sfrac 1p}.
\end{align*}
Therefore by a simple induction argument we conclude that 
\begin{equation}\label{e:medie nablauj}
|(\nabla u)_{B_{\tau^j r}(x_0)}|\leq |(\nabla u)_{B_r(x_0)}|
+\sum_{k=0}^{j-1}\big(c_{K}\tau^{-n}\Exc(x_0,\tau^kr)\big)^{\sfrac 1p}.
\end{equation}
Analogously, by using Lemma \ref{l:Vmu} (i), we have
\begin{align*}
\Big(\fint_{B_{\tau^j r}(x_0)}&|V_\mu(e(u))|^2dx\Big)^{\sfrac12}
\notag\\
&\leq |V_\mu\big((e(u))_{B_{\tau^{j-1}r}(x_0)}\big)|+
\Big(\tau^{-n}\fint_{B_{\tau^{j-1} r}(x_0)}|V_\mu(e(u))-V_\mu\big((e(u))_{B_{\tau^{j-1}r}(x_0)}\big)|^2dx\Big)^{\sfrac12}\notag\\
&\leq |V_\mu\big((e(u))_{B_{\tau^{j-1}r}(x_0)}\big)|+\big(
c(\mu,K)\tau^{-n}\Exc_u(x_0,\tau^{j-1}r)\big)^{\sfrac12},
\end{align*}
provided that $|(e(u))_{B_{\tau^{j-1}r}(x_0)}|\leq K$.
Therefore, using  Lemma \ref{l:Vmu} (v) by induction
\begin{align}\label{e:medie Vj}
\Big(\fint_{B_{\tau^j r}(x_0)}&|V_\mu(e(u))|^2dx\Big)^{\sfrac 12}
\leq |V_\mu\big((e(u))_{B_{r}(x_0)}\big)|+\sum_{k=0}^{j-1}\big(c(\mu,K)\tau^{-n}\Exc_u(x_0,\tau^kr)\big)^{\sfrac12},
\end{align}
provided that $|(e(u))_{B_{\tau^{k}r}(x_0)}|\leq K$ for 
all $0\leq k\leq j-1$.
Moreover, by Poincar\'e's and by Korn's inequalities
we obtain for a constant $c_{KP}=c_{KP}(n,p)>0$
\begin{align*}
\Big(\fint_{B_{\tau^j r}(x_0)}&|u|^pdx\Big)^{\sfrac 1p}\leq
\Big(\fint_{B_{\tau^j r}(x_0)}|u-(u)_{B_{\tau^{j-1}r}(x_0)}
-(\nabla u)_{B_{\tau^{j-1}r}(x_0)}\cdot(x-x_0)|^pdx\Big)^{\sfrac 1p}\notag\\
&+\Big(\fint_{B_{\tau^j r}(x_0)}|(u)_{B_{\tau^{j-1}r}(x_0)}
+(\nabla u)_{B_{\tau^{j-1}r}(x_0)}\cdot(x-x_0)|^pdx\Big)^{\sfrac 1p}\notag\\
&\leq\tau^{j-1}r\Big(c_{KP}\tau^{-n}\fint_{B_{\tau^{j-1}r}(x_0)}|\nabla u-(\nabla u)_{B_{\tau^{j-1}r}(x_0)}|^pdx\Big)^{\sfrac 1p}
+|(u)_{{\tau^{j-1}r}(x_0)}|
+\tau^j r|(\nabla u)_{B_{\tau^{j-1}r}(x_0)}|\notag\\
&\leq \tau^{j-1}r\Big(c_{KP}^{2}\tau^{-n}
\fint_{B_{\tau^{j-1}r}(x_0)}|e(u)-(e(u))_{B_{\tau^{j-1}r}(x_0)}|^pdx\Big)^{\sfrac 1p}\\
&+\Big(\fint_{B_{\tau^{j-1}r}(x_0)}|u|^pdx\Big)^{\sfrac 1p}
+\tau^j r|(\nabla u)_{B_{\tau^{j-1}r}(x_0)}|\notag\\
&\leq\tau^{j-1}r(c_{KP}^{2}\tau^{-n}\Exc_u(x_0,\tau^{j-1}r))^{\sfrac 1p}+\Big(\fint_{B_{\tau^{j-1}r}(x_0)}|u|^pdx\Big)^{\sfrac 1p}
+\tau^j r|(\nabla u)_{B_{\tau^{j-1}r}(x_0)}|.
\end{align*}
Hence, by induction we conclude that 
\begin{align}\label{e:medie uj}
\Big(\fint_{B_{\tau^j r}(x_0)}|u|^pdx\Big)^{\sfrac 1p}&\leq\Big(\fint_{B_r(x_0)}|u|^pdx\Big)^{\sfrac 1p}
\notag\\
&+r\sum_{k=0}^{j-1}\tau^k\big(c_{KP}^{2}\tau^{-n}\Exc_u(x_0,\tau^kr)\big)^{\sfrac 1p}+r
\sum_{k=0}^{j-1}\tau^{k+1}|(\nabla u)_{B_{\tau^kr}(x_0)}|.
\end{align}
Let us then check the basic induction step $j=1$ for \eqref{e:mediej}. 
Indeed, note that for \eqref{e:eccj} it has been established in Step~2 
(see \eqref{e:decay superquadratic} and \eqref{e:decay superquadratic bis}). 
From \eqref{e:medie nablauj} we find
\begin{align*}
|(\nabla u)_{B_{\tau r}(x_0)}| \leq |(\nabla u)_{B_r(x_0)}|+
(c_K\tau^{-n}\Exc_u(x_0,r))^{\sfrac 1p}\leq M+1
\end{align*}
provided that $\varepsilon<c_K^{-1}\tau^{n}$. 
Moreover, from \eqref{e:medie Vj} we have
\begin{align*}
\Big(\fint_{B_{\tau r}(x_0)}|V_\mu(e(u))|^2dx\Big)^{\sfrac12}
\leq |V_\mu\big((e(u))_{B_{r}(x_0)}\big)|
+\big(c(\mu,M)\tau^{-n}\Exc_u(x_0,r)\big)^{\sfrac 12}<M+1,
\end{align*}
provided that $\varepsilon<c^{-1}(\mu,M)\tau^{n}$.
In addition, from \eqref{e:medie uj}
\begin{align*}
\Big(\fint_{B_{\tau r}(x_0)}|u|^pdx\Big)^{\sfrac 1p}&\leq
\big(c_{KP}^{2}\tau^{-n}\Exc_u(x_0,r)\big)^{\sfrac 1p}
+\Big(\fint_{B_r(x_0)}|u|^pdx\Big)^{\sfrac 1p}
+\tau r|(\nabla u)_{B_r(x_0)}|\notag\\
&\leq \big(c_{KP}^{2}\tau^{-n}\Exc_u(x_0,r)\big)^{\sfrac 1p}+M+\tau rM<M+1,
\end{align*}
by choosing  $\varepsilon<2^{-p}c_{KP}^2\tau^n$ and $r< (2M)^{-1}$.
In conclusion, \eqref{e:mediej} is established for $j=1$ and $\tau<\overline{\tau}(M,\alpha)$, if
$\varepsilon<\varepsilon_1:=
\varepsilon_0\wedge c_K^{-1}\tau^{n}\wedge c^{-1}(\mu,M)\tau^{n} \wedge 2^{-p}c_{KP}^{-2}\tau^n$ and $\rho<\rho_1:=\rho_0\wedge(2M)^{-1}$ 
($\varepsilon_0$ and $\rho_0$ have been defined in Step 1).

Let now $j\in \N$, $j\geq 2$, and assume by induction that \eqref{e:mediej} and \eqref{e:eccj} 
hold for all $0\leq k\leq j-1$. Then for such values of $k$ we have
\begin{equation}\label{e:ecck}
\Exc_u(x_0,\tau^k r)<\tau^k\Exc_u(x_0,r)+\frac{\overline{c}\tau^{-n}}{1-\tau^{1-2\delta}}(\tau^{k-1}r)^{2\delta}
\end{equation}
and then 
\begin{equation}\label{e:ecck sum}
\sum_{k=0}^{j-1}\big(\Exc_u(x_0,\tau^k r)\big)^{\sfrac 1p}<\frac{(\Exc_u(x_0,r))^{\sfrac1p}}{1-\tau^{\sfrac 1p}}
+\Big(\frac{\overline{c}\tau^{-n}}{1-\tau^{1-2\delta}}\Big)^{{\sfrac 1p}}\frac{(\tau^{-1}r)^{\sfrac{2\delta}p}}{1-\tau^{\sfrac{2\delta}p}}.
\end{equation}
Hence, having fixed $\tau\in(0,\overline{\tau}]$, we may choose $\varepsilon_2=\varepsilon_2(\varepsilon_1,p,\tau)<\varepsilon_1$ 
and $\rho_2=\rho_2(\varepsilon_1,\overline{c},p,\delta)<\rho_1$ such that
if $\widetilde{C}:=c_K\vee c(\mu,M)\vee c^2_{KP}\vee 1$ and 
$\rho<\rho_2$, $\varepsilon<\varepsilon_2$ we find
\begin{equation}\label{e:ecck sum bound}
(\widetilde{C}\tau^{-n})^{{\sfrac 1p}}\sum_{k=0}^{j-1}\big(\Exc_u(x_0,\tau^k r)\big)^{{\sfrac 1p}}<\varepsilon_1.
\end{equation}
In particular, the inductive hypothesis on \eqref{e:mediej}, \eqref{e:medie nablauj} and \eqref{e:ecck sum bound} 
yield
\begin{align}\label{e:medie nablauj 2}
|(\nabla u)_{B_{\tau^j r}(x_0)}|\leq
M+(c_K\tau^{-n})^{\sfrac 1p}\sum_{k=0}^{j-1}\big(\Exc_u(x_0,\tau^k r)\big)^{\sfrac 1p}<M+1.
\end{align}
In turn, by the inductive assumption $|(e(u))_{B_{\tau^{k}r}(x_0)}|\leq M+1$ for all $0\leq k\leq j-1$, 
so that thanks to \eqref{e:medie Vj} and \eqref{e:ecck sum bound}, as $\sfrac 1p\wedge\sfrac 12=\sfrac 1p$, we infer
\begin{align}\label{e:medie Vj2}
\Big(\fint_{B_{\tau^j r}(x_0)}|V_\mu(e(u))|^2dx\Big)^{\sfrac 12}
&\leq
M+(c(\mu,M)\tau^{-n})^{\sfrac 12}\sum_{k=0}^{j-1}\big(\Exc_u(x_0,\tau^k r)\big)^{\sfrac 12}<M+1.
\end{align}
Finally, in view of \eqref{e:medie uj} and \eqref{e:ecck sum bound} 
we get
\begin{align}\label{e:medie uj 2}
\Big(\fint_{B_{\tau^j r}(x_0)}|u|^pdx\Big)^{\sfrac 1p}&\leq
M+r\,(c_{KP}^2\tau^{-n})^{\sfrac 1p}\sum_{k=0}^{j-1}\big(\Exc_u(x_0,\tau^k r)\big)^{\sfrac 1p}+\frac{rM}{1-\tau}<M
+r\big(\varepsilon_1+\frac{M}{1-\overline{\tau}}\big).
\end{align}
Thus we have concluded \eqref{e:mediej} for the index $j$ provided that $\varepsilon<\varepsilon_2$
and $\rho<\rho_2\wedge (\varepsilon_1+\frac{M}{1-\overline{\tau}})^{-1}$.

Finally, we prove \eqref{e:eccj} for the index $j$ as follows. 
From \eqref{e:ecck sum} we have $\Exc_u(x_0,\tau^{j-1}r)<\varepsilon$, so that by the inductive hypothesis on the means
it turns out that $x_0\in\Omega_u^{M+1,\varepsilon,\rho}$ with corresponding 
radius $\tau^{j-1}r$. Moreover, the choice 
$\varepsilon<\frac1{c_0}\eta(M+1)$ and the definition of 
$\overline{\tau}$ (cf. the paragraph right before \eqref{e:decay superquadratic bis}) imply that \eqref{e:decay superquadratic bis} itself hold with the radii $\tau^{j-1}r$, $\tau^jr$ in place of 
$r$, $\tau r$ respectively. 
Thus, using the inductive assumption on \eqref{e:eccj} for
$j-1$ we conclude
\begin{align*}
\Exc_u(x_0,\tau^jr)&\leq\tau\Exc_u(x_0,\tau^{j-1}r)
+\overline{c}\tau^{-n}\tau^{j-1}r\\&\leq 
\tau^{j}\Exc_u(x_0,r)+\overline{c}\tau^{-n}
(\tau^{j-1}r)^{2\delta}\sum_{k=1}^{j-1}\tau^{(1-2\delta)k}
+\overline{c}\tau^{-n}\tau^{j-1}r\\&\leq 
\tau^{j}\Exc_u(x_0,r)+\overline{c}\tau^{-n}
(\tau^{j-1}r)^{2\delta}\sum_{k=0}^{j-1}\tau^{(1-2\delta)k},
\end{align*}
since $\delta< \sfrac 12$.
\medskip

The proof of \eqref{e:mediej} and \eqref{e:eccj} in the case 
$p\in(1,2)$ is quite similar. Hence, we will highlight only the main differences. First, arguing as in \eqref{e:mean nonautsub} (cf. \eqref{e:mean eu 0}, \eqref{e:mean eu bis}) and using Korn's inequality we have for some constant $c_K=c_K(n,p)$
	\begin{align}
	|(\nabla u)_{B_{\tau^j r}(x_0)}|&\leq
	|(\nabla u)_{B_{\tau^{j-1}r}(x_0)}|
	\notag\\
	&+c_K\tau^{-n} \Big(\mu^{{\sfrac p2}}+ |(\nabla u)_{B_{\tau^{j-1}r}(x_0)}|^p+\Exc_u(x_0,\tau^{j-1}r)\Big)^{\frac1p-\frac 12}
	(\Exc_u(x_0,\tau^{j-1}r))^{\frac 12}\notag.
	\end{align}
Thus, by induction we infer that
\begin{align}\label{mediee_pmin2}
|(\nabla u)_{B_{\tau^j r}(x_0)}|&\leq
|(\nabla u)_{B_{r}(x_0)}|\notag
\\
&+c_K\tau^{-n} \sum_{k=0}^{j-1}\big(\mu^{{\sfrac p2}}+ 
|(\nabla u)_{B_{\tau^{k}r}(x_0)}|^p+\Exc_u(x_0,\tau^{k}r)\big)^{\frac1p-\frac 12}
(\Exc_u(x_0,\tau^{k}r))^{\frac 12}.
\end{align}
Analogously to the derivation of \eqref{e:medie Vj},
by Lemma \ref{l:Vmu} (v) and (ii) we find
\begin{align}\label{e:medie Vj_pmin2}
\Big(\fint_{B_{\tau^j r}(x_0)}&|V_\mu(e(u))|^2dx\Big)^{\sfrac 12}
\leq |V_\mu\big((e(u))_{B_{r}(x_0)}\big)|+
(c(\mu,M)\tau^{-n})^{\sfrac12}\sum_{k=0}^{j-1}\big(\Exc_u(x_0,\tau^kr)\big)^{\sfrac12}.
\end{align}
Again, by Poincar\'e and Korn's inequalities
we find for a constant $c_{KP}=c_{KP}(n,p)>0$ 
(cf. the derivation of \eqref{e:medie uj} and \eqref{e:mean eu 0})
\begin{align*}
\Big(\fint_{B_{\tau^j r}(x_0)}|u|^pdx\Big)^{\sfrac 1p}
\leq&\tau^{j-1}r
\Big(c_{KP}^{2}\tau^{-n}\fint_{B_{\tau^{j-1}r}(x_0)}|e(u)-(e(u))_{B_{\tau^{j-1}r}(x_0)}|^pdx\Big)^{\sfrac 1p}\\
&+\Big(\fint_{B_{\tau^{j-1}r}(x_0)}|u|^pdx\Big)^{\sfrac 1p}
+\tau^j r|(\nabla u)_{B_{\tau^{j-1}r}(x_0)}|\notag\\
\leq&\tau^{j-1}r\,(c_{KP}^{2}\tau^{-n})^{\sfrac 1p}
\big(\mu^{{\sfrac p2}}+|(e(u))_{B_{\tau^{j-1}r}(x_0)}|^p+\Exc_u(x_0,\tau^{j-1}r)\big)^{\frac1p-\frac 12}
(\Exc_u(x_0,\tau^{j-1}r))^{\frac 12}\notag\\
&+\Big(\fint_{B_{\tau^{j-1}r}(x_0)}|u|^pdx\Big)^{\sfrac 1p}
+\tau^j r|(\nabla u)_{B_{\tau^{j-1}r}(x_0)}|.
\end{align*}
Therefore, by induction we conclude that 
\begin{align}\label{e:uj_pmin2}
&\Big(\fint_{B_{\tau^j r}(x_0)}|u|^pdx\Big)^{\sfrac 1p}\leq\Big(\fint_{B_r(x_0)}|u|^pdx\Big)^{\sfrac 1p}+r
\sum_{k=0}^{j-1}\tau^{k+1}|(\nabla u)_{B_{\tau^kr}(x_0)}|
\notag\\
&+r\,(c_{KP}^{2}\tau^{-n})^{\sfrac 1p}\sum_{k=0}^{j-1}\tau^{k}
\big(\mu^{{\sfrac p2}}+|(e(u))_{B_{\tau^{k}r}(x_0)}|^p+\Exc_u(x_0,\tau^{k}r)\big)^{\frac1p-\frac 12}
(\Exc_u(x_0,\tau^{k}r))^{\frac 12}.
\end{align}
From \eqref{mediee_pmin2}-\eqref{e:uj_pmin2} we easily deduce the basic induction step for \eqref{e:mediej}, provided that we choose 
$\varepsilon<\varepsilon_0\wedge
(c_K^{-1}\tau^{n})^2(\mu^{\sfrac p2}+M^p+1)^{1-\sfrac2p}\wedge c^{-1}(\mu,M)\tau^{n} \wedge 2^{-2}(c_{KP}^{-2}\tau^n)^{\sfrac2p}(\mu^{\sfrac p2}+M^p+1)^{1-\sfrac2p}$ and 
$\rho<\rho_0\wedge(2M)^{-1}$ ($\varepsilon_0$ and $\rho_0$ have been defined in Step 2). 
The general induction step $j\in \N$, $j\geq2$, is now completely similar to the case $p\geq2$.
\end{proof}

\appendix 

\section{Technical results}\label{a:technical}

In this section we collect several technical tools we have used to settle partial regularity 
in the autonomous case.
{We recall that for} sequences of scalars $\lambda_h\downarrow 0$ and of 
matrices $\mathbb{A}_h\to \mathbb{A}$ {we set}
\[
\Fh(\xi):=\lambda_h^{-2}\big(f_\mu(\mathbb{A}_h+\lambda_h\xi)-f_\mu(\mathbb{A}_h)
-\lambda_h\langle\nabla f_\mu(\mathbb{A}_h),\xi\rangle\big).
\]
Let us prove some properties of $\Fh$.
\begin{lemma}\label{l:Fhsup}
Let $p\in(1,\infty)$ and $\mu>0$, then
\begin{itemize}
 \item[(i)] $\Fh\to \Finf$ in {$L^\infty_\loc(\R^{n\times n})$} as $h\uparrow\infty$, where 
 $\Finf(\xi):=\frac12\langle\nabla^2f_\mu(\mathbb{A})\xi,\xi\rangle$;
 \item[(ii)] there exists $\omega:(0,+\infty)\to(0,+\infty)$ non{-}decreasing such that 
 $\omega(t)\downarrow 0$ as $t\downarrow 0$ and for every $\xi\in\R^{n\times n}_\sym$ with $\lambda_h|\xi|\le 1$ one has
 \[
 \Fh(\xi)\geq \Finf(\xi)-{\omega(\lambda_h|\xi|+|\mathbb{A}_h-\mathbb{A}|)|\xi|^2; }
 \]
 \item[(iii)] there exists a constant {$c=c(\mu,M)>1$, with $M\geq\sup_h|\mathbb{A}_h|$, such that for} all $\xi\in\R^{n\times n}_\sym$
 \[
\frac {c^{-1}}{\lambda_h^2}|V_\mu(\lambda_h\xi)|^2\leq  \Fh(\xi)\leq \frac c{\lambda_h^2}|V_\mu(\lambda_h\xi)|^2;
 \]
 \item[(iv)] there exists a constant $c(p,\mu)>0$ such that for all $\xi,\,\eta\in\R^{n\times n}_\sym$
\begin{align*}
 \Fh(\xi)-\Fh(\eta)\geq &
  \frac c{\lambda_h^2} |V_\mu(\mathbb{A}_h+\lambda_h\xi)-V_\mu(\mathbb{A}_h+\lambda_h\eta)|^2\\&+
 \frac 1\lambda_h\langle\nabla f_\mu(\mathbb{A}_h+\lambda_h\eta)-\nabla f_\mu(\mathbb{A}_h),(\xi-\eta)\rangle;
\end{align*}
If, additionally, $\lambda_h|\eta|\le \mu$ then for some constant $c(p,\mu,M)>0$, with $M\geq\sup_h|\mathbb{A}_h|$, we have
\[
 \Fh(\xi)-\Fh(\eta)\geq \frac c{\lambda_h^2}|V_\mu(\lambda_h(\xi-\eta))|^2+
 \frac 1{\lambda_h}\langle\nabla f_\mu(\mathbb{A}_h+\lambda_h\eta)-\nabla f_\mu(\mathbb{A}_h),(\xi-\eta)\rangle.
\]
\end{itemize}
\end{lemma}
\begin{proof}
It suffices to take into account the representation formula 
\begin{equation}\label{e:fhrepr}
 \Fh(\xi)=\int_0^1\langle\nabla^2f_\mu(\mathbb{A}_h+t\lambda_h\xi)\xi,\xi\rangle(1-t)dt
\end{equation}
to establish items (i) and (ii). 

To prove (iii) first we notice the basic inequalities: 
\begin{equation}\label{e:A}
\frac{\mu}{2(\mu+M^2)}(\mu+|s\,\xi|^2)\leq 
\mu+|\mathbb{A}_h+s\,\xi|^2\leq 2\frac{\mu+M{^2}}{\mu}(\mu+|s\,\xi|^2),
\end{equation}
where $M\geq\sup_h|\mathbb{A}_h|$ and $s>0$.
Thus, from \eqref{e:bdhessf}, \eqref{e:fhrepr} and \eqref{e:A} we deduce if $p\in(2,\infty)$ 
\[
 \frac 1c\Big(
 \frac{\mu}{2(\mu+M^2)}\Big)^{\frac p2-1} 
 \lambda_h^{-2}|V_\mu(\lambda_h\xi)|^2 
 \leq \Fh(\xi)\leq 
 c\Big(2\frac{\mu+M^{2}}{\mu}\Big)^{\frac p2-1} \lambda_h^{-2}|V_\mu(\lambda_h\xi)|^2
\]
for some constant $c>0$. 
The inequality on the left hand side follows by arguing as in Lemma~\ref{l:tecnico}. 
Analogously, the case with $p\in(1,2)$ holds with opposite inequalities. Instead, if $p=2$ (iii) is trivial.

To prove (iv) a simple computation yields
 \begin{multline*}
 \Fh(\xi)-\Fh(\eta)= \int_0^1\langle\nabla^2f_\mu(\mathbb{A}_h+t\lambda_h(\xi-\eta){+\lambda_h\eta})(\xi-\eta),(\xi-\eta)\rangle(1-t)dt\\
 +\lambda_h^{-1}\langle\nabla f_\mu(\mathbb{A}_h+\lambda_h\eta)-\nabla f_\mu(\mathbb{A}_h),(\xi-\eta)\rangle.
\end{multline*}
Therefore, the first inequality follows from \eqref{e:bdhessf} and Lemmas~\ref{l:tecnico}, \ref{l:tecnico2}.  
Instead, the second inequality follows by estimating the first term on the right hand side as for (iii).
\end{proof}

Consider $\FFh:L^{p}(B_1;\R^n)\times\mathcal{A}(B_1)\to[0,+\infty]$ defined by
\begin{equation}\label{e:Fh}
 \FFh(u,A)=\int_{A}\Fh(e(u))dx
\end{equation}
if $u\in W^{1,p}(B_1;\R^n)$, and $+\infty$ otherwise. Above, $\mathcal{A}(B_1)$ is the class of all
open subsets of $B_1$. We shall drop the dependence on $B_1$ on the left hand side if $A=B_1$.
Let $u_h$ be a local minimizer of $\FFh$, that is $\FFh(u_h)=\inf_{u_h+W^{1,p}_0(B_1;\R^n)}\FFh$,
and moreover assume that 
\begin{equation}\label{e:Fhuh}
 \fint_{B_1} u_hdx=0,\,\fint_{B_1}\nabla u_hdx=0,\textrm{ and }\quad
 \sup_h\int_{B_1}\lambda_h^{-2}|V_\mu(\lambda_h e(u_h))|^2dx<+\infty.
\end{equation}
In view of \eqref{e:Fhuh} and item (v) in  Lemma~\ref{l:Vmu}, 
it follows from Korn's inequality for $N$-functions in \cite[Lemma~2.9]{DieningKaplicky} applied to $|V_\mu(\cdot)|^2$ 
(cf. item (v) in Lemma \ref{l:Vmu}) that 
\begin{equation}\label{e:Fhuh2}
 \sup_h\int_{B_1}\lambda_h^{-2}|V_\mu(\lambda_h \nabla u_h)|^2dx<+\infty.
\end{equation}
Moreover, by item (v) in Lemma~\ref{l:Vmu} an application of Poincar\'e's inequality 
for $N$-functions (\cite[Theorem~6.5]{DieningRuzickaSchumacher10}) yields
\begin{equation}\label{e:Fhuh3}
 \sup_h\int_{B_1}\lambda_h^{-2}|V_\mu(\lambda_h u_h)|^2dx<+\infty.
\end{equation}
{where with abuse of notation we define $V_\mu:\R^n\to\R^n$ by the same formula used for matrices.}

The ensuing result is instrumental to prove that actually 
$(u_h)_h$ converges to $u_\infty$ strongly in $W^{1,p\wedge 2}_\loc(B_1;\R^n)$.

\begin{theorem}\label{t:Gammasup}
Let 
$\FFinf:L^2(B_1;\R^n)\times\mathcal{A}(B_1)\to[0,\infty]$ be given by 
\begin{equation}\label{e:Finftysup}
 \FFinf(u;A)=\frac12\int_{A}
 \langle\nabla^2f_\mu(\mathbb{A})e(u),e(u)\rangle\,dx
 \end{equation}
if $u\in W^{1,2}(B_1;\R^n)$ and $+\infty$ otherwise. 

If $\FFh$ are the functionals in \eqref{e:Fh} and $(u_h)_h$ is the sequence 
in \eqref{e:Fhuh}, then {after extracting a subsequence} $(u_h)_h$ converges weakly in $W^{1,p\wedge 2}(B_1;\R^n)$ 
to some function $u_\infty\in W^{1,2}(B_1;\R^n)$,
\begin{equation}\label{e:Finftysuplb}
\liminf_{h\uparrow\infty}\FFh(u_h,B_r)\geq\FFinf(u_\infty,B_r)\quad\textrm{for all $r\in(0,1]$},
 \end{equation}
and
\begin{equation}\label{e:Finftysupub}
\limsup_{h\uparrow\infty}\FFh(u_h,B_r)\leq\FFinf(u_\infty,B_r)\quad\textrm{for $\calL^1$ a.e. $r\in(0,1)$.}
\end{equation}
\end{theorem}
\begin{proof}
First, we notice that, up to the extraction of a subsequence not relabeled for convenience, there exists 
$u_\infty\in W^{1,p\wedge 2}(B_1;\R^n)$ such that $(u_h)_h$ converges weakly in $W^{1,p\wedge 2}(B_1;\R^n)$ 
to $u_\infty$ with
\begin{equation}\label{e:Fhuh4}
 \fint_{B_1} u_\infty dx=\fint_{B_1}\nabla u_\infty dx=0.
\end{equation}
Indeed, for $p\geq 2$ from \eqref{e:Fhuh} we deduce that 
$\sup_h\|e(u_h)\|_{L^2(B_1;\R^n)}\leq c\,\mu^{1-\sfrac p2}$, thus the Korn's inequality, 
Poincar\`e inequality, and the fact that $u_h$ and its gradient have null mean value (cf. \eqref{e:Fhuh}) 
provide the conclusion.
{We observe that \eqref{e:Fhuh} also implies that $e(\lambda_h^{1-\sfrac2p}u_h)$ is bounded in $L^p(B_1;\R^{n\times n})$; hence, possibly after extracting a further subsequence, we can assume that
$\lambda_h^{1-\sfrac2p}u_h$ converges weakly in $W^{1,p}(B_1;\R^n)$ and pointwise almost everywhere to some function $z$. Since $u_h$ converges pointwise to $u_\infty$, we deduce that $z=0$, and in particular 
$\lambda_h^{1-\sfrac2p}u_h\to0$ in $L^p(B_1;\R^n)$}.

Instead, in case $p\in(1,2)$, we first note that as $\lambda_h\in(0,1)$  we have
\[
\int_{B_1}|V_\mu({e( u_h)})|^2dx\leq\int_{B_1}\lambda_h^{-2}|V_\mu(\lambda_h {e( u_h)})|^2dx\,,
\]
{so that (\ref{e:Fhuh}) implies
$\sup_h\|{e( u_h)}\|_{L^p(B_1;\R^n)}<+\infty$.} 
Arguing as in the previous case we establish the claimed result.

Next, we prove separately \eqref{e:Finftysuplb} and \eqref{e:Finftysupub} in the super-quadratic and in 
the sub-quadratic case.
\smallskip

\noindent\emph{The super-quadratic case {$p>2$}.} We first prove the lower bound inequality for $r\in(0,1]$. 
 Set $E_h:=\{\lambda_h^{\sfrac 12}|e(u_h)|\geq 1\}$, then $\calL^n(E_h)\downarrow 0$ and 
 $e(u_h)\chi_{E_h^c}{\rightharpoonup} e(u_\infty)$ weakly in $L^2(B_1;\R^{n\times n})$ as $h\uparrow\infty$. 
 Therefore, by (ii) in Lemma~\ref{l:Fhsup} 
 \begin{multline*}
  \FFh(u_h,B_r)\geq\int_{B_r\cap E_h^c}\Fh(e(u_h))dx\geq
   \int_{B_r\cap E_h^c}\big(\Finf(e(u_h))
  -\omega(\lambda_h^{\sfrac 12}{+|\mathbb{A}_h-\mathbb{A}|})|e(u_h)|^2\big)dx\\
    \geq\int_{B_r}\Finf(e(u_h)\chi_{E_h^c})dx
  -\omega(\lambda_h^{\sfrac 12}{+|\mathbb{A}_h-\mathbb{A}|})\int_{B_1}|e(u_h)|^2dx,
 \end{multline*}
and thus by $L^2$ weak lower semicontinuity of $\FFinf(\cdot,B_r)$ we conclude \eqref{e:Finftysuplb}.

To prove the upper bound for all but countably many $r\in(0,1)$, we note that by Urysohn's property it 
suffices to show that for every subsequence $h_k\uparrow\infty$ we can extract $h_{k_j}\uparrow\infty$ such that
\[
\limsup_{j\uparrow\infty}\mathscr{F}_{h_{k_j}}(u_{h_{k_j}},B_r)\leq \FFinf(u_\infty,B_r). 
\]
By Friederich's theorem there exists $z_j\in C^\infty(\overline{B_1};\R^n)$ such that $z_j\to u_\infty$ 
in $W^{1,2}(B_1;\R^n)$. Hence, given $h_k\uparrow\infty$ we can extract $h_{k_j}$ such that
\[
 \lim_{j\uparrow\infty}\lambda_{h_{k_j}}^{p-2}\int_{B_1}\big(|\nabla z_j|^p+|z_j|^p\big)dx=0,
\]
and the measures 
{$\nu_j:=\lambda_{h_{k_j}}^{-2}|V_\mu\big(\lambda_{h_{k_j}} e(u_{h_{k_j}})\big)|^2
\calL^n\res B_1$} converge 
weakly$^\ast$ in $B_1$ to some finite measure $\nu$.

Let now $\rho\in(0,r)$ be fixed, let $\varphi\in \mathrm{Lip}\cap C_c(B_r;[0,1])$ be such 
that $\varphi|_{B_\rho}=1$ and $\|\nabla\varphi\|_{L^\infty(B_1;\R^n)}\leq 2(r-\rho)^{-1}$ and set
\[
 w_j:=\varphi z_j+(1-\varphi)u_{h_{k_j}}.
\]
Then, $w_j\in u_{h_{k_j}}+W^{1,2}_0(B_1;\R^n)$ with 
$w_j\to u{_\infty}$ in $L^2(B_1;\R^n)$. Therefore, by local minimality
of $u_{h_{k_j}}$ we get
\[
 \mathscr{F}_{h_{k_j}}(u_{h_{k_j}},B_r)\leq
 \mathscr{F}_{h_{k_j}}(w_j,B_r)=\int_{B_\rho}{F}_{h_{k_j}}(e(z_j))dx
 +\int_{B_r\setminus B_\rho}{F}_{h_{k_j}}(e(w_j))dx.
\]
Clearly, by {generalized} Lebesgue dominated convergence theorem
\[
 \limsup_{j\uparrow\infty}\int_{B_\rho}{F}_{h_{k_j}}(e(z_j))dx\leq \int_{B_\rho}{F}_\infty(e(u_{\infty}))dx,
\]
and by items (ii) and (iii) in Lemma~\ref{l:Vmu} 
\begin{multline*}
\int_{B_r\setminus B_\rho}{F}_{h_{k_j}}(e(w_j))dx
\leq\frac c{\lambda_{h_{k_j}}^2}\int_{B_r\setminus B_\rho}|V_\mu(\lambda_{h_{k_j}}e(w_j))|^2dx\\
\leq \frac c{\lambda_{h_{k_j}}^2}\int_{B_r\setminus B_\rho}\big(|V_\mu(\lambda_{h_{k_j}}e(u_{h_{k_j}}))|^2
+|V_\mu(\lambda_{h_{k_j}}e(z_j))|^2+|V_\mu(\lambda_{h_{k_j}}\nabla\varphi\odot(u_{h_{k_j}}-z_j))|^2\big)dx\\
\leq c\,\nu_j(B_r\setminus \overline{B_\rho})
+c\int_{B_r\setminus B_\rho}\big(|e(z_j)|^2+\lambda_{h_{k_j}}^{p-2}|e(z_j)|^p\big)dx\\
+\frac c{(r-\rho)^p}\int_{B_r\setminus B_\rho}\big(|u_{h_{k_j}}-z_j|^2+\lambda_{h_{k_j}}^{p-2}|{u_{h_{k_j}}-}z_j|^p\big)dx. 
\end{multline*}
 Summarizing, if $r\in(0,1)$ and $\rho\in(0,r)$ are chosen such that 
 $\nu(\partial B_r)=\nu(\partial B_\rho)=0$,
 {recalling that $u_h\to u$, $z_j\to u$ in $L^2(B_1;\R^n)$, and that $\lambda_h^{1-\sfrac2p}u_h\to0$, $\lambda_{h_{k_j}}^{1-\sfrac2p}w_j\to0$ in $L^p(B_1;\R^n)$,}
 we have
\[
\limsup_{j\uparrow\infty}\int_{B_r\setminus B_\rho}{F}_{h_{k_j}}(e({w_j}))dx\leq c\,\nu(B_r\setminus\overline{B_\rho})
+c\int_{B_r\setminus B_\rho}|e({u_\infty})|^2dx. 
\]
Thus, if $\rho_l\uparrow r$, we conclude at once by an easy diagonalization argument.
\medskip

%
%
%
%

\noindent\emph{The sub-quadratic case {$p\le 2$}.}
 We first prove that $u_\infty\in W^{1,2}(B_1;\R^n)$. 
 Set $E_h:=\{\lambda_h^{\sfrac 12}|e(u_h)|\geq 1\}$, then $\calL^n(E_h)\downarrow 0$ as $h\uparrow\infty$ and 
 \[
 (\mu+1)^{\sfrac p2-1} \int_{E_h^c}|e(u_h)|^2dx\leq \|\lambda_h^{-1}V_\mu(\lambda_he(u_h))\|_{L^2(B_1;\R^n)}^2.
 \]
Therefore, up to a subsequence not relabeled, $(e(u_h)\chi_{E_h^c})_h$ converges weakly in 
$L^2(B_1;\R^{n\times n})$ to some function $\vartheta$. Moreover, as for all 
$\varphi\in L^{\frac p{p-1}}(B_1;\R^{n\times n})$, $\varphi\chi_{E_h^c}\to \varphi$ in 
$L^{\frac p{p-1}}(B_1;\R^{n\times n})$, from the weak convergence of $(e(u_h))_h$ to $e(u_\infty)$ 
in $L^p(B_1;\R^{n\times n})$ we conclude
\[
\int_{B_1}\langle \vartheta,\varphi\rangle dx=\lim_{h\uparrow\infty}\int_{B_1}\langle e(u_h)\chi_{E_h^c},\varphi\rangle dx
 =\lim_{h\uparrow\infty}\int_{B_1}\langle e(u_h),\varphi\chi_{E_h^c}\rangle dx=\int_{B_1}\langle e(u_\infty),\varphi\rangle dx,
\]
in turn implying $\vartheta=e(u_\infty)$ $\calL^n$ a.e. in $B_1$. 
Thus, by \eqref{e:Fhuh4}, Korn's inequality yields that $u_\infty\in W^{1,2}(B_1;\R^n)$. 

The lower bound inequality in \eqref{e:Finftysuplb} for $r\in(0,1]$ follows 
by arguing exactly as to derive it in case $p\geq 2$.
 
If $p\in(1,2)$ the proof of \eqref{e:Finftysupub} is similar to the super-quadratic case, though some 
additional difficulties arise. With fixed $r\in(0,1)$, by Urysohn's property it is sufficient to show 
that for every subsequence $h_k\uparrow\infty$ we can extract $h_{k_j}\uparrow\infty$ such that
\[
\limsup_{j\uparrow\infty}\mathscr{F}_{h_{k_j}}(u_{h_{k_j}},B_r)\leq \FFinf(u_{\infty},B_r). 
\]
Given a sequence $h_k\uparrow\infty$ we can find a subsequence $h_{k_j}$ and some finite measure $\nu$,
such that the measures 
$\nu_j:=\lambda_{h_{k_j}}^{-2}|V_\mu(\lambda_{h_{k_j}}e(u_{h_{k_j}}))|^2\calL^n\res B_1$ 
converge weakly$^\ast$ on $B_1$ to $\nu$. 

Let now $\rho\in(0,r)$ and $\varphi\in \mathrm{Lip}\cap C_c(B_r;[0,1])$ be such that 
$\varphi|_{B_\rho}=1$ and $\|\nabla\varphi\|_{L^\infty(B_r;\R^n)}\leq 2({r}-\rho)^{-1}$ and set
\[
 w_j:=\varphi u_{\infty}+(1-\varphi)u_{h_{k_j}}.
\]
Then, $w_j\in u_{h_{k_j}}+W^{1,2}_0(B_1;\R^n)$ with {$w_j\to u_\infty$} in $L^p(B_1;\R^n)$. 
Moreover,
\[
\mathscr{F}_{h_{k_j}}(u_{h_{k_j}},B_r)\leq
\mathscr{F}_{h_{k_j}}(w_j,B_r)=\int_{B_\rho}{F}_{h_{k_j}}(e(u_{\infty}))dx+\int_{B_r\setminus B_\rho}{F}_{h_{k_j}}(e(w_j))dx.
\]
Clearly, by Lebesgue dominated convergence theorem
\[
 \limsup_{j\uparrow\infty}\int_{B_\rho}{F}_{h_{k_j}}(e(u_{\infty}))dx\leq \int_{B_\rho}{F}_\infty(e(u_{\infty}))dx,
\]
and by item (iii) both in Lemma~\ref{l:Vmu} and in Lemma~\ref{l:Fhsup} 
\begin{multline*}
\int_{B_r\setminus B_\rho}{F}_{h_{k_j}}(e(w_j))dx
\leq\frac c{\lambda_{h_{k_j}}^2}\int_{B_r\setminus B_\rho}|V_\mu(\lambda_{h_{k_j}}e(w_j))|^2dx\\
\leq\frac c{\lambda_{h_{k_j}}^2}\int_{B_r\setminus B_\rho}|V_\mu(\lambda_{h_{k_j}}e(u_{h_{k_j}}))|^2dx
+\frac c{\lambda_{h_{k_j}}^2}\int_{B_r\setminus B_\rho}|V_\mu(\lambda_{h_{k_j}}e(u_{\infty}))|^2dx\\
+\frac c{({r}-\rho)^{2}\lambda_{h_{k_j}}^2}
\underbrace{\int_{B_r\setminus B_\rho}|V_\mu(\lambda_{h_{k_j}}(u_{h_{k_j}}-u_{\infty}))|^2dx}_{I_j:=}\\
\leq c\,\nu_j(B_r\setminus\overline{B_\rho})+c\,\int_{B_r\setminus B_\rho}|e(u_{\infty})|^2dx
+\frac c{({r}-\rho)^{2}}\lambda_{h_{k_j}}^{-2}I_j.
\end{multline*}
{In order to estimate the last term we use a Lipschitz truncation in order to use Rellich's theorem separately on the part with quadratic growth and on the one with $p$-growth. Precisely, let
$E_j:=\{\lambda_{h_{k_j}}|\nabla(u_{h_{k_j}}-u_{\infty})|> 1\}$. Then there is a set $F_j$ with $E_j\subset F_j\subset B_1$
such that $\lambda_{h_{k_j}}(u_{h_{k_j}}-u_{\infty})$ is $c$-Lipschitz in $B_1\setminus F_j$ and \cite[Theorem 3, Section 6.6.3]{EvansGariepy}
\begin{equation*}
 |F_j|\le c \lambda_{h_{k_j}}^p \int_{E_j} |\nabla(u_{h_{k_j}}-u_{\infty})|^pdx \le  c  \int_{E_j} |V_\mu(\lambda_{h_{k_j}}\nabla(u_{h_{k_j}}-u_{\infty}))|^2dx \le c \lambda_{h_{k_j}}^2.
\end{equation*}
Let $w_j$ be a $c\lambda_{h_{k_j}}^{-1}$-Lipschitz extension of $u_{h_{k_j}}-u_{\infty}|_{B_1\setminus F_j}$. 
We estimate
\begin{equation*}
 \int_{B_1} |\nabla w_j|^2dx \le c \lambda_{h_{k_j}}^{-2} |F_j| + \int_{B_1\setminus F_j} \lambda_{h_{k_j}}^{-2}  |V_\mu(\lambda_{h_{k_j}}\nabla(u_{h_{k_j}}-u_{\infty}))|^2dx \le c.
\end{equation*}
Therefore $(w_j)_j$ is bounded in $W^{1,2}(B_1;\R^n)$, and, since it converges (up to a subsequence) pointwise almost everywhere to zero, it converges also strongly in $L^2(B_1;\R^n)$ to zero.
Consider now the difference $d_j=u_{h_{k_j}}-u_{\infty}-w_j$. We estimate
\begin{multline*}
 \int_{B_1} |\nabla d_j|^pdx \le c \int_{E_j}  |\nabla(u_{h_{k_j}}-u_{\infty})|^pdx + c |F_j| \lambda_{h_{k_j}}^{-p}\\
 \le 
 c  \int_{E_j} \lambda_{h_{k_j}}^{-p} |V_\mu(\lambda_{h_{k_j}}\nabla(u_{h_{k_j}}-u_{\infty}))|^2dx 
 +
 c \lambda_{h_{k_j}}^{2-p}\le 
 c \lambda_{h_{k_j}}^{2-p}.
\end{multline*}
Therefore $(\lambda_{h_{k_j}}^{1-2/p}d_j)_j$ is bounded in $W^{1,p}(B_1;\R^n)$ and converges in measure to zero, hence it converges also strongly in $L^p(B_1;\R^n)$ to zero.
We finally estimate, recalling that for $p\le 2$ we have $|V_\mu(\xi)|^2\le c (|\xi|^2\wedge |\xi|^p)$, 
\begin{align*}
\lambda_{h_{k_j}}^{-2} \int_{B_1}&|V_\mu(\lambda_{h_{k_j}}(u_{h_{k_j}}-u_{\infty}))|^2dx
 \le c\lambda_{h_{k_j}}^{-2}\int_{B_1}|V_\mu(\lambda_{h_{k_j}}w_{h_{k_j}})|^2dx\\&+c\lambda_{h_{k_j}}^{-2}\int_{F_j}|V_\mu(\lambda_{h_{k_j}}d_{h_{k_j}})|^2dx
 \le c\int_{B_1}|w_{h_{k_j}}|^2dx+ c \int_{B_1} \lambda_{h_{k_j}}^{p-2}|d_{h_{k_j}}|^pdx
\end{align*}
and see that each term in the right-hand side converges to zero.
}

Therefore, we deduce that 
\[
 \limsup_{j\uparrow\infty}\lambda_{h_{k_j}}^{-2}I_j=0.
\]
Thus, in conclusion provided $r\in(0,1)$ and $\rho\in(0,r)$ are such that 
$\nu(\partial B_r)=\nu(\partial B_\rho)=0$ we have
\[
\limsup_{j\uparrow\infty}\int_{B_r\setminus B_\rho}{F}_{h_{k_j}}(e(w_j))dx\leq 
c\,\nu(B_r\setminus \overline{B_\rho})+c\int_{B_r\setminus B_\rho}|e(u_{\infty})|^2dx. 
\]
 Thus, if $\rho_l\uparrow r$, we conclude by an easy diagonalization argument.
\end{proof}

We next deduce that $u_\infty$ is actually the solution of a linear elliptic system.
\begin{corollary}\label{c:pde}
 The limit function $u_\infty\in W^{1,2}(B_1;\R^n)$ satisfies  \begin{equation}\label{e:pdeinfty}
\int_{B_1}
\langle \nabla^2 f_\mu(\mathbb{A}_\infty) e(u_\infty),e(\varphi)\rangle dx=0
 \end{equation}
 for all $\varphi\in C_c^\infty(B_1;\R^n)$.
\end{corollary}
\begin{proof}
 Being $u_h$ a local minimizer of $\FFh$, for all $\varphi\in C_c^\infty(B_1;\R^n)$ it holds
\[
\lambda_h^{-1}\int_{B_1}\langle\nabla f_\mu(\mathbb{A}_h+\lambda_h e(u_h))
-\nabla f_\mu(\mathbb{A}_h),e(\varphi)\rangle dx=0.
\]
Consider the sets 
\[
 E_h^+:=\{x\in B_1:\,|\lambda_h e(u_h)|\geq \sqrt{\mu}\},\quad 
 E_h^-:=\{x\in B_1:\,|\lambda_h e(u_h)|<\sqrt{\mu}\}.
\]
By the weak convergence of $(u_h)_h$ to $u_\infty$ in $W^{1,p\wedge 2}(B_1;\R^n)$, 
we get
\begin{equation}\label{e:measEhp}
 \calL^n(E_h^+)\leq {\mu^{-\sfrac p2\wedge 1}}\int_{B_1}|\lambda_h e(u_h)|^{p\wedge 2}dx\leq c\lambda_h^{p\wedge 2},
\end{equation}
so that $\calL^n(E_h^+)=o(\lambda_h)$ as $h\uparrow\infty$. Hence, 
we deduce that
\begin{align*}
\limsup_{h\uparrow\infty} &\left|\lambda_h^{-1}\int_{E_h^+}\langle\nabla f_\mu(\mathbb{A}_h+\lambda_h e(u_h))
-\nabla f_\mu(\mathbb{A}_h),e(\varphi)\rangle dx\right|\\
&\stackrel{\eqref{e:growth condition nabla}}{\leq}
\limsup_{h\uparrow\infty}\Big(c\,\frac{\calL^n(E_h^+)}{\lambda_h}+c\lambda_h^{p-2}\int_{E_h^+}|e(u_h)|^{p-1}dx\Big)\\
&\leq c\limsup_{h\uparrow\infty}\lambda_h^{p-2}(\calL^n(E_h^+))^{\sfrac 1p}\Big(\int_{E_h^+}|e(u_h)|^pdx\Big)^{\sfrac{(p-1)}p}\\
&\leq c\limsup_{h\uparrow\infty}\lambda_h^{1-\frac2p}(\calL^n(E_h^+))^{\sfrac 1p}
\Big(\int_{E_h^+}\lambda_h^{-2}|V_\mu(\lambda_he(u_h))|^2dx\Big)^{\sfrac{(p-1)}p}\\
&\stackrel{\eqref{e:Fhuh},\,\eqref{e:measEhp}}{\leq}  c\limsup_{h\uparrow\infty}\lambda_h^{(2-\frac2p)\wedge 1}=0,
\end{align*}
by taking into account item (iv) in Lemma~\ref{l:Vmu} to infer 
the last but one inequality. Finally, note that 
\begin{multline*}
\lambda_h^{-1}\int_{E_h^-}\langle\nabla f_\mu(\mathbb{A}_h+\lambda_h e(u_h))
-\nabla f_\mu(\mathbb{A}_h),e(\varphi)\rangle dx\\
=\int_{E_h^-}\langle\Big(\int_0^1\nabla^2 f_\mu(\mathbb{A}_h+t\lambda_h e(u_h))dt\Big)
e(u_h),e(\varphi)\rangle dx,
\end{multline*}
then as $(u_h)_h$ converges weakly to $u_\infty$ in $W^{1,p\wedge 2}(B_1;\R^n)$, 
$\lambda_he(u_h)\to 0$ $\calL^n$ a.e. on $B_1$, and as $f_\mu\in C^2(\R^{n\times n}_\sym)$ 
if $\mu>0$, by the dominated convergence theorem we get
\[
\lim_{h\uparrow\infty}\lambda_h^{-1}\int_{E_h^-}\langle\nabla f_\mu(\mathbb{A}_h+\lambda_h e(u_h))
-\nabla f_\mu(\mathbb{A}_h),e(\varphi)\rangle dx
=\int_{B_1}\langle \nabla^2 f_\mu(\mathbb{A}_\infty) e(u_\infty),e(\varphi)\rangle dx.
\]
\end{proof}

In turn, this last result provides the claimed local strong convergence.
\begin{corollary}\label{c:strongconv}
 Let $(u_h)_h$ be the sequence in \eqref{e:Fhuh} converging weakly in $W^{1,p\wedge 2}(B_1;\R^n)$ 
 to the function $u_\infty\in W^{1,2}(B_1;\R^n)$. Then, for all $r\in(0,1)$
 \[
  \lim_{h\uparrow\infty}\int_{B_r}\lambda_h^{-2}|V_\mu(\lambda_h e(u_h-u_\infty)|^2dx=0.
 \]
In particular, $(u_h)_h$ converges to $u_\infty$ in $W^{1,p\wedge 2}_\loc(B_1;\R^n)$.
\end{corollary}
\begin{proof}
It is sufficient to show the conclusion for all those $r\in(0,1)$ for which both inequalities 
\eqref{e:Finftysuplb} and \eqref{e:Finftysupub} in Theorem~\ref{t:Gammasup} hold true. 
In such a case, we have  
 \[
  \lim_{h\uparrow\infty}\FFh(u_h,B_r)=\FFinf(u_\infty,B_r).
 \]
 {We observe that 
  $u_\infty\in C^\infty(B_1;\R^n)$ by 
Corollary~\ref{c:pde} and the regularity theory for linear elliptic systems. Therefore for $h$ sufficiently large we have 
$\lambda_h|e(u_\infty)|<\mu$ uniformly on $B_r$.}
By item (iv) in Lemma~\ref{l:Fhsup} we get
\begin{align*}
\FFh(u_h,B_r)-\FFh(u_\infty,B_r)
 \geq & c\int_{B_r}\lambda_h^{-2}
 |V_\mu(\lambda_h e(u_h-u_\infty)|^2dx\\
 &+\frac{1}{\lambda_h}\int_{B_r}\langle\nabla f_\mu(\mathbb{A}_h+\lambda_h e(u_\infty))
-\nabla f_\mu(\mathbb{A}_h),e(u_h-u_\infty)\rangle dx \\
= &c\int_{B_r}\lambda_h^{-2}|V_\mu(\lambda_h e(u_h-u_\infty)|^2dx\\
&+\int_{B_r}\langle\Big(\int_0^1\nabla^2 f_\mu(\mathbb{A}_h+t\lambda_h e(u_\infty))dt\Big)e(u_\infty),
e(u_h-u_\infty)\rangle dx.
\end{align*}
Since $\FFh(u_\infty,B_r)\to\FFinf(u_\infty;B_r)$ as $h\uparrow\infty$, and 
\[
 \int_0^1\nabla^2 f_\mu(\mathbb{A}_h+t\lambda_h e(u_\infty)) e(u_\infty)dt
 \to\int_0^1\nabla^2 f_\mu(\mathbb{A}_\infty) e(u_\infty)dt
\]
in $L^\infty_\loc(B_1;\R^{n\times n})$, we conclude by the weak convergence of $(u_h)_h$ to $u_\infty$
in $W^{1,p\wedge 2}(B_1;\R^{n\times n})$. 
\end{proof}

\section*{Acknowledgments} 
This work was partially supported 
by the Deutsche Forschungsgemeinschaft through the Sonderforschungsbereich 1060 
{\sl ``The mathematics of emergent effects''}.
S.~Conti thanks the University of Florence for
the warm hospitality of the DiMaI ``Ulisse Dini'', where part of this work was 
carried out. F.~Iurlano has been partially supported by the project EMERGENCE of the Sorbonne Université 
and by a post-doctoral fellowship of the Fondation Sciences Mathématiques de Paris.
M.~Focardi and F.~Iurlano are members of the Gruppo Nazionale per
l'Analisi Matematica, la Probabilit\`a e le loro Applicazioni (GNAMPA)
of the Istituto Nazionale di Alta Matematica (INdAM).


\bibliographystyle{siam}
\bibliography{biblio}

\begin{thebibliography}{10}

\bibitem{AcerbiFusco89}
{\sc E.~Acerbi and N.~Fusco}, {\em Regularity for minimizers of nonquadratic
  functionals: the case {$1<p<2$}}, J. Math. Anal. Appl., 140 (1989),
  pp.~115--135.

\bibitem{AmbrosioFuscoPallara}
{\sc L.~Ambrosio, N.~Fusco, and D.~Pallara}, {\em Functions of bounded
  variation and free discontinuity problems}, Oxford Mathematical Monographs,
  The Clarendon Press Oxford University Press, New York, 2000.

\bibitem{BeiraoCrispo}
{\sc H.~Beir\~ao~da Veiga and F.~Crispo}, {\em On the global regularity for
  nonlinear systems of the {$p$}-{L}aplacian type}, Discrete Contin. Dyn. Syst.
  Ser. S, 6 (2013), pp.~1173--1191.

\bibitem{BildhauerFuchs}
{\sc M.~Bildhauer and M.~Fuchs}, {\em Variants of the {S}tokes problem: the
  case of anisotropic potentials}, J. Math. Fluid Mech., 5 (2003),
  pp.~364--402.

\bibitem{BildhauerFuchsZhong05}
{\sc M.~Bildhauer, M.~Fuchs, and X.~Zhong}, {\em A lemma on the higher
  integrability of functions with applications to the regularity theory of
  two-dimensional generalized {N}ewtonian fluids}, Manuscripta Math., 116
  (2005), pp.~135--156.

\bibitem{BourdinFrancfortMarigo2008}
{\sc B.~Bourdin, G.~A. Francfort, and J.-J. Marigo}, {\em The variational
  approach to fracture}, J. Elasticity, 91 (2008), pp.~5--148.

\bibitem{BreitDieningFuchs}
{\sc D.~Breit, L.~Diening, and M.~Fuchs}, {\em Solenoidal {L}ipschitz
  truncation and applications in fluid mechanics}, J. Differential Equations,
  253 (2012), pp.~1910--1942.

\bibitem{CarozzaFuscoMingione98}
{\sc M.~Carozza, N.~Fusco, and G.~Mingione}, {\em Partial regularity of
  minimizers of quasiconvex integrals with subquadratic growth}, Ann. Mat. Pura
  Appl. (4), 175 (1998), pp.~141--164.

\bibitem{ChambolleContiIurlano2017}
{\sc A.~Chambolle, S.~Conti, and F.~Iurlano}, {\em Approximation of functions
  with small jump sets and existence of strong minimizers of {Griffith's}
  energy}, preprint arXiv:1710.01929,  (2017).

\bibitem{ContiFocardiIurlano2016-CRAS}
{\sc S.~Conti, M.~Focardi, and F.~Iurlano}, {\em Existence of minimizers for
  the 2d stationary {Griffith} fracture model}, C. R. Acad. Sci. Paris, Ser. I,
  354 (2016), pp.~1055--1059.

\bibitem{ContiFocardiIurlano17}
\leavevmode\vrule height 2pt depth -1.6pt width 23pt, {\em Existence of strong
  minimizers for the {G}riffith static fracture model in dimension two},
  preprint arXiv:1611.03374,  (2017).

\bibitem{ContiFocardiIurlanoGBD}
{\sc S.~Conti, M.~Focardi, and F.~Iurlano}, {\em Approximation of fracture
  energies with $p$-growth via piecewise affine finite elements}, to appear on
  ESAIM COCV, preprint arXiv:1706.01735,  (2018).

\bibitem{dm-toa}
{\sc G.~Dal~Maso and R.~Toader}, {\em A model for the quasi-static growth of
  brittle fractures: existence and approximation results}, Arch. Ration. Mech.
  Anal., 162 (2002), pp.~101--135.

\bibitem{DegiorgiCarrieroLeaci1989}
{\sc E.~De~Giorgi, M.~Carriero, and A.~Leaci}, {\em Existence theorem for a
  minimum problem with free discontinuity set}, Arch. Rational Mech. Anal., 108
  (1989), pp.~195--218.

\bibitem{DieningEttwein}
{\sc L.~Diening and F.~Ettwein}, {\em Fractional estimates for
  non-differentiable elliptic systems with general growth}, Forum Math., 20
  (2008), pp.~523--556.

\bibitem{DieningKaplicky}
{\sc L.~Diening and P.~Kaplick{\'y}}, {\em {$L^q$} theory for a generalized
  {S}tokes system}, Manuscripta Math., 141 (2013), pp.~333--361.

\bibitem{DieningKaplickySchwarzacher14}
{\sc L.~Diening, P.~Kaplick{\'y}, and S.~Schwarzacher}, {\em Campanato
  estimates for the generalized {S}tokes system}, Ann. Mat. Pura Appl. (4), 193
  (2014), pp.~1779--1794.

\bibitem{DieningRuzickaSchumacher10}
{\sc L.~Diening, M.~Ruzicka, and K.~Schumacher}, {\em A decomposition technique
  for {J}ohn domains}, Ann. Acad. Sci. Fenn. Math., 35 (2010), pp.~87--114.

\bibitem{DuzGrotKr}
{\sc F.~Duzaar, J.~F. Grotowski, and M.~Kronz}, {\em Regularity of almost
  minimizers of quasi-convex variational integrals with subquadratic growth},
  Ann. Mat. Pura Appl. (4), 184 (2005), pp.~421--448.

\bibitem{EvansGariepy}
{\sc L.~C. Evans and R.~F. Gariepy}, {\em Measure theory and fine properties of
  functions}, Boca Raton CRC Press, 1992.

\bibitem{FrancfortMarigo1998}
{\sc G.~A. Francfort and J.-J. Marigo}, {\em Revisiting brittle fracture as an
  energy minimization problem}, J. Mech. Phys. Solids, 46 (1998),
  pp.~1319--1342.

\bibitem{FuchsSeregin}
{\sc M.~Fuchs and G.~Seregin}, {\em Variational methods for problems from
  plasticity theory and for generalized {N}ewtonian fluids}, vol.~1749 of
  Lecture Notes in Mathematics, Springer-Verlag, Berlin, 2000.

\bibitem{Giaquinta83}
{\sc M.~Giaquinta}, {\em Multiple integrals in the calculus of variations and
  nonlinear elliptic systems}, vol.~105 of Annals of Mathematics Studies,
  Princeton University Press, Princeton, NJ, 1983.

\bibitem{GiaquintaGiusti83}
{\sc M.~Giaquinta and E.~Giusti}, {\em Differentiability of minima of
  nondifferentiable functionals}, Invent. Math., 72 (1983), pp.~285--298.

\bibitem{GiaquintaMartinazzi2012}
{\sc M.~Giaquinta and L.~Martinazzi}, {\em An introduction to the regularity
  theory for elliptic systems, harmonic maps and minimal graphs}, vol.~11 of
  Appunti. Scuola Normale Superiore di Pisa (Nuova Serie), Edizioni della
  Normale, Pisa, second~ed., 2012.

\bibitem{Giusti}
{\sc E.~Giusti}, {\em Direct methods in the calculus of variations}, World
  Scientific Publishing Co., Inc., River Edge, NJ, 2003.

\bibitem{hutchinson1989course}
{\sc J.~W. Hutchinson}, {\em A course on nonlinear fracture mechanics},
  Department of Solid Mechanics, Techn. University of Denmark, 1989.

\bibitem{KristensenMingione05}
{\sc J.~Kristensen and G.~Mingione}, {\em Non-differentiable functionals and
  singular sets of minima}, C. R. Math. Acad. Sci. Paris, 340 (2005),
  pp.~93--98.

\bibitem{KristensenMingione}
\leavevmode\vrule height 2pt depth -1.6pt width 23pt, {\em The singular set of
  minima of integral functionals}, Arch. Ration. Mech. Anal., 180 (2006),
  pp.~331--398.

\bibitem{KristensenMingione07}
\leavevmode\vrule height 2pt depth -1.6pt width 23pt, {\em The singular set of
  {L}ipschitzian minima of multiple integrals}, Arch. Ration. Mech. Anal., 184
  (2007), pp.~341--369.

\bibitem{Mingione08}
{\sc G.~Mingione}, {\em Singularities of minima: a walk on the wild side of the
  calculus of variations}, J. Global Optim., 40 (2008), pp.~209--223.

\bibitem{Morrey-Multiple}
{\sc C.~B. Morrey, Jr.}, {\em Multiple integrals in the calculus of
  variations}, Classics in Mathematics, Springer-Verlag, Berlin, 2008.
\newblock Reprint of the 1966 edition.

\end{thebibliography}

\end{document}